\documentclass[12pt]{article} 
\usepackage{authblk}
\usepackage{graphicx}
\usepackage{latexsym,amsfonts,amssymb}
\usepackage{amsmath, amsthm}
\usepackage{graphicx}
\usepackage[dvips]{color}
\usepackage{colordvi,multicol}
\usepackage[marginal]{footmisc}
\usepackage{mathrsfs}

\usepackage{cite}

\usepackage{enumitem}

\setlist[itemize,2]{      
	leftmargin=0.3em    
}

\setlist[itemize,3]{      
	leftmargin=0.6em              
}

\setlist[itemize,4]{      
	leftmargin=0.9em             
}
\setlist[enumerate]{
	leftmargin=1.1em         
}

\usepackage{hyperref}
\usepackage{marvosym}
\newtheorem{thm}{Theorem}[section]

\newtheorem*{thm*}{Theorem}
\newtheorem{cor}[thm]{Corollary}
\newtheorem{lem}[thm]{Lemma}
\newtheorem{pro}[thm]{Proposition}

\newtheorem{rem}[thm]{Remark}

\numberwithin{equation}{section}

\allowdisplaybreaks[4] 

\date{}

\begin{document}
	
	\title{\bf Quenched correlation decay for random splittings of some prototypical 3D flows including the ABC flow}
	\author[1,2]{Nianci Jiang\thanks{jiangnianci22@mails.ucas.ac.cn} }
	\author[3,4]{Weili Zhang\thanks{zhangweili@amss.ac.cn}}
	\affil[1]{School of Mathematical Sciences, University of Chinese Academy of Sciences, Beijing 100049, China}
	\affil[2]{Academy of Mathematics and Systems Science, Chinese Academy of Sciences, Beijing 100190, China}
	\affil[3]{Beijing Institute of Mathematical Sciences and Applications, Beijing 101408, China}
	\affil[4]{Yau Mathematical Sciences Center, Tsinghua University, Beijing 100084, China}
	\renewcommand\Authfont{\normalsize}  
	\renewcommand\Affilfont{\small}             
	\maketitle
	\noindent{\bf Abstract:}\quad For the long-time dynamical challenges of some prototypical 3D flows including the ABC flow on $\mathbb{T}^3$, we apply a random splitting method to establish two fundamental indicators of chaotic dynamics. First, under general assumptions, we establish that these random splittings exhibit Lagrangian chaos, characterized by a  positive top Lyapunov exponent. Furthermore, we demonstrate the almost-sure quenched correlation decay of these random splittings, which is a stronger property than the almost-sure positivity of Lyapunov exponents alone. This framework is then applied to construct ideal dynamo in kinematic dynamo theory and to establish exponential mixing of passive scalars. 
	
	
	
	
	\tableofcontents
	\section{Introduction}
	\quad The long-term behavior of deterministic equations, such as turbulence and chaos in fluid systems, eludes direct analytical deduction from the governing equations themselves, a hallmark challenge in nonlinear dynamics. To bridge this gap, we implement the random splitting method \cite{AMM1}, which strategically incorporates stochasticity to enhance analytical tractability. Consider the ordinary differential equation on $\mathbb{T}^3=\mathbb{R}^3/(2\pi\mathbb{Z})^3$:
	\begin{equation}\label{ODE}
		\dot{x} = X(x) = \sum_{i=1}^3 X_i(x), \quad x = (x_1, x_2, x_3) \in \mathbb{T}^3,
	\end{equation}
	where $X$ and $\{X_i\}_{i=1}^3$ are smooth vector fields on $\mathbb{T}^3$. Let $\Phi_t$ denote the flow of $\dot{x} = X(x)$, and let $\varphi_t^{(i)}$ denote the flow generated by $\dot{x} = X_i(x)$ for $i = 1, 2, 3$. Fix $h > 0$, and let $\{\tau_i\}_{i=1}^\infty$ be a sequence of independent, identically distributed exponential random variables with mean $h$. The \textbf{random splitting} $\{\Phi_{\underline{\tau}}^m\}_{m \geq 0}$ on $\mathbb{T}^3$ is defined iteratively as:
	\[
	\Phi_{\underline{\tau}}^m = 
	\begin{cases}
		I, & m = 0; \\
		\varphi_{\tau_{3m}}^{(3)} \circ \varphi_{\tau_{3m-1}}^{(2)} \circ \varphi_{\tau_{3m-2}}^{(1)}\left(\Phi_{\underline{\tau}}^{m-1}\right), & m \geq 1,
	\end{cases}
	\]
	where $I$ is the identity map and $\underline{\tau} = (\tau_1, \tau_2, \dots)$. We prove that under general conditions, the random splitting trajectories ${\Phi_{\underline{\tau}}^m}$ converge almost surely to the deterministic flow $\Phi_t$ as $h \to 0$ (see Theorem \ref{convergence th} for the rigorous statement). Consequently, this framework not only provides a rigorous approximation scheme but also offers a novel mechanism for incorporating stochastic agitation. Our primary focus lies in analyzing the intrinsic properties of random splitting methods rather than their role as mere approximations of the system \eqref{ODE}. Central to this study is the characterization of their long-term dynamical behavior. Specifically, we establish almost sure positivity of the top Lyapunov exponent (see Theorem \ref{Positivity of the top Lyapunov exponent} for the rigorous statement) and quenched correlation decay (see Theorem \ref{Exponential mixing} for the rigorous statement) for random splitting of broad fluid models.
	
	One particularly important application of this framework lies in dynamo theory. The generation of magnetic fields in celestial bodies or conductive fluids is the focus of dynamo theory. Kinematic dynamo theory studies which fluid motions can lead to exponential growth of the magnetic field under low magnetic diffusivity. Let $B$ denote the magnetic field and $u$ the incompressible velocity field. The magnetic induction equation is given by:
	\begin{equation}\label{magnetic induction equation}
		\begin{cases}
			\partial_t B + (u \cdot \nabla)B - (B \cdot \nabla)u = \kappa \Delta B, & \text{in } (0, \infty) \times \mathbb{T}^3, \\
			\nabla \cdot B = 0, & \text{in } (0, \infty) \times \mathbb{T}^3, \\
			B(0, \cdot) = B_0, & \text{in } \mathbb{T}^3,
		\end{cases}
	\end{equation}
	where $\kappa$ represents magnetic diffusivity. Given $u(t,x)$ we define kinematic dynamo action to occur when 
	\begin{equation*}
		\gamma(\kappa) := \sup_{B_0 \in L^2(\mathbb{T}^3)} \limsup_{t \to \infty} \frac{1}{t} \log \| B(t;\kappa) \|_{L^2(\mathbb{T}^3)} > 0.
	\end{equation*}
    A velocity field $u$ is defined to be an \textbf{ideal dynamo} if $\gamma(0) > 0$, and a \textbf{fast kinematic dynamo} if $\lim_{\kappa\to 0} \gamma(\kappa)>0$. A central open problem in dynamo theory, originally posed by Ya. B. Zeldovich \cite{Z} and listed as Problem 1994-28 in V. I. Arnold's monograph \cite{AVI}, concerns whether there exists a divergence-free velocity field $u$ on $\mathbb{T}^3$ which is a fast kinematic dynamo. Building on our stochastic framework, we construct a random velocity field $u_{\underline{\tau}}$ (see \eqref{random velocity field}) that almost surely exhibits ideal dynamo action through chaotic advection mechanisms inherent to the random splitting (see Corollary \ref{Ideal dynamo} for the rigorous statement).
    
    Another important application of this framework lies in fluid dynamics. Mixing in fluid dynamics characterizes the efficiency of a divergence-free velocity field $u_t$ in homogenizing a passive scalar $f_t$ through repeated stretching and folding, erasing large-scale tracer patterns to achieve near-uniformity. This process is governed by the advection equation:
    \begin{equation}\label{advection equation}
    	\begin{cases}
    		\partial_t f + u \cdot \nabla f = 0, & \text{in } (0, \infty) \times \mathbb{T}^3, \\
    		f(0, \cdot) = f_0, & \text{in } \mathbb{T}^3,
    	\end{cases}
    \end{equation}
    where $u_t$ satisfies $\nabla \cdot u_t = 0$, and $f_t$ represents the tracer concentration (e.g., dye or temperature) advected by the flow. To measure mixing, we use the negative Sobolev norm $\|\cdot\|_{H^{-s}}$ ($s > 0$) for mean-zero functions, for any mean-zero initial function $f_0 \in H^s$, defined via the dual-space supremum
    \[
    \|f(t) \|_{H^{-s}} = \sup_{\|g\|_{H^s} = 1} \left|\int_{\mathbb{T}^3} g(x) f(t,x) \, dx \right|,
    \]
    where the supremum is taken over mean-zero functions $g\in H^s$. The velocity field $u$ is called \textbf{exponentially mixing} if for all $s>0$, there exists a constant $\alpha>0$ and a constant $D$ (depending on the initial velocity) such that
    \begin{equation*}
    	\|f(t) \|_{H^{-s}} \leq D e^{-\alpha t} \|f_0\|_{H^{s}},
    \end{equation*}
    for all $t >0$. This definition can likewise be extended to the random version. And we prove that the random velocity field $u_{\underline{\tau}}$ (see \eqref{random velocity field} below) is almost-sure exponential mixing (see Corollary \ref{exponential mixing} for the rigorous statement). 
    
    The Arnold-Beltrami-Childress (ABC) flow is a three-dimensional incompressible velocity field which is an exact solution of Euler's equation. Defined on $\mathbb{T}^3$ through parameters $A, B, C \in \mathbb{R}$, these flows are governed by the divergence-free vector field
    \begin{equation}\label{ABC}
    	X(x) = \begin{pmatrix} 
    		A \sin x_3 + C \cos x_2 \\ 
    		B \sin x_1 + A \cos x_3 \\ 
    		C \sin x_2 + B \cos x_1 
    	\end{pmatrix}, \quad x = (x_1, x_2, x_3) \in \mathbb{T}^3,
    \end{equation}
    which simultaneously satisfies the Beltrami property (vorticity parallel to velocity). The parameter values critically determine the flow's dynamical regime: when one of the parameters $A$, $B$ or $C$ vanishes, the flow is integrable, whereas non-vanishing parameters ($ABC \neq 0$) typically lead to complex and unpredictable streamline behaviors. In this paper, we consider the ABC flow with $ABC \neq 0$ as a special case of our more general mechanism with 
    \[
    X_1(x) = A(\sin x_3, \cos x_3, 0), \
    X_2(x) = B(0, \sin x_1, \cos x_1), \
    X_3(x) = C(\cos x_2, 0, \sin x_2).
    \]
    By employing the random splitting method, we establish the almost-sure positivity of the top Lyapunov exponent and prove almost-sure quenched correlation decay. Furthermore, we construct the random velocity field $u_{\underline{\tau}}$, demonstrating that it exhibits almost-sure exponential mixing for the advection equation and almost-sure ideal dynamo action for the magnetic induction equation.
    
    Our proofs are grounded in the following key elements. For any $h > 0$, a crucial step in proving the almost-sure positivity of the top Lyapunov exponent for the random splitting of some prototypical 3D flows including the ABC flow, is to show that the one-point Markov chain is uniformly geometrically ergodic. To go from almost-sure positivity of the top Lyapunov exponent to quenched correlation decay, we demonstrate uniform geometric ergodicity for both the projective and two-point Markov chains. This necessitates overcoming technical challenges. For the projective Markov chain, this involves verifying topological irreducibility by constructing explicit movement strategies to connect any given point to a designated point, requiring thorough classification analysis. In the case of the two-point Markov chain, critical steps include (i) identifying and excluding invariant sets that otherwise support additional stationary measures for the transition kernel $P^{(2)}$ corresponding to two-point Markov chain, which would obstruct uniform geometric ergodicity of $P^{(2)}$, and (ii) establishing topological irreducibility through the construction of orbits using rationally linear independent components, serving as ``transfer stations''.
    
    \textbf{Relation to existing results}. Closest to the present work is the results of Coti Zelati and Navarro-Fern{\'a}ndez \cite{CZNF}, who constructed a velocity field by introducing random amplitudes and random phases into the ABC vector field \eqref{ABC}. They established that this velocity field is almost-sure exponential mixing for a passive scalar, while simultaneously serving as an ideal dynamo for a passive vector. For comparison, we develop a robust dynamics-based framework by applying random splitting of some prototypical 3D flows including the ABC flow \eqref{ABC}. This framework enables the construction of velocity fields that simultaneously achieve exponential mixing for passive scalars and act as ideal dynamos for passive vectors. Furthermore, the random splitting trajectories converge almost surely to the deterministic flow as the mean of the exponentially distributed random time approaches zero. 
    
    The random splitting method, first introduced in \cite{AMM1}, serves as a stochastic counterpart to the classical operator splitting approach. Operator splitting is a widely used numerical technique for approximating solutions to differential equations, as seen in \cite{AR, BDMR, NN}. The decomposition in random splitting is also motivated by progress in establishing ergodicity results for piecewise deterministic Markov processes; see, for instance, \cite{BLBMZ, Bak}. In \cite{AMM1}, it is shown that for a class of fluid models, random splittings converge to their deterministic counterparts as the mean time step $h$ tends to zero, and admit a unique ergodic measure. Building on this, \cite{AMM2} proves that for the conservative Lorenz-96 model and Galerkin approximations of the 2D Euler equations, the top Lyapunov exponent is almost surely positive for sufficiently small $h$.
    
	The literature on Lyapunov exponents and exponential mixing is extensive. The local mechanism underlying mixing involves hyperbolic stretching and contraction in phase space, quantified by the top Lyapunov exponent. A positive top Lyapunov exponent implies exponential divergence of trajectories. However, even for the Chirikov standard map \cite{Ch}, verifying positivity remains an open problem, due to the intricate coexistence of hyperbolic orbits and elliptic islands \cite{DP, PC, GA}. Introducing randomness provides an intrinsic method to disrupt potential coherent structures. For instance, \cite{BXY} establishes a positive lower bound for the top Lyapunov exponent of the Chirikov standard map under small random perturbations. Building on Furstenberg's seminal work \cite{F}, which establishes verifiable criteria for the strict positivity of the top Lyapunov exponent in i.i.d. products of determinant-one matrices, a substantial body of research has further advanced these ideas, including \cite{AV, BL, C, L} and references therein. Under the positive top Lyapunov exponent condition, \cite{DKK} establishes almost-sure exponential mixing for a class of nondegenerate stochastic differential equations. For passive scalars, Pierrehumbert \cite{P} provided early numerical evidence that alternating sine shear flows with randomized phases induce almost-sure exponential mixing. Recent advances have established rigorous proofs of almost-sure exponential mixing for advecting velocity fields through various constructions, including: velocity fields arising as solutions to stochastically forced Navier-Stokes equations \cite{BBP1,BBP}; alternating shear flows on $\mathbb{T}^2$ with randomized phases \cite{BCZG} or random switching times \cite{Co,Zh}; a random version of the ABC flow on $\mathbb{T}^3$ \cite{CZNF}.
	
	We also present results pertaining to the ABC flow. The ABC flows were discovered by Gromeka in 1881 \cite{Gr} and rediscovered by Beltrami in 1889. These flows were analyzed in \cite{DFGHMS}, which introduced the name A-B-C because this example was independently introduced by Arnold (1965) \cite{A1965} and Childress (1970) \cite{CS} as an interesting class of Beltrami flows. When one of the parameters vanishes, the flow is integrable. If all three parameters are nonzero, the flow becomes non-integrable \cite{A1965, H1966, DFGHMS} and exhibits a mixture of chaotic regions and regular islands. The presence of chaos in ABC flows makes them particularly attractive for dynamo modeling. Childress \cite{CS} proposed the 1:1:1 ABC flow as a prototype dynamo, demonstrating large-scale magnetic field growth through an isotropic alpha effect. Numerical evidence supporting fast dynamo action in ABC flows can be found in works such as \cite{AK, BD, A2011, JG}. For more detailed discussions on ABC flows and fast dynamos, we refer to the monographs \cite{CG, AK2021}.
	
	This paper is organized as follows. Section \ref{Results and discussions} introduces the notation and presents the main results. Section \ref{section 3} provides preliminaries on Markov chains and proves the uniform geometric ergodicity for the one-point Markov chain $\{\Phi_{\underline{\tau}}^m\}$ (Theorem \ref{Uniformly geometrically ergodic}). In Section \ref{section 4}, we show the almost-sure positivity for the top Lyapunov exponent for $\{\Phi_{\underline{\tau}}^m\}$ (Theorem \ref{Positivity of the top Lyapunov exponent}). Section \ref{section 5} focuses on the projective Markov chain, establishing its uniform geometric ergodicity and constructing a random velocity field $u_{\underline{\tau}}$, which acts as an ideal dynamo for passive vectors (Corollary \ref{Ideal dynamo}). In Section \ref{section 6}, we establish the $V$-uniform geometric ergodicity for the two-point Markov chain. Section \ref{section 7} establishes the quenched correlation decay for $\{\Phi_{\underline{\tau}}^m\}$ (Theorem \ref{Exponential mixing}), and proves that $u_{\underline{\tau}}$ exhibits exponential mixing for passive scalars (Corollary \ref{exponential mixing}).
	
	\section{Results and discussions}\label{Results and discussions}
	\quad Consider the ordinary differential equation on $\mathbb{T}^3$:
	\[
	\dot{x} = X(x) = \sum_{i=1}^{3} X_i(x),
	\]
	where $x=(x_1,x_2,x_3)$ and $X_i: \mathbb{T}^3 \to \mathbb{R}^3$ are defined as:
	\[
	X_1(x) = (f_1(x_3), f_1'(x_3), 0), \quad
	X_2(x) = (0, f_2(x_1), f_2'(x_1)), \quad
	X_3(x) = (f_3'(x_2), 0, f_3(x_2)),
	\]
	with $f_i \in C^\omega(\mathbb{S}^1, \mathbb{R})$ being non-constant functions and $\mathbb{S}^1 = \mathbb{R}/(2\pi \mathbb{Z})$. We assume the condition:
    \begin{flalign*}
    	&\text{(H1)} \quad C_{f_i} \cap C_{f_i'} = \emptyset \quad \text{and} \quad C_{f_i'} \cap C_{f_i''} = \emptyset, \quad \text{for} \ \ i = 1, 2, 3, &
    \end{flalign*}
	where $C_g := \{x \in \mathbb{S}^1 : g(x) = 0\}$.
	
	Denote $\mathcal{X}:=\{X_1, X_2, X_3\}$. Let $\Phi_t$ denote the flow of $\dot{x} = X(x)$ and $\varphi_t^{(i)}$ the flow of $\dot{x} = X_i(x)$. The flows $\varphi_t^{(i)}$ are explicitly given by:
	\[
	\varphi_{t}^{(1)}(x) =
	\begin{pmatrix}
		x_1 + t f_1(x_3) \\
		x_2 + t f_1'(x_3) \\
		x_3
	\end{pmatrix},
	\quad
	\varphi_{t}^{(2)}(x) =
	\begin{pmatrix}
		x_1 \\
		x_2 + t f_2(x_1) \\
		x_3 + t f_2'(x_1)
	\end{pmatrix},
	\quad
	\varphi_{t}^{(3)}(x) =
	\begin{pmatrix}
		x_1 + t f_3'(x_2) \\
		x_2 \\
		x_3 + t f_3(x_2)
	\end{pmatrix}.
	\]
    
    We aim to construct a random dynamical approximation of the flow $\Phi_t$. Fix $h > 0$, and let $\{\tau_i\}_{i=1}^\infty$ be a sequence of independent, exponentially distributed random variables with mean $h$. The approximating dynamics is defined as a Markov chain $\{\Phi_{\underline{\tau}}^m\}$ on $\mathbb{T}^3$, given by:
    \[
    \Phi_{\underline{\tau}}^m =
    \begin{cases}
    	I, & m = 0; \\
    	\varphi_{\tau_{3m}}^{(3)} \circ \varphi_{\tau_{3m-1}}^{(2)} \circ \varphi_{\tau_{3m-2}}^{(1)}(\Phi_{\underline{\tau}}^{m-1}), & m = 1, 2, \dots,
    \end{cases}
    \]
    where $I$ denotes the identity map, and $\underline{\tau} = (\tau_1, \tau_2, \dots) \in \mathbb{R}_{\geq 0}^\mathbb{N} := [0, \infty)^\mathbb{N}$. Let $(\mathbb{R}_{\geq 0}^\mathbb{N}, \theta)$ represent the space of non-negative real sequences equipped with the \textbf{left shift operator} $\theta$, defined for $n \in \mathbb{N}$ as:
    \[
    \theta_n(\underline{\tau}) = (\tau_{n+1}, \tau_{n+2}, \dots),
    \]
    where $\underline{\tau} = (\tau_1, \tau_2, \dots)$. This approximating dynamics, known as a \textbf{random splitting} associated with $\mathcal{X}$, was introduced in \cite{AMM1}.  
	
	Let $C^i(\mathbb{T}^3)$ denote the space of $i$-times continuously differentiable functions $g: \mathbb{T}^3 \to \mathbb{R}$, equipped with the norm
	$$\|g\|_i := \sup_{x \in \mathbb{T}^3} \max_{0 \leq k \leq i} \|D^k g(x)\|, $$
	where $D^k g(x)$ denotes the $k$-th derivative (as a multilinear operator) with $D^0 g(x) = g(x)$. For a linear operator $P: C^i(\mathbb{T}^3) \to C^j(\mathbb{T}^3)$, define its operator norm as
	\[
	\|P\|_{i \to j} := \sup_{\|g\|_i = 1} \|Pg\|_j.
	\]

    The operator $S_t$ represents the deterministic semigroup generated by the flow $\Phi_t$. For a function $f \in C^{k}(\mathbb{T}^3)$,
    \[
    S_t f(x):= f(\Phi_t(x)).
    \]
    Similarly, the semigroup $\{S_t^{(j)}\}_{t \geq 0}$ generated by $\varphi_t^{(j)}$ is given by
    \[
    S_t^{(j)} f(x):=f(\varphi_t^{(j)}(x)),
    \]
    for $j=1,2,3$. In particular, $m$ iterations of the random splitting corresponds to $S_{\underline{\tau}}^{m}:=S_{\tau_{3m}}^{(3)} \cdots S_{\tau_{1}}^{(1)}$, where $\underline{\tau}=(\tau_1, \tau_2, \dots)$.
	\begin{thm}[\textbf{Almost sure convergence}]\label{convergence th}
		Fix $t > 0$. Then for any $\varepsilon > 0$,
		\[
		P \left( \limsup_{m \to \infty} \left\| S_{\underline{\tau}}^{m^2} - S_t \right\|_{2 \to 0} > \varepsilon \right) = 0,
		\]
		where $\underline{\tau}=(\tau_1, \tau_2, \dots)$, and $\{\tau_i\}_{i=1}^{\infty}$ is a sequence of independent exponential random variables with mean $t/m^2$. 
	\end{thm}
	Theorem \ref{convergence th} asserts that as the mean evolution time $h \to 0$, the random splitting ${\Phi_{\underline{\tau}}^m}$ converges almost surely to the deterministic flow $\Phi_t$ over finite time intervals. The proof follows the argument of Theorem 4.5 in \cite{AMM1}. For brevity, we omit the detailed proof.
	
	We refer to $\{\Phi_{\underline{\tau}}^m\}$ as the one-point Markov chain. Its transition kernel $P:\mathbb{T}^3 \times \mathscr{B}(\mathbb{T}^3) \to [0,1]$ is defined by
    \[
    P(x,D)=\mathbb{P}(\Phi_{\underline{\tau}}(x)\in D)=\mathbb{E}\chi_D(\Phi_{\underline{\tau}}(x))
    \]
    where $\mathscr{B}(\mathbb{T}^3)$ is the Borel $\sigma$-algebra on $\mathbb{T}^3$ and $\chi_D$ is the indicator function of $D$. The operator $P$ acts on measurable functions $g:\mathbb{T}^3 \to \mathbb{R}$ as follows:
    \[
    Pg(x) =\mathbb{E}g(\Phi_{\underline{\tau}}(x)) = \int_{\mathbb{R}_+^3}\frac{1}{h} g(\Phi_{t}(x))e^{-\sum_{i=1}^3 t_i/h} dt,
    \]
    where $\mathbb{R}_+ := (0, \infty)$, $t = (t_1, t_2, t_3)$, and $dt = dt_1 dt_2 dt_3$. And $P$ acts on a measure $\mu$ on $\mathbb{T}^3$ by
    \[
    P^*\mu(D) := \int P(x,D) \mu(dx), \quad D \in \mathscr{B}(\mathbb{T}^3).
    \]
	The stationary measure plays an important role in our results on ergodicity and chaos. Recall that a probability measure $\mu$ on $\mathbb{T}^3$ is \textbf{stationary} for the Markov chain $\{\Phi_{\underline{\tau}}^m\}$ if $P^*\mu = \mu$. We say that $P$ is \textbf{uniquely ergodic} if the set of stationary measures has cardinality one. 
	
	Given a function $V: \mathbb{T}^3 \to [1, \infty)$, we define the weighted norm
	\[
	\|g\|_V := \sup_{x \in \mathbb{T}^3} \frac{|g(x)|}{V(x)},
	\]
	and define $\mathcal{M}_V(\mathbb{T}^3)$ to be the space of measurable observables $g: \mathbb{T}^3 \to \mathbb{R}$ such that $\|g\|_V < \infty$.
	
	We say that $P$ is $V$-\textbf{uniformly geometrically ergodic} if $P$ admits a unique stationary measure $\mu$ on $\mathbb{T}^3$, and there exist constants $C > 0$ and $\gamma \in (0, 1)$ such that the bound
	\[
	\left| P^m g(x) - \int g \, d\mu \right| \leq C V(x) \|g\|_V \gamma^m,\ \text{ for all } m\in\mathbb{N},
	\]
	holds for all $x \in \mathbb{T}^3$ and $g \in \mathcal{M}_V(\mathbb{T}^3)$. We say that $P$ is \textbf{uniformly geometrically ergodic} if it is $V$-uniformly geometrically ergodic in the special case where $V\equiv 1$.
	
	\begin{thm}[\textbf{Uniformly geometrically ergodic}]\label{Uniformly geometrically ergodic}
		Let $\{\Phi_{\underline{\tau}}^m\}$ be as above. Then the transition kernel $P$ of $\{\Phi_{\underline{\tau}}^m\}$ is uniformly geometrically ergodic.
	\end{thm}

	\begin{thm}[\textbf{Positivity of the top Lyapunov exponent}]\label{Positivity of the top Lyapunov exponent}
		Let $\{\Phi_{\underline{\tau}}^m\}$ be as above. Then the random dynamical system $\{\Phi_{\underline{\tau}}^m\}$ almost surely has a positive top Lyapunov exponent 
		\begin{equation*}
			\lambda_1:=\lim_{m\to\infty}\frac{1}{m}\log \|D_x\Phi_{\underline{\tau}}^m(x)\|
		\end{equation*}
		for $\mu$-a.e. $x\in\mathbb{T}^3$.
	\end{thm}
    
    Consider the magnetic induction equation introduced in \eqref{magnetic induction equation} with $\kappa=0$, we demonstrate that the velocity field. We demonstrate that the velocity field
    \begin{align}\label{random velocity field}
    	&u_{\underline{\tau}}(t,x)\nonumber\\
    	=&\sum_{n=0}^{\infty} \left( \tau_{3n+1} X_1(x) \chi_{(3n,3n+1]}(t) + \tau_{3n+2} X_2(x) \chi_{(3n+1,3n+2]}(t) + \tau_{3n+3} X_3(x) \chi_{(3n+2,3n+3]}(t) \right),
    \end{align}
    is an ideal dynamo. This is formalized in the following corollary:
    \begin{cor}[\textbf{Ideal dynamo}]\label{Ideal dynamo}
    	Let $u_{\underline{\tau}}(t,x)$ be as above. Then, for $\mathbb{P}$-a.s. $\underline{\tau}$, the velocity field $u_{\underline{\tau}}$ is an ideal dynamo.
    \end{cor}
    
    \begin{thm}[\textbf{Quenched correlation decay}]\label{Exponential mixing}
    	Let $\{\Phi_{\underline{\tau}}^m\}$ be as above. Then, for any $q, s>0$, there exists a random variable $\xi=\xi_{q,s}\geq 1$ and a constant $\alpha=\alpha_{q,s}>0$ such that for all mean-zero function $g_1,g_2\in H^s(\mathbb{T}^3)$, we have the almost sure estimate
    	\begin{equation}\label{quenched correlation decay}
    		\left|\int_{\mathbb{T}^3} f(x)g\left(\Phi_{\underline{\tau}}^m(x)\right)dx\right|\leq \xi(\underline{\tau})e^{-\alpha m}\|f\|_{H^s}\|g\|_{H^s},
    	\end{equation}
    	while the random variable $\xi$ satisfies the moment bound $\mathbb{E}[\xi^q]<\infty$.
    \end{thm}
    We see that \eqref{quenched correlation decay} quantifies an almost-sure exponential decay of correlations, commonly referred to as quenched correlation decay in the dynamics literature.
    
    \begin{cor}[\textbf{Exponential mixing}]\label{exponential mixing}
    	Let $u_{\underline{\tau}}(t,x)$ be as defined above. For any $s > 0$, let $f$ be a solution to the advection equation \eqref{advection equation} transported by $u_{\underline{\tau}}$ with mean-zero initial data $f_0 \in H^s(\mathbb{T}^3)$. Then, there exists an almost surely finite random variable $\tilde{\xi}=\tilde{\xi}_s$ and a constant $\beta=\beta_{s}>0$ such that
    	\begin{equation*}
    		\|f(t)\|_{H^{-s}} \leq \tilde{\xi}(\underline{\tau}) e^{-\beta t} \|f_0\|_{H^s}.
    	\end{equation*}
    \end{cor}
    Since $f(x) = \sin x$ satisfies assumption (H1), the above results directly lead to the following corollary.
    \begin{cor}[\textbf{Random splitting of ABC flow}]\label{Random splitting of ABC flow}
    	Let $\{\Phi_{\underline{\tau}}^m\}$ be as above with $f_1(x_3)=A\sin x_3$, $f_2(x_1)=B\sin x_1$ and $f_3(x_2)=C\sin x_2$, where $A$, $B$ and $C$ are nonzero constants. Then Theorems \ref{Uniformly geometrically ergodic}, \ref{Positivity of the top Lyapunov exponent}, and \ref{Exponential mixing} hold for the random splitting of the ABC flow. Moreover, Corollaries \ref{Ideal dynamo} and \ref{exponential mixing} apply to the corresponding velocity field $u_{\underline{\tau}}$.
    \end{cor}
	
	\section{Uniform geometric ergodicity for the one-point Markov chain}\label{section 3}
	\quad Some of the concepts used in the sequel are best introduced in a general abstract setting. We begin by covering some preliminaries on Markov chains in this abstract setting.
	\subsection{Markov chain preliminaries}\label{subsection 3.1}
	\quad The lifted chain is the Markov chain $\{\hat{\Phi}_{\underline{\tau}}^m\}$ on the tangent bundle $T\mathbb{T}^3$ with transition kernel 
	\[
	\hat{P}((x,u),\hat{A})=\mathbb{P}(\hat{\Phi}_{\underline{\tau}}(x,u)\in\hat{A}),\ \ (x,u)\in T\mathbb{T}^3,
	\] 
	its corresponding vector fields are $\hat{\mathcal{X}}=\{\hat{X_1}, \hat{X_2}, \hat{X_3}\}$, where for $(x,u)\in T\mathbb{T}^3$,
	\begin{equation*}
		\hat{\Phi}_{\underline{\tau}}^m(x,u)=\left(\Phi_{\underline{\tau}}^m(x), D_x\Phi_{\underline{\tau}}^m u\right),
	\end{equation*}
	and
	\[ 
	\hat{X_i}(x,u) = \begin{pmatrix}
		X_i(x) \\
		DX_i(x)u
	\end{pmatrix},\quad 
	i=1, 2,3.
	\]
	
	The projective chain is the Markov chain $\{\check{\Phi}_{\underline{\tau}}^m\}$ on the sphere bundle $S\mathbb{T}^3$ of $\mathbb{T}^3$ with transition kernel 
	\[
	\check{P}((x,u),\check{A})=\mathbb{P}(\check{\Phi}_{\underline{\tau}}(x,u)\in\check{A}), \ \ (x,u)\in S\mathbb{T}^3,
	\]
	its corresponding vector fields are $\check{\mathcal{X}}=\{\check{X_1}, \check{X_2}, \check{X_3}\}$, where for $(x,u)\in S\mathbb{T}^3$, 
	\begin{equation*}
		\check{\Phi}_{\underline{\tau}}^m(x,u)=\left(\Phi_{\underline{\tau}}^m(x), \frac{D_x\Phi_{\underline{\tau}}^m u}{|D_x\Phi_{\underline{\tau}}^m u|}\right),
	\end{equation*}
	and 
	\[ 
	\check{X_i}(x,u) = \begin{pmatrix}
		X_i(x) \\
		DX_i(x)u-\langle DX_i(x)u, u \rangle u
	\end{pmatrix},\quad 
	i=1, 2, 3.
	\]
	
	The two-point chain is the Markov chain $\{\tilde{\Phi}_{\underline{\tau}}^m\}$ on $\mathcal{T}:=\mathbb{T}^3\times \mathbb{T}^3\setminus\mathcal{I}$ ($\mathcal{I}$ is the invariant set of $\tilde{\Phi}_{\underline{\tau}}$, i.e., $\tilde{\Phi}_{\underline{\tau}}\mathcal{I}=\mathcal{I}$ for every $\underline{\tau}$) with transition kernel 
	\[
	P^{(2)}((x,y),A^{(2)})=\mathbb{P}(\tilde{\Phi}_{\underline{\tau}}(x,y)\in A^{(2)}), \ \ (x,y)\in\mathcal{T},
	\]
	its corresponding vector fields are $\tilde{\mathcal{X}}=\{\tilde{X_1}, \tilde{X_2}, \tilde{X_3}\}$, where for $(x,y)\in\mathcal{T}$,
	\begin{equation*}
		\tilde{\Phi}_{\underline{\tau}}^m(x,y)=(\Phi_{\underline{\tau}}^m(x), \Phi_{\underline{\tau}}^m(y)),
	\end{equation*}
	and 
	\[ 
	\tilde{X_i}(x,y) = \begin{pmatrix}
		X_i(x) \\
		X_i(y)
	\end{pmatrix},\quad
	i=1, 2, 3.
	\]
	
	Throughout this paper, we use $\{\Psi_{\underline{\tau}}^m\}$ to denote any of the aforementioned chains, including the one-point chain, two-point chain, and projective chain. Let $M$ denote the corresponding state space of dimension $d$, $Q$ the associated Markov transition kernel, and $\mathcal{F}:=\{Y_1, Y_2, Y_3\}$ the corresponding vector field. Furthermore, $\Psi_{\underline{\tau}}^m$ is defined recursively by $\Psi_{\underline{\tau}}^m = \psi_{\tau_{3m}}^{(3)} \circ \psi_{\tau_{3m-1}}^{(2)}\circ\psi_{\tau_{3m-2}}^{(1)}(\Psi_{\underline{\tau}}^{m-1})$. We say that $Q$ is \textbf{Feller} if $Qg$ is continuous for any bounded, continuous function $g$ on $M$. We say that a set $D\subset M$ is \textbf{small} for $\{\Psi_{\underline{\tau}}^m\}$  if there exists an $m_0>0$, and a positive measure $\nu$ on $M$ such that for any $x\in D$,
	\[
	Q^{m_0}(x,E)\geq \nu(E), 
	\]
	for all $E\in\mathscr{B}(M)$.  We say that $\{\Psi_{\underline{\tau}}^m\}$ is \textbf{topologically irreducible} if for every $x\in M$ and open set $U\subset M$, there exists $m_0>0$ such that 
	\[
	Q^{m_0}(x,U)>0.
	\]
	We say that $\{\Psi_{\underline{\tau}}^m\}$ is \textbf{strongly aperiodic} if there exists $x_0\in M$ such that for all open sets $U$ containing $x_0$, we have
	\[
	Q(x_0,U)>0.
	\]
	We say that a function $V:M \to[1,\infty)$ satisfies a \textbf{Lyapunov-Foster drift condition} for $\{\Psi_{\underline{\tau}}^m\}$ if there exist $\alpha\in (0,1)$, $b>0$ and a compact set $C\subset M$ such that 
	\[
	QV\leq \alpha V+b\chi_C,
	\]
	where $\chi_C$ is the characteristic function on $C$.
	\begin{thm}[\cite{BCZG}, Theorem 2.3]\label{conditions for uniformly geometrically ergodic}
		Let $\{\Psi_{\underline{\tau}}^m\}_{m\geq 0}$ and $Q$ be as above. If $Q$ is a Feller transition kernel and assume the following:
		\begin{itemize}
			\item[(a)] There exists an open small set for $\{\Psi_{\underline{\tau}}^m\}$.
			\item[(b)] $\{\Psi_{\underline{\tau}}^m\}$ is topologically irreducible.
			\item[(c)] $\{\Psi_{\underline{\tau}}^m\}$ is strongly aperiodic.
			\item[(d)] There exists a function $V: X\to[1,\infty)$ that satisfies a Lyapunov-Foster drift condition for $\{\Psi_{\underline{\tau}}^m\}$.
		\end{itemize}
		Then $Q$ is $V$-uniformly geometrically ergodic.
	\end{thm}
	\subsection{Proof of Theorem \ref{Uniformly geometrically ergodic}}\label{subsection 3.2}
	\quad The proof of Theorem \ref{Uniformly geometrically ergodic} using the following results. The first is a version of Theorem 1 in \cite[Chapter 3]{J}, it says that if the Lie bracket condition holds at $x\in M$, then due to the surjectivity of $D_t\Psi^m(x,t)$, the random dynamics can locally reach any infinitesimal direction in the tangent space from $x$ in arbitrarily small times. 
	
	\begin{lem}\label{surjective}
		Let $\{\Psi_{\underline{\tau}}^m\}$ be as above. Assume $\text{Lie}_x(\mathcal{F})=T_xM$ for some $x\in M$. Then for any neighborhood $U$ of $x$ and any $T'>0$ there exists a point $y$ in $U$, an $m$ and a $\tilde{t}=(\tilde{t}_1,\cdots,\tilde{t}_{3m})\in\mathbb{R}_+^{3m}$ such that $\sum_{i=1}^{3m}\tilde{t}_i\leq T'$ and $\Psi^m(x,\cdot): t\mapsto \Psi^m(x,t)$ is a submersion at $t=\tilde{t}$ and $\Psi^m(x,\tilde{t})=y$, i.e., $D_t\Psi^m(x,\tilde{t}):T_{\tilde{t}}\mathbb{R}_+^{3m}\to T_yM$ is surjective.
	\end{lem}
	\begin{proof}
		Fix any neighborhood $U$ of $x$ and any $T'>0$. Choose $Y_{i_1} \in\mathcal{F}$ such that $Y_{i_1}(x)\neq 0$. Let $\varepsilon_1>0$ be such that $Y_{i_1}(\psi_t^{(i_1)}(x))\neq 0$ for all $t\in(0,\varepsilon_1)$. For sufficiently small $\varepsilon_1<T'$, $M_1=\{\psi_t^{(i_1)}(x): t\in(0,\varepsilon_1)\}\subset U$. By the constant rank theorem, $M_1$ is a one-dimensional submanifold of $M$, with $t\mapsto \psi_t^{(i_1)}(x)$ its coordinate map. If $\dim M=\dim M_1$, the proof is finished. Otherwise, since a classical result asserts that if $X$ and $Y$ are vector fields that are tangent to a submanifold, then their Lie bracket $[X,Y]$ and any linear combination $\alpha X+\beta Y$ are also tangent to this submanifold, there exists a vector field $Y_{i_2}$ in $\mathcal{F}$ such that $Y_{i_2}$ is not tangent to $M_1$.
		
		Let $x_1$ be a point in $M_1$ such that $Y_{i_2}(x_1)$ is not colinear with $Y_{i_1}(x_1)$. Denote by $ \hat{t}_{1}$ the point in $(0,\varepsilon_1)$ such that $\psi_{\hat{t}_{1}}^{(i_1)}(x)=x_1$. Then $Y_{i_1}$ and $Y_{i_2}$ are not collinear for all points in a neighborhood of $x_1$. Therefore, there exists a neighborhood $I_1\subset (0,\varepsilon_1)$ of $\hat{t}_1$ and $\varepsilon_2>0$ such that the mapping $F_2$ given by $\psi_{t_2}^{(i_2)}\circ\psi_{t_1}^{(i_1)}(x)$ from $I_1\times (0,\varepsilon_2)$ into $M$ satisfies $\varepsilon_1+\varepsilon_2<T'$, $F_2(I_1\times (0,\varepsilon_2))\subset U$ and the tangent map of $F$ has rank $2$ at each point of $I_1\times (0,\varepsilon_2)$. 
		
		Continuing in this way, we obtain a sequence $Y_{i_1}, Y_{i_2},\cdots, Y_{i_m}$ of vector fields in $\mathcal{F}$ with the property that the tangent map of the mapping $F_m(t_1,\cdots,t_m)=\psi_{t_m}^{(i_m)}\circ\cdots\circ\psi_{t_1}^{(i_1)}(x)$ has rank $m$ at a point $\hat{t}=(\hat{t}_{1},\cdots,\hat{t}_{m})\in\mathbb{R}_+^{m}$, which further satisfies $\hat{t}_{1}+\cdots+\hat{t}_{m}<T'$ and $F_m(\hat{t}_{1},\cdots,\hat{t}_{m})\in U$. Again by the constant rank theorem, the image of a neighborhood of $(\hat{t}_{1},\cdots,\hat{t}_{m})$ under the mapping $F_m$ is an $m$-dimensional submanifold $M_m$ of $M\cap U$.
		
		The procedure stops precisely when each element of $\mathcal{F}$ is tangent to $M_m$. Then, there exists a point $\hat{y}=\psi_{\hat{t}_{m}}^{(i_{m})}\circ\cdots\circ\psi_{\hat{t}_{1}}^{(i_1)}(x)$ in $U$, $m\in\mathbb{N}$ and $\hat{t}=(\hat{t}_{1},\cdots,\hat{t}_{m})\in\mathbb{R}_+^{m}$ such that $\hat{t}_{1}+\cdots+\hat{t}_{m}<T'$ and $t\mapsto\psi_{t_{m}}^{(i_{m})}\circ\cdots\circ\psi_{t_{1}}^{(i_1)}(x)$ is a submersion at $\hat{t}$. Then, we can take a neighborhood $U_{\hat{y}}\subset U$ of $\hat{y}$ and $\tilde{t}=(\tilde{t}_1,\cdots,\tilde{t}_{3m})\in\mathbb{R}_+^{3m}$ with $\tilde{t}_{3(k-1)+i_k}=\hat{t}_k$ for $k=1,\cdots,m$ and $\tilde{t}_{3(k-1)+j}>0$ for $k=1,\cdots,m$ and $j\in \{1,2,3\}$ not equal to $i_k$ sufficiently small such that $\sum_{i=1}^{3m}\tilde{t}_i\leq T'$ and $t\mapsto\Psi^m(x,t)$ is a submersion at $\tilde{t}$ and $y:=\Psi^m(x,\tilde{t})\in U_{\hat{y}}\subset U$.
	\end{proof}
	Secondly, we present the subsequent lemma, which essentially provides a lower bound on the probability that the image of the random variable lies in a given set, under certain geometric and probabilistic conditions.
	\begin{lem}\label{absolutely continuous}
		Let $\tau$ be a random variable taking values in $U$ which is an open subset of $\mathbb{R}^m$, with a continuous density $\rho$. Let $d\leq m$ and let $\phi:U\times M\to M$, $(t,x)\mapsto \phi_x(t)$ be a $C^1$ map. Suppose that $\phi: U \times M \to M$, $(t, x) \mapsto \phi_x(t)$ is a $C^1$ map, and that for some $(t_0, x_0) \in U \times M$, the map $\phi_{x_0}$ is a submersion at $t_0$ and $\rho$ is bounded below by a constant $c_0 > 0$ in a neighborhood of $t_0$. Then, there exists a constant $c>0$ and neighborhoods $U_{x_0}$ of $x_0$ and $U_{\hat{x}}$ of $\hat{x}:=\phi(t_0,x_0)$ such that 
		\begin{equation}\label{lower bound}
			\mathbb{P}(\phi(\tau,x)\in D)\geq c{\text{Leb}}_M(D\cap U_{\hat{x}}), \ \ \text{for all}\ \  x\in U_{x_0},\ D\in\mathscr{B}(M).
		\end{equation}
	\end{lem}
	\begin{proof}
		Write $t=(t^1,t^2)\in\mathbb{R}^d\times\mathbb{R}^{m-d}$. Since $\phi_{x_0}$ is a submersion at $t_0$, by the constant rank theorem, we may assume $\frac{\partial \phi_x(t^1,t^2)}{\partial t^1}$ is surjective on some neighborhood $U_0$ of $x_0$ and $V_0$ of $t_0$. Define 
		\begin{align*}
			\tilde{\phi}_x: V_0&\to M\times \mathbb{R}^{m-d},\\
			(t^1,t^2) &\mapsto (\phi_x(t^1,t^2),t^2).
		\end{align*}
		Then $D_t\tilde{\phi}_x$ is invertible on $V_0$ for any $x\in U_0$. Note that $\tilde{\phi}_{x_0}(t_0)=(\hat{x},t_0^2)$, we can choose a neighborhood $U_{\hat{x}}$ of $\hat{x}$, $U_2$ of $t_0^2$ and $U_{x_0}$ of $x_0$ such that $|\det D\tilde{\phi}_x^{-1}(y)|\geq c_1>0$ for any $x\in U_{x_0}$ and any $y\in U_{\hat{x}}\times U_2$.
		
		Write the random variable $\tau$ as a couple $(\tau^1,\tau^2)$, and let $E\subset U_{x_0}$ be a Borel set,
		\begin{align*}
			\mathbb{P}(\phi(x,\tau)\in B)&\geq \mathbb{P}(\phi(x,\tau)\in E, \tau^2\in U_2)\\
			&= \mathbb{P}(\tilde{\phi}_x(\tau)\in E\times U_2)=\mathbb{P}(\tau\in \tilde{\phi}_x^{-1}(E\times U_2))\\
			&=\int_{\tilde{\phi}_x^{-1}(E\times U_2)}\rho(t)dt\geq c_0\int_{\tilde{\phi}_x^{-1}(E\times U_2)}dt\\
			&=c_0\int_{E\times U_2}|\det D\tilde{\phi}_x^{-1}(y)|dy\\
			&\geq c_0c_1\int_{E\times U_2}dy=c_0c_1\text{Leb}_M(E)\text{Leb}_{\mathbb{R}^{m-d}}(U_2).
		\end{align*}
		Therefore, \eqref{lower bound} holds with $c=c_0c_1\text{Leb}_{\mathbb{R}^{m-d}}(U_2)$.
	\end{proof}
	We are now prepared to prove the existence of an open small set for $\{\Psi_{\underline{\tau}}^m\}$.
	\begin{lem}\label{small set}
		Let $\{\Psi_{\underline{\tau}}^m\}$ and $Q$ be as above. Assume $\text{Lie}_x(\mathcal{F})=T_xM$ for some $x\in M$. Then there exists an open small set for $\{\Psi_{\underline{\tau}}^m\}$.
	\end{lem}
	\begin{proof}
		It follows from Lemma \ref{surjective} that there exists $\tilde{t}=(\tilde{t}_1,\cdots,\tilde{t}_{3m})\in\mathbb{R}_+^{3m}$ such that $\psi^m(x,\cdot)$ is a submersion at $\tilde{t}$. Then, by Lemma \ref{absolutely continuous}, there exist a constant $c>0$ and neighborhood $U$ of $x$ and $U'$ of $\psi^m(x,\tilde{t})$ such that
		\begin{equation*}
			\mathbb{P}(\psi_{\underline{\tau}}^m(y)\in D)\geq c\text{Leb}_M(D\cap U'), \ \ \text{for all}\ \  y\in U,\ D\in\mathscr{B}(M).
		\end{equation*}
		We conclude $U$ is a $\nu$-small set with $\nu(D):=c\text{Leb}_M(D\cap U')$.
	\end{proof}
	Next, we prove that the one-point Markov chain is topologically irreducible. We say that a set $D\subseteq M$ is \textbf{reachable} from $E\subseteq M$ for $\{\Psi_{\underline{\tau}}^m\}$ if for every point $x\in E$, there exists $m\geq 0$ and $\underline{\tau}=(\tau_1,\tau_2,\cdots,\tau_{3m})\in \mathbb{R}_{\geq 0}^{3m}$ such that $\Psi_{\underline{\tau}}^m(x)\in D$. We say that $\{\Psi_{\underline{\tau}}^m\}$ is \textbf{exactly controllable} if for any $x, y\in M$, there exists $m\geq 0$ and $\underline{\tau}=(\tau_1,\tau_2,\cdots,\tau_{3m})\in \mathbb{R}_{\geq 0}^{3m}$ such that $\Psi_{\underline{\tau}}^m(x)=y$. Clearly, exactly controllability of $\{\Psi_{\underline{\tau}}^m\}$ implies that it is topologically irreducible. 
	\begin{lem}\label{irreducibility}
		Let $\{{\Phi}_{\underline{\tau}}^m\}$ be the one-point chain as above. Then, $\{{\Phi}_{\underline{\tau}}^m\}$ is exactly controllable.
	\end{lem}
	\begin{proof}
		Since $f_i$, $i=1, 2, 3$ are analytic on $\mathbb{S}^1$, let $y_3$ be the point in $\mathbb{S}^1$ where $|f_1|$ attains its maximum value, $y_1$ be the point in $\mathbb{S}^1$ where $|f_2|$ attains its maximum value and $y_2$ be the point in $\mathbb{S}^1$ where $|f_3|$ attains its maximum value. It is evident that $f_1'(y_3)=f_2'(y_1)=f_3'(y_2)=0$. 
		
		Next, we prove that for any $(x_1, x_2, x_3)\in \mathbb{T}^3$, there exists $m\geq 0$ and $\underline{\tau}=(\tau_1,\tau_2,\cdots,\tau_{3m})\in [0,\infty )^{3m}$ such that $\Phi_{\underline{\tau}}^{m}(x_1, x_2, x_3)=(y_1, y_2, y_3)$. 
		
		\begin{itemize}
			\item[(i)] If $f_1(x_3)$, $f_2(x_1)$, or $f_3(x_2)$ is nonzero, without loss of generality, we assume that $f_1(x_3) \neq 0$. Choose $i, j, k \in \mathbb{Z}$ such that 
			\[
			\tau_1 = \frac{y_1 - x_1 + 2i\pi}{f_1(x_3)} \geq 0, \ \tau_2 = \frac{y_2 - x_2 -\tau_1 f_1'(x_3)+ 2j\pi}{f_2(y_1)} \geq 0,\ \tau_3 = \frac{y_3 - x_3 + 2k\pi}{f_3(y_2)} \geq 0.
			\]
			Then, $\Phi_{\underline{\tau}}(x_1, x_2, x_3) = (y_1, y_2, y_3)$.
			\item[(ii)] If $f_1(x_3)$, $f_2(x_1)$, and $f_3(x_2)$ are all zero, by assumption (H1), choose $i_1, j_1, k_1 \in \mathbb{Z}$ such that 
			\[
			\tau_1 = \frac{y_2 - x_2 + 2i_1\pi}{f_1'(x_3)} \geq 0, \quad \tau_3 = \frac{y_3 - x_3 + 2j_1\pi}{f_3(y_2)} \geq 0, \quad \text{and} \quad \tau_4 = \frac{y_1 - x_1 + 2k_1\pi}{f_1(y_3)} \geq 0.
			\]
			Set $\tau_2 = \tau_5 = \tau_6 = 0$. Then, $\Phi_{\underline{\tau}}^2(x_1, x_2, x_3) = (y_1, y_2, y_3)$.
			
		\end{itemize}
		
		We conclude by noting that for any two points $(x_1,x_2,x_3)$ and $(z_1,z_2,z_3)$ there exists $\tilde{m}\geq 0$ and $\underline{\tilde{\tau}}=(\tilde{\tau}_1,\tilde{\tau}_2,\cdots,\tilde{\tau}_{3\tilde{m}})$ such that $\Phi_{\underline{\tilde{\tau}}}^{\tilde{m}}(x_1,x_2,x_3)=(z_1,z_2,z_3)$ by first traveling to $(y_1,y_2,y_3)$ and then traveling to $(z_1,z_2,z_3)$ (using the same arguments as above, but in reverse).
	\end{proof}
	We are now ready to prove the main theorem on the uniform geometric ergodicity for the one-point chain.
	\begin{proof}[\textbf{Proof of Theorem \ref{Uniformly geometrically ergodic}}]
		It suffices to verify the conditions of Theorem \ref{conditions for uniformly geometrically ergodic}. Notice that $\Phi_{\underline{\tau}}$ is a volume-preserving diffeomorphism of $\mathbb{T}^3$ for every $\underline{\tau}$. Clearly, the Lebesgue probability measure $\mu$ on $\mathbb{T}^3$ is a stationary measure for the Markov chain $\{\Phi_{\underline{\tau}}^m\}$. Since $f_i$, $i=1, 2, 3$ are analytic on $\mathbb{S}^1$, let $x_3$ be the point in $\mathbb{S}^1$ where $|f_1|$ attains its maximum value, $x_1$ be the point in $\mathbb{S}^1$ where $|f_2|$ attains its maximum value and $x_2$ be the point in $\mathbb{S}^1$ where $|f_3|$ attains its maximum value. Then, the following matrix
		\begin{align*}
			(X_1(x),X_2(x), X_3(x))=
			\begin{pmatrix}
				f_1(x_3) & 0 & 0\\
				0 & f_2(x_1) & 0\\
				0 & 0 & f_3(x_2)
			\end{pmatrix}
		\end{align*}
		has rank 3 for $x=(x_1, x_2, x_3)\in\mathbb{T}^3$, implying $\text{Lie}_x(\mathcal{X})=T_x\mathbb{T}^3$. Then, the existence of an open small set can be deduced from Lemma \ref{small set}. The topological irreducibility property is established via Lemma \ref{irreducibility}. Strong aperiodicity is a direct consequence of $\Phi_0(x)=x$ and the fact that $\mathbb{P}(\underline{\tau}<\varepsilon)>0$ for any $\varepsilon>0$. Finally, since $\mathbb{T}^3$ is compact, $P$ automatically satisfies a Lyapunov-Foster drift condition for the drift function $V \equiv 1$. Consequently, $P$ is uniformly geometrically ergodic, with $\mu$ as the unique stationary measure.
	\end{proof}
	
	\section{Positivity of the top Lyapunov exponent}\label{section 4}
	\quad The proof of Theorem \ref{Positivity of the top Lyapunov exponent} begins with demonstrating the existence and almost sure constancy of the top Lyapunov exponent $\lambda_1$, using tools from random dynamical systems theory. In order to proceed with this, we introduce some notations. For any $x, y\in\mathbb{T}^3$, we use $|u|$ to represent the norm of a vector $u$ in the tangent space $T_x\mathbb{T}^3$, and $\|T\|$ for the operator norm of a linear mapping $T: T_x\mathbb{T}^3 \rightarrow T_y\mathbb{T}^3$. For any positive real number $r$, we define $\log^+(r) = \max(\log r, 0)$. Let $\mu$ be as in Theorem \ref{Uniformly geometrically ergodic}, which is the unique stationary measure for $P$ on $\mathbb{T}^3$. Given that $f_i$ belongs to $C^{\omega}(\mathbb{S}^1, \mathbb{R})$ and the torus $\mathbb{T}^{3}$ is bounded, we can establish the following integrability condition:
	\begin{equation*}
		\mathbb{E}\int_{\mathbb{T}^3} (\log^+\|D_x\Phi_{\underline{\tau}}(x)\|+\log^+\|D_x\Phi_{\underline{\tau}}(x)^{-1}\|)\mu(dx)<\infty.
	\end{equation*}
    In fact, for any $(x,u) \in S\mathbb{T}^3$, there exists a constant $C_0>1$ satisfying
    \begin{equation*}
    	\frac{1}{C_0 (1 + 2\tau_1)(1 + 2\tau_2)(1 + 2\tau_3)}\leq \|D_x \Phi_{\underline{\tau}}(x)\| \leq C_0 (1 + 2\tau_1)(1 + 2\tau_2)(1 + 2\tau_3). 
    \end{equation*}
    Then, for any $(x,u) \in S\mathbb{T}^3$, we have 
    \begin{align*}
        \mathbb{E}\left| \log \|D_{x} \Phi_{\underline{\tau}}(x)\|\right|
    	 &\leq \int_{\mathbb{R}_{+}^3} \left( \sum_{i=1}^3 \log [C_0^{1/3} (1+2t_i)]\right) \frac{1}{h^3} e^{-(t_1+t_2+t_3)/h} dt_1dt_2dt_3 \nonumber\\
    	&= 3 \int_0^\infty \log [C_0^{1/3} (1+2ht)]  e^{-t} dt.
    \end{align*}
    Similar estimates holds for $\mathbb{E}\left| \log \|D_{x} \Phi_{\underline{\tau}}(x)^{-1}\|\right|$. Thus, by Fubini theorem, we have 
    \begin{align*}
    	&\mathbb{E}\int_{\mathbb{T}^3} \left(\log^+\|D_x\Phi_{\underline{\tau}}(x)\|+\log^+\|D_x\Phi_{\underline{\tau}}(x)^{-1}\| \right) d\mu \\
    	\leq & \int_{\mathbb{T}^3} \mathbb{E} \left( \left|\log  \|D_x\Phi_{\underline{\tau}}(x)\| \right|+\left|\log  \|D_x\Phi_{\underline{\tau}}(x)^{-1}\| \right| \right) d\mu \\
    	= & 6 \int_0^\infty \log [C_0^{1/3} (1+2ht)]  e^{-t} dt <\infty.
    \end{align*}
    This confirms the integrability condition.
    Under this condition, the multiplicative ergodic theorem ensures the existence of the top Lyapunov exponent $\lambda_1$ and is constant for $\mu$-a.e. $x\in\mathbb{T}^3$ and almost every $\underline{\tau}$. Since $\Phi_{\underline{\tau}}^m$ is volume-preserving, it follows that $\det(D_x\Phi_{\underline{\tau}}^m) \equiv 1$. Consequently, the sum of the Lyapunov exponents $\lambda_{\Sigma}$ which is defined as $\lim_{m \to \infty} \frac{1}{m}\log |\det(D_x\Phi_{\underline{\tau}}^m)|$ equals to zero. The task ahead involves proving that the top Lyapunov exponent, $\lambda_1$, does not degenerate to zero, ensuring its positive value. To achieve this, we first analyze the consequences of the condition $3\lambda_1= \lambda_{\Sigma}$. By subsequently ruling out these consequences, we conclude that $3\lambda_1 > \lambda_{\Sigma} = 0$.
	
	\subsection{Consequences of \boldmath $3\lambda_1=\lambda_{\Sigma}$}\label{subsection 4.1}
	\quad For any $m$ and $\underline{\tau}$, we define the pushforward of $\mu$ under $\Phi_{\underline{\tau}}^m$ as the probability measure $\mu_m$ on $\mathbb{T}^3$, which is given by
	\begin{equation*}
		\mu_m(D):=(\Phi_{\underline{\tau}}^m)_*\mu(D)=\mu((\Phi_{\underline{\tau}}^m)^{-1}(D)),\ D \in \mathscr{B}(\mathbb{T}^3).
	\end{equation*}  
	This construction allows us to track the evolution of the measure $\mu$ under the action of $\Phi_{\underline{\tau}}^m$. For probability measure $\mu$ and $\mu_m$ on $\mathbb{T}^3$, following Lemma 2.1 in \cite{DV}, 
	the entropy of $\mu_m$ with respect to $\mu$ is defined as follows:
	\[h(\mu;\mu_m)=
	\begin{cases}
		\int\frac{d\mu_m}{d\mu}(x)\log(\frac{d\mu_m}{d\mu}(x))\mu(dx) & \text{if} \ \ \mu_m\ll \mu,\\
		\infty & \text{otherwise},
	\end{cases}
	\]
	where $\mu_m\ll\mu$ means that $\mu_m$ is absolutely continuous with respect to $\mu$, with $d\mu_m/d\mu$ being its Radon-Nikodym derivative. Since $\mu_m$ is a random measure depending on $\underline{\tau}$, we are led to consider the average relative entropy with respect to $\mu$, denoted by $\mathbb{E}h(\mu;\mu_m)$. The finiteness of this value is crucial for our subsequent analysis. Specifically, in cases where $\mu_m$ equals to $\mu$ almost surely, as in our setting, it follows that $\mathbb{E}h(\mu;\mu_m)=\mathbb{E}h(\mu;\mu)=0$.
	
	\begin{thm}\label{degeneracy}
		Let $\{\Phi_{\underline{\tau}}^m\}$ be the one-point chain as above. Assume (H1) holds and $3\lambda_1=\lambda_\Sigma$. Then either
		\begin{itemize}
			\item[(a)] there exists a Riemannian structure $\{g_x:x\in \mathbb{T}^3\}$ on $\mathbb{T}^3$ such that 
			\begin{equation}\label{conformal}
				g_{\Phi_t^m(x)}(D_x\Phi_t^m(x)u,D_x\Phi_t^m(x)v)=g_x(u,v)
			\end{equation}
			for all $m$, $t\in\mathbb{R}_{\geq 0}^{3m}$, $x\in \mathbb{T}^3$ and $u,v\in T_x\mathbb{T}^3$; or
			\item[(b)] there exist proper linear subspaces $E_x^1,\cdots,E_x^p$ of $T_x \mathbb{T}^3$ for all $x\in \mathbb{T}^3$ such that 
			$$D_x\Phi_t^m(E_x^i)=E_{\Phi_t^m(x)}^{\sigma(i)}, \ \ 1\leq i\leq p$$
			for all $m$, $t\in\mathbb{R}_{\geq 0}^{3m}$, $x\in \mathbb{T}^3$ and some permutation $\sigma$.
		\end{itemize}
	\end{thm}
    Theorem \ref{degeneracy} is a version of Theorem 6.8 in \cite{B}. The slight difference between them lies in the fact that, in our case, the uniform geometric ergodicity of $P$ implies that any invariant conditional distribution $\nu$ (to be defined later) has a version where $x\mapsto \nu_x$ is continuous on $\mathbb{T}^3$, as established in Lemma \ref{RCPD}. As a result, we can apply the conclusion of Proposition 6.3 in \cite{B} to our situation. To achieve this, we establish some preliminary notations. 
    
    Let $\mathscr{M}(S\mathbb{T}^3)$ denote the space of Borel probability measures on $S\mathbb{T}^3$ with the weak topology and $\pi:S\mathbb{T}^3\to\mathbb{T}^3$ is the natural bundle map. Let $\mathscr{M}_{\mu}(S\mathbb{T}^3)$ denote the set of $\nu\in\mathscr{M}(S\mathbb{T}^3)$ such that $\nu \pi^{-1}=\mu$. Since the bundle has compact fibre, the set $\mathscr{M}_{\mu}(S\mathbb{T}^3)$ is uniformly tight and hence it is a compact subset of $\mathscr{M}(S\mathbb{T}^3)$. Suppose $\nu \in \mathscr{M}_{\mu}(S\mathbb{T}^3)$, we call $\{\nu_x\}_{x\in \mathbb{T}^3}$ a \textbf{factorization}  (or disintegration, or sample measure) of $\nu$ respect to $\mu$, if
    \begin{itemize}
    	\item[(1)] $x \mapsto \nu_x$ is measurable;
    	\item[(2)] for $\mu-\text{a.e.}$ $x \in \mathbb{T}^3$, $\nu_x$ is a probability measure on  fiber $\pi^{-1}(x)$;
    	\item[(3)] for any non-negative Borel measurable function f on $S\mathbb{T}^3$, 
    	\[
    	\int_{S\mathbb{T}^3} f(\check{x}) d\nu(\check{x}) = \int_{\mathbb{T}^3} \int_{\pi^{-1}(x)} f(v) d\nu_x(v)d\mu(x).
    	\]
    \end{itemize}
    For any $\nu \in \mathscr{M}_{\mu}(S\mathbb{T}^3)$, by Proposition 1.4.3 in \cite{A}, a factorization of $\nu$ with respect to $\mu$ exists and is $\mu$-a.e. unique. Let us recall the projective chain ${\check{\Phi}_{\underline{\tau}}^m}$, defined for $(x,u) \in S\mathbb{T}^3$, as follows:
    \begin{equation*}		\check{\Phi}_{\underline{\tau}}^m(x,u)=\left(\Phi_{\underline{\tau}}^m(x),\frac{D_x\Phi_{\underline{\tau}}^m(x)u}{|D_x\Phi_{\underline{\tau}}^m(x)u|}\right).
    \end{equation*}
    Denote 
    \[A_{\underline{\tau},x}^m(u):=\frac{D_x\Phi_{\underline{\tau}}^m(x)u}{|D_x\Phi_{\underline{\tau}}^m(x)u|}.\] 
    We say $\nu\in \mathscr{M}_{\mu}(S\mathbb{T}^3)$ satisfies \textbf{invariant conditional distribution} if, for all $m$, $\mu_m\ll \mu$ and $\mu\{x:(A_{\underline{\tau},x}^m)_*\nu_x=\nu_{\Phi_{\underline{\tau}}^m(x)}\}=1$ for almost every $\underline{\tau}$. By Corollary 5.6 in \cite{B}, assuming the integrability condition and $\mathbb{E}h(\mu;\mu_m)<\infty$, if $3\lambda_1=\lambda_{\Sigma}$, there exists $\nu\in\mathscr{M}_{\mu}(S\mathbb{T}^3)$ that satisfies invariant conditional distribution. 
    We are now prepared to present the following lemma, the proof of which can be found in Lemma B.2 of \cite{BCZG}.
	
	\begin{lem}\label{RCPD}
		Let $\{\Phi_{\underline{\tau}}^m\}$ be as above. Assume (H1) holds and $3\lambda_1=\lambda_{\Sigma}$, then there exists $\nu\in\mathscr{M}_{\mu}(S\mathbb{T}^3)$ that satisfies invariant conditional distribution. Moreover, there exists a version of the factorization of $\nu$ respect to $\mu$ such that $x\mapsto \nu_x$ is continuous on $\mathbb{T}^3$. In particular, for all $m$, $(A_{\underline{\tau},x}^m)_*\nu_x=\nu_{\Phi_{\underline{\tau}}^m(x)}$ for every $x\in \mathbb{T}^3$ and $\underline{\tau}\in \mathbb{R}_{\geq 0}^{3m}$.
	\end{lem}
	\begin{proof}
		First, we prove that there exists an increasing sequence of compact subsets $K_1\subset K_2\subset\cdots$ of $\mathbb{T}^3$ such that for each $n$, we have $\mu(K_n)\geq 1-1/n$ and $x\mapsto \nu_x$ is continuous along $x\in K_n$. By the compactness of $S\mathbb{T}^3$ and Stone-Weierstrass theorem, there exists $\{g_k\}_{k \in \mathbb{N}}\subset C(S\mathbb{T}^3,\mathbb{R})$ which is dense in $C(S\mathbb{T}^3,\mathbb{R})$. For any $k$ and $n$, by Lusin's theorem, there exists a compact subset $C_{n,k}$ of $\mathbb{T}^3$ such that $
		\mu(C_{n,k})\geq 1-1/(n2^k)$, and the measurable function
		\[
		G_k(x):=\int_{S_x\mathbb{T}^3} g_k(v)d\nu_x(v)
		\]
		is continuous on $C_{n,k}$. Define
		\[		C_n:=\bigcap_{k=1}^{\infty}C_{n,k},\quad
		K_n:=\bigcup_{m=1}^n C_m.
		\]
		Then
		\[
		\mu(K_n)\geq \mu(C_n)=1-\mu(C_n^c)\geq 1-\frac{1}{n}.
		\]
		Furthermore, for any $k$ and $n$, $G_k$ is continuous on $K_n$. Since $\{g_k\}$ is dense in $C(S\mathbb{T}^3,\mathbb{R})$,  $x\mapsto \nu_x$ is continuous along $x\in K_n$.
		
		For any fixed $g\in C(S\mathbb{T}^3,\mathbb{R})$, denote $F_g : \mathbb{T}^3 \times GL_3(\mathbb{R}) \to \mathbb{R}$ by
		\[
		F_g(x, M):=\int_{S_x \mathbb{T}^3} g(Mv)d\nu_x(v).
		\]
		Then $F_g$ is continuous on $K_n\times GL_3(\mathbb{R})$. By Tietze extension theorem, for each $n$ there exists a function $F_{g,n}: \mathbb{T}^3\times GL_3(\mathbb{R})\to\mathbb{R}$ such that $F_g \chi_n=F_{g,n}\chi_n$ and $\|F_{g,n}\|_{L^\infty}\leq\|F_g\|_{L^\infty}$ where $\chi_n$ represents the indicator function of $K_n\times GL_3(\mathbb{R})$. 
		Denote 
		\[
		G(x):=\int_{S_x \mathbb{T}^3} g(v) d\nu_x(v),
		\]
		and for $F: \mathbb{T}^3\times GL_3(\mathbb{R})\to\mathbb{R}$, denote $Q^n F:\mathbb{T}^3\to\mathbb{R}$ by
		\[
		Q^nF(x):=\mathbb{E}[F(\Phi_{\underline{\tau}}^n(x), (A_{\underline{\tau},x}^n)^{-1})].
		\]
		By Corollary 5.6 in \cite{B}, there exists $\nu\in\mathscr{M}_{\mu}(S\mathbb{T}^3)$ satisfing invariant conditional distribution. Let $\mathcal{O}\subset \mathbb{T}^3$ be the $\mu$-full measure set such that for any $x \in \mathcal{O}$, $(A_{\underline{\tau},x}^n)_*\nu_x=\nu_{\Phi_{\underline{\tau}}^n(x)}$ for almost every $\underline{\tau}$. Then, for any $x\in\mathcal{O}$, we have 
		\begin{align*}
			|G(x)-Q^nF_{g,n}(x)|&=|Q^nF_g(x)-Q^nF_{g,n}(x)|\\
			&\leq |Q^n[\chi_n F_g](x)-Q^n[\chi_n F_{g,n}](x)|+2\|F_g\|_{L^\infty}|Q^n[1-\chi_n](x)|\\
			&=2\|F_g\|_{L^\infty} P^n\chi_{K_n^c}(x) \\
			&\leq 2\|F_g\|_{L^\infty}(1/n+c\gamma^n),
		\end{align*}
		where the last inequality holds because $P$ is uniformly geometrically ergodic by Theorem \ref{Uniformly geometrically ergodic}. Therefore, $G$ is the uniform limit of continuous functions $Q^n F_{g,n}$ on $\mathcal{O}$. Since $\mathcal{O}\subset \mathbb{T}^3$ is dense, there exists a continuous function $G'$ on $\mathbb{T}^3$ such that $G'=G$ holds $\mu$-a.e.. This completes the proof.   
	\end{proof}
    We are now ready to prove the main theorem of this subsection on the consequence of $3\lambda_1=\lambda_{\Sigma}$.
	\begin{proof}[\textbf{Proof of Theorem \ref{degeneracy}}]
		Using Theorem 6.8 in \cite{B}, it suffices to show that there exists a version of the factorization of $\nu$ respect to $\mu$ such that $x\mapsto \nu_x$ is continuous on $\mathbb{T}^3$. By the assumption, this result is a direct consequence of Theorem \ref{Uniformly geometrically ergodic} and Lemma \ref{RCPD}.
	\end{proof}
	\subsection{Proof of Theorem \ref{Positivity of the top Lyapunov exponent}}\label{subsection 4.2}
	\quad Let us show how one can rule out Theorem \ref{degeneracy} $(a)$ and $(b)$. 
	\begin{lem}\label{ruling out a}
		Let $\{\Phi_{\underline{\tau}}^m\}$ be as above. Then Theorem \ref{degeneracy} $(a)$ does not hold.
	\end{lem}
    \begin{proof} 
    	Let $x=(x_1,x_2,x_3)\in\mathbb{T}^{3}$ such that $f_1(x_3)\neq 0$ and $f_1'(x_3)\neq 0$. Denote $\varphi_t(x):=\varphi_t^{(1)}(x)=(x_1+t f_1(x_3),x_2+tf_1'(x_3),x_3)$, a simple computation shows 
    	\[ D\varphi_t(x)=\left(
    	\begin{array}{ccc}
    		1 & 0 & tf_1'(x_3)\\
    		0  & 1 & tf_1''(x_3)\\
    		0 & 0 & 1
    	\end{array} \right).	\]
    	Now suppose Theorem \ref{degeneracy} $(a)$ holds, then there exists a Riemannian structure $\{g_x:x\in \mathbb{T}^3\}$ on $\mathbb{T}^3$ such that (\ref{conformal}) holds for all $m$, $t\in\mathbb{R}_{\geq 0}^{3m}$, $x\in \mathbb{T}^3$ and $u,v\in T_x\mathbb{T}^3$. In particular,
    	\begin{equation*}
    		g_{\varphi_t(x)}(D\varphi_t(x)u,D\varphi_t(x)v)=g_x(u,v).
    	\end{equation*}
        \begin{itemize}
        	\item[(i)] If $f_1(x_3)$ and $f_1'(x_3)$ are rationally linearly dependent, there exist distinct $t_1, t_2, t_3\in \mathbb{R}_{\geq 0}$ such that $\varphi_{t_1}(x)=\varphi_{t_2}(x)=\varphi_{t_3}(x)=x$. Choose $v=u$ and write $u=(u_1,u_2,u_3)^{T}$. Then for each $k\in1, 2, 3$, we have
        	\begin{align*}
        		g_x(u,u)
        		&=g_{\varphi_{t_k}(x)}(D\varphi_{t_k}(x) u, D\varphi_{t_k}(x) u)\\
        		&= g_x(D\varphi_{t_k}(x) u, D\varphi_{t_k}(x) u)\\
        		&= g_x\left(  \left(
        		\begin{array}{ccc}
        			u_1+t_kf_1'(x_3)u_3\\
        			u_2+t_kf_1''(x_3)u_3\\
        			u_3
        		\end{array} \right)	, 
        		\left(\begin{array}{ccc}
        			u_1+t_kf_1'(x_3)u_3\\
        			u_2+t_kf_1''(x_3)u_3\\
        			u_3
        		\end{array}\right) \right)\\
        		&=t_k^2 g_x\left(  \left(
        		\begin{array}{ccc}
        			f_1'(x_3)u_3\\
        			f_1''(x_3)u_3\\
        			0
        		\end{array} \right)	, 
        		\left(\begin{array}{cc}
        			f_1'(x_3)u_3\\
        			f_1''(x_3)u_3\\
        			0
        		\end{array}\right) \right)+t_k C_1+C_2,
        	\end{align*} 
        	where $C_1$ and $C_2$ are independent of $t_k$. By assumption $f_1'(x_3)\neq 0$, we can choose $u$ such that $f_1'(x_3)u_3\neq 0$. Consequently, the last equation is a nontrivial quadratic equation in $t_k$ with at most two roots, which contradicts the existence of three distinct values $t_1, t_2, t_3$ satisfying the equation.
        	\item[(ii)]If $f_1(x_3)$ and $f_1'(x_3)$ are rationally linearly independent, without loss of generality, we assume $f_1(x_3)>0$. Let $\alpha:=f_1'(x_3)/f_1(x_3)$ be an irrational number. Take a sufficiently small neighborhood $V$ of $x$ such that $\frac{1}{2}g_x\leq g_{x'} \leq 2g_x$ for any $x'\in V$. For any $k\in\mathbb{N}$, let $t_k>0$ such that $t_kf_1(x_3)=2k\pi$, then $t_kf_1'(x_3)=2\alpha k\pi$. Since $ \alpha $ is an irrational number, the sequence of fractional parts $\{ k\alpha \}$ as $k$ varies over natural numbers is dense in the interval $(0, 1)$. Therefore, there exists a sequence $\{k_i\}$, $k_i\to\infty$ such that $\{k_i \alpha\} \to 0$ as $i \to \infty $. Then there exists $N(V)\in\mathbb{N}$ such that $ \varphi_{t_{k_i}}(x) \in V$ for all $i>N(V)$. Choose $v=u$ and write $u=(u_1,u_2,u_3)^T$, then
        	\begin{align*}
        		&2g_x(u,u)\\
        		=&2g_{\varphi_{t_{k_i}}(x)}(D\varphi_{t_{k_i}}(x) u, D\varphi_{t_{k_i}}(x) u)\\
        		\geq& g_x(D\varphi_{t_{k_i}}(x) u, D\varphi_{t_{k_i}}(x) u)\\
        		=&t_{k_i}^2 g_x\left(  \left(
        		\begin{array}{ccc}
        			f_1'(x_3)u_3\\
        			f_1''(x_3)u_3\\
        			0
        		\end{array} \right)	, 
        		\left(\begin{array}{cc}
        			f_1'(x_3)u_3\\
        			f_1''(x_3)u_3\\
        			0
        		\end{array}\right) \right)+t_{k_i} C_1+C_2,\  \text{ for all } i>N(V)
        	\end{align*} 
        	where $C_1$ and $C_2$ are independent of $t_{k_i}$. By assumption $f_1'(x_3)\neq 0$, we can choose $u$ such that $f_1'(x_3)u_3\neq 0$. The coefficient of the quadratic term $t_{k_i}^2$ is positive, and this leads to a contraction as $i$ goes to infinity.
        \end{itemize} 
    \end{proof}
    \begin{lem}\label{ruling out b}
    	Let $\{\Phi_{\underline{\tau}}^m\}$ be as above. Then Theorem \ref{degeneracy} $(b)$ does not hold.
    \end{lem}
    In order to prove this lemma, we first present some auxiliary results. Let $\{\Psi_{\underline{\tau}}^m\}$, $Q$, $M$ and $\mathcal{F}$, etc., be as introduced in Section \ref{subsection 3.1}. We say that $Q$ is \textbf{Strong Feller} if $Qg$ is continuous for any bounded, measurable function $g$ on $M$. Next, we present a version of Proposition 5.4 from \cite{AMM2} that is more convenient for our purposes in this paper, and the main idea of the proof follows the same lines as in \cite{AMM2}.
    \begin{thm}\label{Strong Feller}
    	Let $\{\Psi_{\underline{\tau}}^m\}$ and $Q$ be as above. Assume $\text{Lie}_{x_0}(\mathcal{F})=T_{x_0} M$ for some $x_0\in M$. Then there exist some $m$ and open neighborhood $U$ of $x_0$ such that the transition kernel $Q^{m}$ is strong Feller on $U$.
    \end{thm}
    In preparation for the proof of this theorem, we will first present a few lemmas. For any $m\in\mathbb{N}$, $x\in M$, consider $\Psi_x^m:\mathbb{R}_+^{3m}\to M$, $\Psi_x^m(t):=\Psi_t^m(x)$. Let $J{\Psi_x^m}(t)$ represent the Jacobian determinant of $\Psi_x^m$, then
    \begin{equation*}
    	|J{\Psi_x^m}(t)|=(\det(D\Psi_x^m(t)D\Psi_x^m(t)^T))^{1/2},
    \end{equation*}
    where $M^T$ is the transpose of matrix $M$.
    \begin{lem}\label{submersion for almost all time}
    	Assume $\text{Lie}_{x_0}(\mathcal{F})=T_{x_0} M$ for some $x_0\in M$, then there exist some $m$ and open neighborhood $U$ of $x_0$ such that for every $x\in U$, the map $t\mapsto \Psi^{m}_{x}(t)$ is a submersion for almost every $t\in\mathbb{R}_+^{3m}$. Moreover, for any bounded, measurable function $f$ on $M$, writting $t=(t_{3m},\cdots,t_1)$ and $dt=dt_{3m}\cdots dt_1$,
    	\begin{equation*}
    		Q^mf(x)=\int_{M}f(y)\left(\int_{\{t\in\mathbb{R}_{+}^{3m}:\Psi_{x}^{m}(t)=y\}}\frac{e^{-\sum_{i=1}^{3m}t_i/h}}{h^{3m}|J{\Psi_x^m}(t)|}\mathcal{H}^{3m-d}(dt)\right)\text{Leb}_M(dy),\ \ \text{for any} \ x\in U,
    	\end{equation*}
    	where $\mathcal{H}^{3m-d}$ is $(3m-d)$-dimensional Hausdorff measure on $\mathbb{R}_+^{3m}$.
    \end{lem}
    \begin{proof}
    	It follows from Lemma \ref{surjective} that there exist $m$ and $t_0\in \mathbb{R}_{+}^{3m}$ such that $\Psi_{x_0}^{m}: \mathbb{R}_{+}^{3m}\to M$ is a submersion at $t$. Define $g:M\times \mathbb{R}_{+}^{3m}\to\mathbb{R}$ by 
    	\begin{equation*}
    		g(x,t):=|J{\Psi_x^{m}}(t)|.
    	\end{equation*}
    	Then $g(x_0,t_0)>0$. With the function $g_{t_0}(\cdot) = g(\cdot, t_0): M\to \mathbb{R}$ being continuous, it naturally follows that the set $U= g_{t_0}^{-1}((0, \infty))$ defines an open neighborhood around $x_0$ in $M$. Now since the vector fields (and hence their flows) are analytic and analyticity is preserved under addition, multiplication, composition and differentiation, $g_{x}(\cdot):=g(x,\cdot):\mathbb{R}_{+}^{3m}\to\mathbb{R}$ is analytic for all $x\in M$. Consequently, based on a result that is widely acknowledged, with instances provided in \cite{M} and \cite[Lemma 5.22]{KP}, it follows that the set $\mathcal{N}(g_{x}):=\{t\in\mathbb{R}_{+}^{3m}:g_{x}(t)=0\}$ is a zero Lebesgue measure set for every $x\in U$. Therefore, for every $x\in U$, the map $t\mapsto \Psi^{m}_{x}(t)$ is a submersion for almost every $t\in \mathbb{R}_{+}^{3m}$. The first part of the lemma now follows as above. 
    	
    	Next, we proceed to prove the `moreover' part. For any bounded, measurable function $f$ on $M$, $x\in U$, writting $t=(t_{1},\cdots,t_{3m})$ and $dt=dt_1\cdots dt_{3m}$,
    	\begin{align*}
    		Q^mf(x)=&\int_{\mathbb{R}_{+}^{3m}}\frac{f(\Psi_{t}^{m}(x))}{h^{3m}e^{-\sum_{i=1}^{3m}t_i/h}}dt\\
    		=&\int_{M}\left(\int_{\{t\in\mathbb{R}_{+}^{3m}:\Psi_{x}^m(t)=y\}}\frac{f(\Psi_{t}^{m}(x))e^{-\sum_{i=1}^{3m}t_i/h}}{h^{3m}|J{\Psi_x^{m}}(t)|}\mathcal{H}^{3m-d}(dt)\right)\text{Leb}_M(dy)\\
    		=&\int_{M}f(y)\left(\int_{\{t\in\mathbb{R}_{+}^{3m}:\Psi_{x}^{m}(t)=y\}}\frac{e^{-\sum_{i=1}^{3m}t_i/h}}{h^{3m}|J{\Psi_x^{m}}(t)|}\mathcal{H}^{3m-d}(dt)\right)\text{Leb}_M(dy), 
    	\end{align*}
    	where the second equality follows from the coarea formula, as can be found, for example, in \cite{FH}. One caveat in our application of the coarea formula is that our assumption says $J{\Psi_x^{m}}(t)$ is nonzero only almost-surely. This is not an issue however since $\{y:\Psi_{t}^{m}(x)=y\ \text{and}\ J{\Psi_x^{m}}(t)=0\}$ has measure zero in $M$ by Sard's lemma.
    \end{proof}
    Considering any subset $K$ from $\mathbb{R}_{+}^{3m}$, any point $x$ in $M$, we define
    \begin{equation*}
    	A_{K}(x):=\{t\in K: |J{\Psi_x^m}(t)|=0\}.
    \end{equation*}
    \begin{lem}\label{positivity of det J}
    	Assume $K$ is a compact set in $\mathbb{R}_{+}^{3m}$, $U$ is an open subset of $M$, and the matrix product $D\Psi_x^m(t)D\Psi_x^m(t)^T$ is almost surely invertible for every $x$ in $U$. Then for any $x_*\in U$ and $\varepsilon>0$ there exists an open set $V\subset \mathbb{R}_{+}^{3m}$ and $\delta>0$ such that 
    	\begin{equation*}
    		\mathbb{P}(V)<\varepsilon,\ \ B_\delta(x_*)\subset U \ \ \text{and}\ \ A_{K}(x)\subset V\ \text{for any} \ x\in B_\delta(x_*).
    	\end{equation*}
    	Moreover, there exists an open neighborhood $W$ of $K\cap V^c$ such that 
    	\begin{equation*}
    		\inf\{	|J{\Psi_x^m}(t)|:x\in B_\delta(x_*), t\in W\}>0.
    	\end{equation*}
    \end{lem}
    \begin{proof}
    	Fix any $x_*\in U$ and any $\varepsilon>0$.  Since the matrix product $D\Psi_x^m(t)D\Psi_x^m(t)^T$ is almost surely invertible for every $x$ in $U$, $A_{K}(x)$ has Lebesgue measure zero for all $x\in U$. In particular, since $\mathbb{P}$ is absolutely continuous with respect to Lebesgue measure, there exists an open neighborhood $V$ of $A_{K}(x_*)$ such that $\mathbb{P}(V)<\varepsilon$. We will next prove our conclusion using the method of contradiction. Suppose there is a $\delta>0$ such that $B_\delta(x_*)\subset U$ and there exists a sequence $\{x_n\}\subset B_\delta(x_*)$ converging to $x_*$ such that $A_{K}(x_n)$ is not contained in $V$, that is, for each $x_n$ there is a $t_n$ in $K\cap V^c$ satisfying $|J{\Phi_{x_n}^m}(t_n)|=0$. Then since $K\cap V^c$ is compact, there is a subsequence $\{t_{n_k}\}$ which converges to some $t_*\in K\cap V^c$. It follows from the continuous of $(x,t)\mapsto |J{\Psi_x^m}(t)|$ that
    	\begin{equation*}
    		0=|J{\Psi_{x_{n_k}}^m}(t_{n_k})|=|J{\Psi_{x_*}^m}(t_*)|.
    	\end{equation*} 
    	But this implies that $t_*$ is in $A_{K}(x_*)$, which contradicts the fact that $A_{K}(x_*)\cap (K\cap V^c)=\emptyset$.     	
    	By the preceding argument and the continuous of $(x,t)\mapsto |J{\Psi_x^m}(t)|$, we have
    	\begin{equation*}
    		\inf\{	|J{\Psi_x^m}(t)|:x\in B_{\delta/2}(x_*), t\in K\cap V^c\}\geq 2c, 
    	\end{equation*}
    	for some $c>0$. Set $g_x(t):=|J{\Psi_x^m}(t)|$ and let $K'$ be the closure in $\mathbb{R}_+^{mn}$ of 
    	\begin{equation*}
    		\bigcup_{x\in B_{\delta/2}(x_*)}g_x^{-1}((0,c)).
    	\end{equation*}
    	Since $K\cap V^c$ and $K'$ are closed and disjoint, they can be seperated by disjoint open sets. By continuity and construction, there exists an open neighborhood $W$ of $K\cap V^c$ satisfies
    	\begin{equation*}
    		\inf\{	|J{\Psi_x^m}(t)|:x\in B_{\delta/2}(x_*), t\in W\}\geq c >0.
    	\end{equation*}
    \end{proof}
    We are now ready to prove the strong Feller property in a general setting.
    \begin{proof}[\textbf{Proof of Theorem \ref{Strong Feller}}]
    	Since $\text{Lie}_{x_0}(\mathcal{F})=T_{x_0}M$ for some $x_0\in M$, by the first part of Lemma \ref{submersion for almost all time}, there exist some $m$ and open neighborhood $U$ of $x_0$ such that for every $x\in U$, $D\Psi_x^{m}(t)D\Psi_x^{m}(t)^T$ is invertible for almost every $t\in\mathbb{R}_{+}^{3m}$. Consider a bounded, measurable function $f$ on $M$. Without loss of generality, we assume that $f$ is not identically zero. Denote $ \|f\|_{L^\infty} =c_1/6$, where $c_1>0$ is a constant. Fix any $x_*\in U$ and any $\varepsilon>0$. Since $\mathbb{P}$ is absolutely continuous with respect to Lebesgue measure, there exists a compact subset $K$ of $\mathbb{R}_{+}^{3m}$ such that $\mathbb{P}(K^c)<\varepsilon/{c_1}$. By Lemma \ref{positivity of det J}, there exists an open set $V\subset \mathbb{R}_{+}^{3m}$ and $\delta_1>0$ such that $\mathbb{P}(V)<\varepsilon/{c_1}$, $B_{\delta_1}(x_*)\subset U$ and an open neighborhood $W$ of $K\cap V^c$ such that
    	\begin{equation}\label{44}
    		\inf\{|J{\Psi_x^{m}}(t)|:x\in B_{\delta_1}(x_*), t\in W\}=:c>0. 
    	\end{equation}
    	Since $\mathbb{R}_{+}^{3m}=K^c\cup(K\cap V)\cup(K\cap V^c)$, for any $x\in U$, we have 
    	\begin{align}\label{45}
    		Q^{m}f(x)-Q^{m}f(x_*)=&\mathbb{E}\left([f(\Psi_{\underline{\tau}}^{m}(x))-f(\Psi_{\underline{\tau}}^{m}(x_*))]\chi_{K^c}\right)+\mathbb{E}\left([f(\Psi_{\underline{\tau}}^{m}(x))-f(\Psi_{\underline{\tau}}^{m}(x_*))]\chi_{K\cap V}\right)\nonumber\\
    		&~~~~~~~+\mathbb{E}\left([f(\Psi_{\underline{\tau}}^{m}(x))-f(\Psi_{\underline{\tau}}^{m}(x_*))]\chi_{K\cap V^c}\right).
    	\end{align}
    	Due to the selection of $K$, it follows that
    	\begin{equation}\label{46}
    		|\mathbb{E}\left([f(\Psi_{\underline{\tau}}^{m}(x))-f(\Psi_{\underline{\tau}}^{m}(x_*))]\chi_{K^c}\right)|\leq 2\|f\|_{L^\infty}\mathbb{P}(K^c)<\varepsilon/3.
    	\end{equation}
    	Similarly, the choice of $V$ ensures that
    	\begin{equation}\label{47}
    		|\mathbb{E}\left([f(\Psi_{\underline{\tau}}^{m}(x))-f(\Psi_{\underline{\tau}}^{m}(x_*))]\chi_{K\cap V}\right)|\leq 2\|f\|_{L^\infty}\mathbb{P}(V)<\varepsilon/3.
    	\end{equation}
    	To handle the term involving $\chi_{K\cap V^c}$, let us denote $L^1 := L^1(\text{Leb}_M)$. We then select a compactly supported, continuous function $\tilde{f}$ on $M$ such that $\|\tilde{f} - f\|_{L^1} \leq \sqrt{c}\varepsilon/12$. Note that this can always be done since compactly supported, continuous functions are dense in $L^1$. By adding and subtracting $\tilde{f}\circ\Psi_{\underline{\tau}}^{m}$ appropriately,
    	\begin{align}\label{48}
    		&\mathbb{E}\left([f(\Psi_{\underline{\tau}}^{m}(x))-f(\Psi_{\underline{\tau}}^{m}(x_*))]\chi_{K\cap V^c}\right)\nonumber\\
    		=&\mathbb{E}\left([f(\Psi_{\underline{\tau}}^{m}(x))-\tilde{f}(\Psi_{\underline{\tau}}^{m}(x))]\chi_{K\cap V^c}\right)+\mathbb{E}\left([\tilde{f}(\Psi_{\underline{\tau}}^{m}(x))-\tilde{f}(\Psi_{\underline{\tau}}^{m}(x_*))]\chi_{K\cap V^c}\right)\nonumber\\
    		&+\mathbb{E}\left([\tilde{f}(\Psi_{\underline{\tau}}^{m}(x_*))-f(\Psi_{\underline{\tau}}^{m}(x_*))]\chi_{K\cap V^c}\right).
    	\end{align} 
    	Since $\tilde{f}\circ\Psi_{\underline{\tau}}^{m}$ is continuous, there exists $\delta_2>0$ such that for every $x\in B_{\delta_2}(x_*)$, 
    	\begin{equation}\label{49}
    		\left|\mathbb{E}\left([\tilde{f}(\Psi_{\underline{\tau}}^{m}(x))-\tilde{f}(\Psi_{\underline{\tau}}^{m}(x_*))]\chi_{K\cap V^c}\right)\right|<\frac{\varepsilon}{6}.
    	\end{equation}
    	Let us focus on the first and third terms on the right side of \eqref{48}, writting $t=(t_1,\cdots,t_{3m})$ and $dt=dt_1\cdots dt_{3m}$, it is established that for all $x\in B_{\delta_1}(x_*)$,
    	\begin{align}\label{410}
    		&\left|\mathbb{E}\left([f(\Psi_{\underline{\tau}}^{m}(x))-\tilde{f}(\Psi_{\underline{\tau}}^{m}(x))]\chi_{K\cap V^c}\right)\right|\nonumber\\
    		=&\left|\int_{W} \frac{[f(\Psi_{t}^{m}(x))-\tilde{f}(\Psi_{t}^{m}(x))]\chi_{K\cap V^c}e^{-\sum_{i=1}^{3m}t_i/h}}{h^{3m}}dt\right|\nonumber\\
    		=&\left|\int_{M}\left(\int_{W\cap \{t\in\mathbb{R}_{+}^{3m}:\Psi_{x}^{m}(t)=\tilde{x}\}}\frac{[f(\Psi_{t}^{m}(x))-\tilde{f}(\Psi_{t}^{m}(x))]\chi_{K\cap V^c}e^{-\sum_{i=1}^{3m}t_i/h}}{h^{3m}|J{\Psi_x^{m}}(t)|}\mathcal{H}^{3m-d}(dt)\right)\text{Leb}_M(d\tilde{x})\right|\nonumber\\
    		=&\left|\int_{M}\left(f(\tilde{x})-\tilde{f}(\tilde{x})\right)\left(\int_{W\cap \{t\in\mathbb{R}_{+}^{3m}:\Psi_{x}^{m}(t)=\tilde{x}\}} \frac{\chi_{K\cap V^c}e^{-\sum_{i=1}^{3m}t_i/h}}{h^{3m}|J{\Psi_x^{m}}(t)|}\mathcal{H}^{3m-d}(dt)\right)\text{Leb}_M(d\tilde{x})\right|\nonumber\\
    		\leq&\frac{1}{\sqrt{c}}\|f-\tilde{f}\|_{L^1}<\frac{\varepsilon}{12},
    	\end{align}
    	where the first equality is due to $K\cap V^c$ being a subset of $W$, the second is confirmed by Lemma \ref{submersion for almost all time}, and the first inequality is established through \eqref{44}. Subsequently, we set $\delta$ to be the minimum of $\delta_1$ and $\delta_2$, by substituting \eqref{49} and \eqref{410} into \eqref{48} and considering all $x$ in the ball $B_\delta(x_*)$, we obtain
    	\begin{equation}\label{411}
    		\left|\mathbb{E}\left([f(\Psi_{\underline{\tau}}^{m}(x))-f(\Psi_{\tau}^{m}(x_*))]\chi_{K\cap V^c}\right)\right|<\frac{\varepsilon}{3}.
    	\end{equation}
    	Ultimately, substituting \eqref{46}, \eqref{47} and \eqref{411} into \eqref{45} yields,
    	\begin{equation*}
    		\left|Q^{m}f(x)-Q^{m}f(x_*)\right|<\varepsilon,
    	\end{equation*}
    	for every $x\in B_\delta(x_*)$. Hence, $Q^mf$ is continuous at every point $x_*$ in $U$, establishing that $Q^{m}$ possesses the strong Feller property on $U$.
    \end{proof}
    To apply Theorem \ref{Strong Feller}, we now prove the Lie bracket condition for the lifted Markov chain $\{\hat{\Phi}_{\underline{\tau}}^m\}$.
    \begin{lem}\label{Lie bracket of lifted chain}
    	Let $\hat{\mathcal{X}}$ be as above. Then, there exists some ${\hat{x}}\in T\mathbb{T}^3$ such that $\text{Lie}_{\hat{x}}(\hat{\mathcal{X}})=T_{\hat{x}}T\mathbb{T}^3$.
    \end{lem}
    \begin{proof}
    	For any $(x,u)=(x_1, x_2, x_3, u_1, u_2, u_3)\in T\mathbb{T}^3$,
    	\begin{equation*}
    		\hat{X}_1(x,u)=\begin{pmatrix}
    			f_1(x_3)\\
    			f_1'(x_3)\\
    			0\\
    			u_3f_1'(x_3)\\
    			u_3f_1''(x_3)\\
    			0
    		\end{pmatrix},\quad
    		\hat{X}_2(x,u)=\begin{pmatrix}
    			0\\
    			f_2(x_1)\\
    			f_2'(x_1)\\
    			0\\
    			u_1 f_2'(x_1)\\
    			u_1 f_2''(x_1)
    		\end{pmatrix},\quad
    		\hat{X}_3(x,u)=\begin{pmatrix}
    			f_3'(x_2)\\
    			0\\
    			f_3(x_2)\\
    			u_2f_3''(x_2)\\
    			0\\
    			u_2f_3'(x_2)
    		\end{pmatrix}.
    	\end{equation*}
    	And their Lie bracket are
    	\begin{equation*}
    		[\hat{X}_1,\hat{X}_2](x,u)=\begin{pmatrix}
    			-f_1'(x_3)f_2'(x_1)\\
    			(f_1(x_3)-f_1''(x_3))f_2'(x_1)\\
    			f_1(x_3)f_2''(x_1)\\
    			-(u_3f_1''(x_3)f_2'(x_1)+u_1f_1'(x_3)f_2''(x_1))\\
    			u_1(f_1(x_3)-f_1''(x_3))f_2''(x_1)+u_3(f_1'(x_3)-f_1'''(x_3))f_2'(x_1)\\
    			u_1 f_1(x_3)f_2'''(x_1)+u_3f_1'(x_3)f_2''(x_1)
    		\end{pmatrix},
    	\end{equation*}
    	\begin{equation*}
    		[\hat{X}_1,\hat{X}_3](x,u)=\begin{pmatrix}
    			(f_3''(x_2)-f_3(x_2))f_1'(x_3)\\
    			-f_1''(x_3)f_3(x_2)\\
    			f_1'(x_3)f_3'(x_2)\\
    			u_2(f_3'''(x_2)-f_3'(x_2))f_1'(x_3)+u_3(f_3''(x_2)-f_3(x_2))f_1''(x_3)\\
    			-(u_3f_1'''(x_3)f_3(x_2)+u_2f_1''(x_3)f_3'(x_2))\\
    			u_2f_1'(x_3)f_3''(x_2)+u_3f_1''(x_3)f_3'(x_2)
    		\end{pmatrix},
    	\end{equation*}
    	and
    	\begin{equation*}
    		[\hat{X}_2,\hat{X}_3](x,u)=\begin{pmatrix}
    			f_2(x_1)f_3''(x_2)\\
    			-f_2'(x_1)f_3'(x_2)\\
    			(f_2(x_1)-f_2''(x_1))f_3'(x_2)\\
    			u_2f_2(x_1)f_3'''(x_2)+u_1f_2'(x_1)f_3''(x_2)\\
    			-(u_1f_2''(x_1)f_3'(x_2)+u_2f_2'(x_1)f_3''(x_2))\\
    			u_1(f_2'(x_1)-f_2'''(x_1))f_3'(x_2)+u_2(f_2(x_1)-f_2''(x_1))f_3''(x_2)
    		\end{pmatrix}.
    	\end{equation*}
    	Define matrix 
    	\begin{equation*}
    		M(x,u):=\left(\hat{X}_1(x,u),\hat{X}_2(x,u), \hat{X}_3(x,u), [\hat{X}_1,\hat{X}_2](x,u), [\hat{X}_1,\hat{X}_3](x,u), [\hat{X}_2,\hat{X}_3](x,u)\right).
    	\end{equation*}
    	Let $z_{i-1}\in C_{f_{i}'}$, to be determined later. Then, for $z=(z_1,z_2,z_3)$ and arbitrary $u=(u_1,u_2,u_3)$, 
    	\[
    	\resizebox{\textwidth}{!}{$
    		M(z,u)=\begin{pmatrix}
    			f_1(z_3)& 0& 0& 0&0&f_2(z_1) f_3''(z_2)\\
    			0&f_2(z_1)&0&0&-f_1''(z_3) f_3(z_2)&0 \\
    			0&0&f_3(z_2)&f_1(z_3) f_2''(z_1)&0&0 \\
    			0&0&u_2 f_3''(z_2)&0&u_3 f_1''(z_3) (f_3''(z_2)-f_3(z_2))&u_2 f_2(z_1) f_3'''(z_2) \\
    			u_3 f_1''(z_3)&0&0&u_1 f_2''(z_1)(f_1(z_3)-f_1''(z_3))&-u_3 f_1'''(z_3) f_3(z_2)&0 \\
    			0&u_1 f_2''(z_1)&0&u_1 f_1(z_3) f_2'''(z_1)&0&u_2 f_3''(z_2)(f_2(z_1)-f_2''(z_1)).
    		\end{pmatrix}
    		$}
    	\]
    	By a direct calculation, we obtain that 
    	\begin{equation}\label{det of M}
    		\det (M(z,u))=f_1(z_3) f_2(z_1) f_3(z_2) \cdot S(z,u),
    	\end{equation}
    	where $S=S_1-S_2+S_3$ and
    	\begin{align*}
    		S_1(z,u) &= u_2^2 u_3 f_1(z_3) f_1'''(z_3) f_2''(z_1) (f_2(z_1)-f_2''(z_1)) (f_3''(z_2))^2 \\
    		&\quad -u_1 u_2 u_3 (f_1''(z_3) f_2''(z_1) f_3''(z_2))^2, \\
    		S_2(z,u) &= u_1 u_2 u_3 f_1''(z_3) (f_1(z_3)-f_1''(z_3))f_2''(z_1)(f_2(z_1)-f_2''(z_1)) f_3''(z_2) (f_3''(z_2)-f_3(z_2)) \\
    		&\quad -u_1^2 u_2 f_1''(z_3) (f_1(z_3)-f_1''(z_3)) (f_2''(z_1))^2 f_3(z_2) f_3'''(z_2), \\
    		S_3(z,u) &= -u_1 u_3^2 (f_1''(z_3))^2 f_2(z_1) f_2'''(z_1) f_3''(z_2) (f_3''(z_2)-f_3(z_2)) \\
    		&\quad + u_1 u_2 u_3 f_1(z_3) f_1'''(z_3) f_2(z_1) f_2'''(z_1) f_3(z_2) f_3'''(z_2).
    	\end{align*}
    	We claim that there exist $(z,u)$ such that $\det (M(z,u)) \neq 0$. Notice that by (H1), this is equivalent to $S(z,u) \neq 0$. We divide the analysis into the following cases. 
    	\begin{itemize}
    		\item [(1)] If there exists $i \in \{1,2,3\}$ such that $f_{i}'''(z_{i-1})$, we conventionally set $z_0=z_3$. We split into the following three cases.
    		\begin{itemize}
    			\item [(i)] If $f_1'''(z_3) \neq 0$, set $u_1=0$, $u_2u_3 \neq 0$. Since $f_2 \in C^\omega(\mathbb{S}^1, \mathbb{R})$ is non-constant, let $x_1$ be a point where $f_2$ achieves its maximum value, and $y_1$ be a point where $f_2$ achieves its minimum value. By assumption (H1), $f_2''(x_1) < 0$ and $f_2''(y_1) > 0$. Moreover, either $f_2(x_1) > 0$ or $f_2(y_1) < 0$. Consequently, either $f_2(x_1) - f_2''(x_1) > 0$ or $f_2(y_1) - f_2''(y_1) < 0$. Then, we can choose $z_1\in C_{f_2'}$ such that $f_2(z_1)-f_2''(z_1)\neq 0$. Then, combining with (H1) and previous estimates, 
    			\[
    			S(z,u)= u_2^2 u_3 f_1(z_3) f_1'''(z_3) f_2''(z_1) (f_2(z_1)-f_2''(z_1)) (f_3''(z_2))^2 \neq 0.
    			\]
    			\item [(ii)] If $f_2'''(z_1)\neq 0$, set $u_2=0$, $u_1u_3 \neq 0$. Similarly, we can choose $z_2\in C_{f_3'}$ such that $f_3(z_2)-f_3''(z_2) \neq 0$. Then, combining with (H1) and previous estimates, 
    			\[
    			S(z,u)= -u_1 u_3^2 (f_1''(z_3))^2 f_2(z_1) f_2'''(z_1) f_3''(z_2) (f_3''(z_2)-f_3(z_2))\neq 0.
    			\]
    			\item [(iii)] If $f_3'''(z_2)\neq 0$, set $u_3=0$, $u_1u_2 \neq 0$. Similarly, we can choose $z_3\in C_{f_1'}$ such that $f_1(z_3)-f_1''(z_3)\neq 0$. Then, combining with (H1) and previous estimates,
    			\[
    			S(z,u)= u_1^2 u_2 f_1''(z_3) (f_1(z_3)-f_1''(z_3)) (f_2''(z_1))^2 f_3(z_2) f_3'''(z_2)\neq 0.
    			\]
    		\end{itemize}
    		\item[(2)] If for all $i$, there holds $f_{i}'''(z_{i-1}) = 0$, then 
    		\begin{align}\label{eq equiv 0}
    			S(z,u)=&u_1 u_2 u_3 f_1''(z_3) f_2''(z_1) f_3''(z_2)(f_1(z_3)-f_1''(z_3))(f_2(z_1)-f_2''(z_1)) (f_3(z_2)-f_3''(z_2))\nonumber\\
    			&-u_1 u_2 u_3 (f_1''(z_3) f_2''(z_1) f_3''(z_2))^2.
    		\end{align}
    		Define $ x_1$, $x_2$, $x_3 $ as the maximum values and $y_1$, $y_2$, $y_3 $ as the minimum values of $f_2$, $f_3$, $f_1$, respectively. We now analyze the following cases.
    		\begin{itemize}
    			\item [(i)] If $(f_1(x_3)-f_1''(x_3))(f_2(x_1)-f_2''(x_1)) (f_3(x_2)-f_3''(x_2)) \geq 0$, let $z_i=x_i$ and $u_1u_2u_3 \neq 0$. Then, combining with $f_i''(x_{i-1}) <0$ and \eqref{eq equiv 0}, $S(z,u) \neq 0$. 
    			\item[(ii)] If among the three elements in the set $\{f_i(x_{i-1}) - f_i''(x_{i-1})\}$, exactly two are positive and one is negative, we may assume without loss of generality that $ f_j(x_{j-1}) - f_j''(x_{j-1}) > 0$ for $j = 1,2$, and $f_3(x_2) - f_3''(x_2) < 0$. Let $z_1=x_1, z_2=y_2, z_3=x_3$ and $u_1u_2u_3 \neq 0$. Then, $f_1''(x_3) f_2''(x_1) f_3''(y_2) >0$ and $f_3(y_2)-f_3''(y_2) <0$ since $f_3(x_2) - f_3''(x_2) < 0$. Thus, by \eqref{eq equiv 0}, $S(z,u) \neq 0$.
    			\item[(iii)] If $f_i(x_{i-1}) - f_i''(x_{i-1})$ are all negative for all $i = 1,2,3$, let $z_i=y_i$ and $u_1u_2u_3 \neq 0$. By \eqref{eq equiv 0}, $S(z,u) \neq 0$.
    		\end{itemize}
    	\end{itemize}
    	From the above analysis, the claim holds.
    \end{proof}
    With these auxiliary results in place, we now proceed to rule out Theorem \ref{degeneracy} $(b)$.
    \begin{proof}[\textbf{Proof of Lemma \ref{ruling out b}}]
    	By Lemma \ref{Lie bracket of lifted chain} and Theorem \ref{Strong Feller}, there exists an $m_0$ and a neighborhood $\hat{U}$ of $\hat{x}$ such that $\hat{P}^{m_0}|_{\hat{U}}$ is strong Feller. Assume Theorem \ref{degeneracy} $(b)$ holds, there exist proper linear subspaces $E_x^1,\cdots,E_x^p$ of $T_x \mathbb{T}^3$ for all $x\in \mathbb{T}^3$ such that 
    	\begin{equation}\label{permutation}
    		D_x\Phi_t^m(E_x^i)=E_{\Phi_t^m(x)}^{\sigma(i)}, \ \ 1\leq i\leq p
    	\end{equation}
    	for all $m$, $t\in \mathbb{R}_{\geq 0}^{3m}$ and some permutation $\sigma$. For every $x\in\mathbb{T}^3$, define $f:\hat{U}\to \mathbb{R}$ by
    	\[f(x,v):=\chi_{E_x}(v)
    	\]
    	where $\chi_{E_x}$ is the indicator function of $E_x$ and $E_x=\bigcup_{i=1}^p E_x^i$. Then $f$ is well-defined and is bounded, measurable but discontinuous since $E_x^i$ are proper subspaces. By (\ref{permutation}),
    	\begin{align*}
    		\hat{P}^{m_0}f(x,v)&=\mathbb{E}\big(f(\Phi_{\underline{\tau}}^{m_0}(x),D_x \Phi_{\underline{\tau}}^{m_0}(x)v)\big)\\
    		&=\mathbb{E}\big(\chi_{E_{\Phi_{\underline{\tau}}^{m_0}(x)}}(D_x \Phi_{\underline{\tau}}^{m_0}(x)v)\big)\\
    		&=\mathbb{E}\big(\chi_{D_x\Phi_{\underline{\tau}}^{m_0} (E_x)}(D_x \Phi_{\underline{\tau}}^{m_0}(x)v)\big)\\
    		&=f(x,v),
    	\end{align*}
    	is discontinuous which is a contradiction with $\hat{P}^{m_0}f(x,v)$ is continuous since $\hat{P}^{m_0}|_{\hat{U}}$ is strong Feller.
    \end{proof}
    We are now ready to prove the main theorem on the positivity of the top Lyapunov exponent.
    \begin{proof}[\textbf{Proof of Theorem \ref{Positivity of the top Lyapunov exponent}}]
    	Given assumption (H1), it follows from Lemma \ref{ruling out a} and Lemma \ref{ruling out b} that Theorem \ref{degeneracy} $(a)$ and $(b)$ do not hold. Therefore, $3\lambda_1\neq \lambda_\Sigma=0$, leading to the conclusion that $\lambda_1>0$.
    \end{proof}

\section{Ideal dynamo}\label{section 5}
\quad To establish that Corollary \ref{Ideal dynamo}, we invoke a variant of the `non-random' multiplicative ergodic theorem (Theorem III.1.2 in \cite{Ki}), which necessitates the uniqueness of the stationary probability measure for the associated projective Markov chain. In the following subsection, we establish the uniform geometric ergodicity of this projective Markov chain. This property ensures the uniqueness of its stationary measure and is crucial for proving the almost-sure quenched correlation decay, as stated in Theorem \ref{Exponential mixing}.

\subsection{Uniform geometric ergodicity for the projective Markov chain}\label{subsection 5.1}
\quad Recall the projective chain $\{\check{\Phi}_{\underline{\tau}}^m\}$ introduced in Section \ref{subsection 3.1}, for $(x, u) \in S\mathbb{T}^3$, 
\begin{equation*}	        \check{\Phi}_{\underline{\tau}}^m(x,u)=\left(\Phi_{\underline{\tau}}^m(x), \frac{D_x\Phi_{\underline{\tau}}^m u}{|D_x\Phi_{\underline{\tau}}^m u|}\right),
\end{equation*}
and the corresponding lifted vector fields $\check{X}_i(x, u)$ are defined by 
\[ 
\check{X_i}(x,u) = \begin{pmatrix}
	X_i(x) \\
	DX_i(x)u-\langle DX_i(x)u, u \rangle u
\end{pmatrix},
\]
for $i=1, 2, 3$. We will now demonstrate the irreducibility of the projective Markov chain.
\begin{lem}\label{irreducibility for projective chain}
	Let $\{\check{\Phi}_{\underline{\tau}}^m\}$ be the projective chain as above. Then, $\{\check{\Phi}_{\underline{\tau}}^m\}$ is exactly controllable. In particular, $\{\check{\Phi}_{\underline{\tau}}^m\}$ is topologically irreducible.
\end{lem}
\begin{proof}
	Under the assumptions that $f_i \in C^\omega(\mathbb{S}^1, \mathbb{R})$ and (H1) are satisfied, we make the following selections:  
	\begin{itemize}
		\item $y_i \in \mathbb{S}^1$, where $f_1'(y_3) = f_2'(y_1) = f_3'(y_2) = 0$, $f_2''(y_1) < 0$, and $f_3''(y_2) < 0$,
		\item $z_1 \in \mathbb{S}^1$, where $f_2'(z_1) = 0$, $f_2''(z_1) > 0$,
		\item $z_2 \in \mathbb{S}^1$, where $f_3(z_2) \neq 0$, $f_3'(z_2) \neq 0$, and $\frac{f_3(z_2)}{f_3'(z_2)}$ is irrational,
		\item $z_3 \in \mathbb{S}^1$, where $f_1'(z_3) \neq 0$, $f_1''(z_3) = 0$,
		\item $p_2 \in \mathbb{S}^1$, where $f_3'(p_2) < 0$,
		\item $q_1 \in \mathbb{S}^1$, where $f_2'(q_1) > 0$, $f_2''(q_1) = 0$.
	\end{itemize}
	The sets are defined as:
	\[
	D = \{(x_1, x_2, x_3, u_1, u_2, u_3) \in S\mathbb{T}^3 : u_3 = 0\},
	\]
	\[
	E = \{(x_1, x_2, x_3, u_1, u_2, u_3) \in S\mathbb{T}^3 : u_1 = 1, u_2 = 0, u_3 = 0\}, \quad
	F = \{(y_1, y_2, y_3, 1, 0, 0)\}.
	\]
	\begin{itemize}
		\item[(1)] First, we prove that $D$ is reachable from $S\mathbb{T}^3$. We consider the following two cases:
		\begin{itemize}
			\item[(I)] The first case is when $u_2 \neq 0$. If $u_3 = 0$, it is clear that $D$ is reachable from $S\mathbb{T}^3$. If $u_3 \neq 0$, the proof is divided into two cases:
			\begin{itemize}
				\item[(A)] When $u_2u_3 > 0$, we analyze the problem by considering the following three cases:
				\begin{itemize}
					\item[(i)] If $f_3'(x_2) < 0$, choose 
					\[
					\tau_1 = \tau_2 = 0, \quad \tau_3 = -\frac{u_3}{u_2 f_3'(x_2)}.
					\]
					It is evident that $D$ is reachable from $S\mathbb{T}^3$.
					
					\item[(ii)] If $f_3'(x_2) = 0$, by assumption (H1), we have $f_3(x_2) \neq 0$. Choose $i, j \in \mathbb{Z}$ such that
					\[
					\tau_3 = \frac{z_3 - x_3 + 2i\pi}{f_3(x_2)} \geq 0, \quad 
					\tau_4 = \frac{p_2 - x_2 + 2j\pi}{f_1'(z_3)} \geq 0.
					\]
					Set $\tau_1 = \tau_2 = \tau_5 = \tau_6 = 0$, now this situation is analogous to (1)(I)(A)(i).
					
					\item[(iii)] Now we consider the case that $f_3'(x_2) > 0$. If $f_3(x_2) \neq 0$, the argument is similar to case (1)(I)(A)(ii). If $f_3(x_2) = 0$, we consider three cases:
					\begin{enumerate}
						\item[(a)] If $u_1 = 0$, choose $i \in \mathbb{Z}$ such that
						\[
						\tau_3 = \frac{y_1 - x_1 + 2i\pi}{f_3'(x_2)} \geq 0, \quad 
						\tau_5 = \frac{p_2 - x_2 + 2j\pi}{f_2(y_1)} \geq 0.
						\]
						Set $\tau_1 = \tau_2 = \tau_4 = \tau_6 = 0$, now this situation is analogous to (1)(I)(A)(i).
						\item[(b)] If $u_1u_3 > 0$, choose $i\in\mathbb{Z}$ such that
						\[\tau_3=\frac{y_1-x_1+2i\pi}{f_3'(x_2)}\geq 0.\]
						Set 
						\[\tau_1=\tau_2=\tau_4=\tau_6=0,\quad 
						\tau_5=-\frac{u_3}{u_1f_2''(y_1)}.\]
						We obtain that $D$ is reachable from $S\mathbb{T}^3$.
						\item[(c)] If $u_1u_3 < 0$, choose $i\in\mathbb{Z}$ such that 
						\[\tau_3=\frac{z_1-x_1+2i\pi}{f_3'(x_2)}\geq 0.\]
						Set 
						\[\tau_1=\tau_2=\tau_4=\tau_6=0,\quad 
						\tau_5=-\frac{u_3}{u_1f_2''(z_1)}.\]
						We obtain that $D$ is reachable from $S\mathbb{T}^3$.
					\end{enumerate} 
				\end{itemize}
				\item[(B)] When $u_2u_3 < 0$, the argument follows similarly to case (1)(I)(A).
			\end{itemize}
			
			\item[(II)] The second case is when $u_2 = 0$. If $u_3 = 0$, it is clear that $D$ is reachable from $S\mathbb{T}^3$. If $u_3 \neq 0$, the proof is divided into two cases:
			\begin{itemize}
				\item[(A)] When $u_1\neq 0$, let $p_1\in \mathbb{S}^1$ be chosen such that $f_2''(p_1)u_1u_3<0$. We analyze the problem by considering the following three cases: 
				\begin{itemize}
					\item[(i)] If $f_2''(x_1)u_1u_3<0$, set
					\[\tau_1=\tau_3=0,\quad \tau_2=-\frac{u_3}{u_1f_2''(x_1)}.\] 
					Then, we conclude that $D$ is reachable from $S\mathbb{T}^3$.
					\item[(ii)] If $f_2''(x_1)u_1u_3>0$, we divide the analysis into two cases. We begin with the first case: $f_1(x_3) = f_3'(x_2) = 0$. Choose $i, j\in\mathbb{Z}$ such that
					\[\tau_3=\frac{y_3-x_3+2i\pi}{f_3(x_2)}\geq 0,\quad
					\tau_4=\frac{p_1-x_1+2j\pi}{f_1(y_3)}.\]
					Set $\tau_1=\tau_2=\tau_5=\tau_6=0$. This situation is now analogous to (1)(II)(A)(i). Next, consider the second case: 
					$f_1(x_3) \neq 0$ or $f_3'(x_2) \neq 0$. Without loss of generality, we may assume $f_1(x_3) \neq 0$. Choose $i\in\mathbb{Z}$ such that 
					\[\tau_1=\frac{p_1-x_1+2i\pi}{f_1(x_3)}\geq 0.\]
					Set $\tau_2=\tau_3=0$. This situation is now analogous to (1)(II)(A)(i).
					\item[(iii)] If $f_2''(x_1)=0$, choose $i, j\in\mathbb{Z}$ such that 
					\[\tau_2=\frac{y_3-x_3+2i\pi}{f_2'(x_1)}\geq 0,\quad
					\tau_4=\frac{p_1-x_1+2j\pi}{f_1(y_3)}\geq 0.\]
					Set $\tau_1=\tau_3=\tau_5=\tau_6=0$. This situation is now analogous to (1)(II)(A)(i). 
				\end{itemize}
				
				\item[(B)] When $u_1 = 0$, the proof is divided into two cases:
				\begin{itemize}
					\item[(i)] If $f_1'(x_3)\neq 0$, choosing $\tau_1 > 0$ arbitrarily and setting $\tau_2 = \tau_3 = 0$, this situation is now analogous to (1)(I) or (1)(II)(A).
					\item[(ii)] If $f_1'(x_3)=0$, we divide the analysis into two cases.
					\begin{enumerate}
						\item[(a)] If $f_2'(x_1) \neq 0$ or $f_3(x_2) \neq 0$. Without loss of generality, we may assume $f_2'(x_1) \neq 0$. Choose $i\in\mathbb{Z}$ such that 
						\[\tau_2=\frac{z_3-x_3+2i\pi}{f_2'(x_1)}\geq 0.\]
						Set $\tau_1=\tau_3=0$, this situation is now analogous to (1)(II)(A).
						\item[(b)] If $f_2'(x_1)=f_3(x_2)=0$, choose $i\in\mathbb{Z}$ such that 
						\[\tau_2=\frac{y_2-x_2+2i\pi}{f_2(x_1)}.\]
						Set $\tau_1=\tau_3=0$, this situation is now analogous to (1)(II)(B)(ii)(a).
					\end{enumerate}
				\end{itemize}
			\end{itemize}
		\end{itemize}
		
		\item[(2)] Second, we prove that $E$ is reachable from $D$. The proof is divided into two cases:
		\begin{itemize}
			\item[(I)] The first case is when $u_2=0$. We only need to prove that $E$ is reachable from $\{(x_1, x_2, x_3, -1, 0, 0): (x_1,x_2,x_3)\in\mathbb{T}^3\}$. 
			\begin{itemize}
				\item[(A)] If $f_2''(x_1)=0$, by assumption (H1), we have $f_2'(x_1)\neq 0$. Let $q_2\in\mathbb{S}^1$ be choose such that $f_3'(q_2)=0$ and $f_2'(x_1)f_3''(q_2)<0$. Choose $i, j\in\mathbb{Z}$ such that
				\[\tau_2=\frac{z_3-x_3+2i\pi}{f_2'(x_1)}>0,\quad
				\tau_4=\frac{q_2-x_2-\tau_2f_2(x_1)+2j\pi}{f_1'(z_3)}\geq 0.\]
				Since $f_i\in C^\omega(\mathbb{S}^1, \mathbb{R})$, we can choose $\tau_6>0$ such that $1+\tau_2\tau_6f_2'(x_1)f_3''(q_2)<0$ and $f_1(z_3+\tau_6f_3(q_2))\neq 0$. Choose $k\in\mathbb{Z}$ such that
				\[\tau_7=\frac{-\tau_4f_1(z_3)+2k\pi}{f_1(z_3+\tau_6f_3(q_2))}\geq 0.\] 
				Choose $\tau_8>0$ such that $(1+\tau_2\tau_6f_2'(x_1)f_3''(q_2))\tau_8+\tau_2=0$. Set $\tau_1=\tau_3=\tau_5=\tau_9=0$, we conclude that $E$ is reachable from $\{(x_1, x_2, x_3, -1, 0, 0): (x_1,x_2,x_3)\in\mathbb{T}^3\}$. 
				\item[(B)] If $f_2''(x_1)\neq 0$,we divide the analysis into two case. 
				\begin{itemize}
					\item[(i)] If $f_1(x_3) \neq 0$ or $f_3'(x_2) \neq 0$. Without loss of generality, we may assume $f_1(x_3) \neq 0$. Choose $i\in\mathbb{Z}$ such that 
					\[\tau_1=\frac{q_1-x_1+2i\pi}{f_1(x_3)}\geq 0.\]
					Set $\tau_2=\tau_3=0$. This situation is now analogous to (2)(I)(A).
					\item[(ii)] If $f_1(x_3) = f_3'(x_2) = 0$.  Choose $i$ such that
					\[\tau_3=\frac{y_3-x_3+2i\pi}{f_3(x_2)}\geq 0.\]
					Set $\tau_1=\tau_2=0$. This situation is now analogous to (2)(I)(B)(i).
				\end{itemize}
			\end{itemize}
			\item[(II)] The second case is when $u_2\neq 0$. The proof is divided into two cases:
			\begin{itemize}
				\item[(A)] If $u_1 \neq 0$, we divide the analysis into two case. 
				\begin{itemize}
					\item[(i)] The first case is when $u_1u_2<0$. If $f_2'(x_1)>0$, choose $\tau_2>0$ such that $u_1\tau_2f_2'(x_1)+u_2=0$, set $\tau_1=\tau_3=0$, this situation is now analogous to (2)(I). If $f_2'(x_1)\leq 0$, we consider the following two cases:
					\begin{enumerate}
						\item[(a)] If $f_1(x_3)\neq 0$, choose $i\in\mathbb{Z}$ such that
						\[\tau_1=\frac{q_1-x_1+2i\pi}{f_1(x_3)}\geq 0.\]
						Set $\tau_2=\tau_3=0$, this situation is now analogous to (2)(II)(A)(i) with $f_2'(x_1)>0$.
						\item[(b)] If $f_1(x_3)=0$, choose $i, j\in\mathbb{Z}$ such that
						\[\tau_1=\frac{y_2-x_2+2i\pi}{f_1'(x_3)}\geq 0,\quad 
						\tau_3=\frac{y_3-x_3+2j\pi}{f_3(y_2)}\geq 0.\]
						Set $\tau_1=0$, this situation is now analogous to (2)(II)(A)(i)(a).
					\end{enumerate}
					\item[(ii)] If $u_1u_2>0$, the argument follows similarly to case (2)(II)(A)(i).
				\end{itemize}
				\item[(B)] If $u_1 = 0$, we divide the analysis into two case. 
				\begin{itemize}
					\item[(i)] The first case is when $f_3''(x_2)\neq 0$. If $f_3'(x_2)=0$, choose any $\tau_3>0$, set $\tau_1=\tau_2=0$, this situation is now analogous to (2)(II)(A). If $f_3'(x_2)\neq 0$, we divide the analysis into two case.
					\begin{enumerate}
						\item[(a)] If $f_1'(x_3)\neq 0$ or $f_2(x_1)\neq 0$. Without loss of generality, we assume $f_1'(x_3)\neq 0$, choose $i\in\mathbb{Z}$ such that 
						\[\tau_1=\frac{y_2-x_2+2i\pi}{f_1'(x_3)}\geq 0.\]
						Set $\tau_2=\tau_3=0$, this situation is now analogous to (2)(II)(B)(i) with $f_3'(x_2)=0$.
						\item[(b)] If $f_1'(x_3)=f_2(x_1)=0$, choose $i\in\mathbb{Z}$ such that 
						\[\tau_2=\frac{z_3-x_3+2i\pi}{f_2'(x_1)}\geq 0.\]
						Set $\tau_1=\tau_3=0$, this situation is now analogous to (2)(II)(B)(i)(a).
					\end{enumerate}
					\item[(ii)] The second case is when $f_3''(x_2)=0$, the argument follows similarly to case (2)(II)(B)(i) with $f_3'(x_2)\neq 0$.
				\end{itemize}
			\end{itemize}
		\end{itemize}
		\item[(3)] Third, prove that $F$ is reachable from $E$. 
		For any $(x_1, x_2, x_3, 1, 0, 0)\in E$, 
		\begin{itemize}
			\item[(i)] if $f_1'(x_3)\neq 0$, by assumption (H1), we can choose $i, j, k \in \mathbb{Z}$ such that 
			\[
			\tau_1=\frac{y_2-x_2+2i\pi}{f_1'(x_3)}\geq 0,\quad \tau_3=\frac{y_3-z_3+2j\pi}{f_3(y_2)}\geq 0, \quad \tau_4=\frac{y_1-x_1+2k\pi}{f_1(y_3)}\geq 0.
			\]
			Set $\tau_2=\tau_5=\tau_6=0$. Then $\check{\Phi}_{\underline{\tau}}^2(x_1, x_2, z_3, 1, 0, 0)=(y_1, y_2, y_3, 1, 0, 0)$.
			\item[(ii)] For the case $f_1'(x_3)=0$,
			\begin{itemize}
				\item[(a)] if $f_3(x_2)\neq 0$, we can choose $i\in\mathbb{Z}$ such that
				\[\tau_3=\frac{z_3-x_3+2i\pi}{f_3(x_2)}\geq 0.\]
				Set $\tau_1=\tau_2=0$. Then employing a method of proof similar to that used in case (i), we establish that $F$ is reachable from $E$.
				\item[(b)] if $f_3(x_2)=0$, by assumption (H1), $f_3'(x_2)\neq 0$, choose $i, j\in\mathbb{Z}$ such that 
				\[\tau_3=\frac{y_1-x_1+2i\pi}{f_3'(x_2)}\geq 0, \quad \tau_5=\frac{z_2-x_2+2j\pi}{f_2(y_1)}\geq 0.
				\]
				Since $f_3, f_3'\in C^\omega(\mathbb{S}^1,\mathbb{R})$ and $f_3(z_2)/f_3'(z_2)$ is an irrational number, we can choose $k\in\mathbb{Z}$ such that 
				\[
				\tau_6=\frac{z_1-y_1+2k\pi}{f_3'(z_2)}\geq 0, \quad f_1'(x_3+\tau_6f_3(z_2))\neq 0.
				\]
				Set 
				\[
				\tau_1=\tau_2=\tau_4=\tau_7=\tau_9=0, \quad \tau_8:=-\frac{\tau_5f_2''(y_1)}{f_2''(z_1)}\geq 0.
				\]
				Then employing a method of proof similar to that used in case (i), we establish that $F$ is reachable from $E$.
			\end{itemize} 
		\end{itemize}
	\end{itemize}
	We conclude by noting that any point $\check{x}\in S\mathbb{T}^3$ is reachable from any point $\check{w}\in S\mathbb{T}^3$ by first traveling to $F$ and then traveling to $\check{x}$, using the same arguments as above.
\end{proof}
Next, we establish the Lie bracket condition for the projective Markov chain, which ensures the existence of an open small set $\check{\Phi}_{\underline{\tau}}$. Recall that for $x = (x_1, x_2, x_3)$, the vector fields $X_1$, $X_2$, and $X_3$ are defined as  
\[
X_1(x) = (f_1(x_3), f_1'(x_3), 0), \quad 
X_2(x) = (0, f_2(x_1), f_2'(x_1)), \quad \text{and} \quad 
X_3(x) = (f_3'(x_2), 0, f_3(x_2)).
\]
For $(x, u) \in S\mathbb{T}^3$, the corresponding vector fields $\check{X_i}$ are given by  
\[
\check{X_i}(x, u) = 
\begin{pmatrix}
	X_i(x) \\
	DX_i(x)u - \langle DX_i(x)u, u \rangle u
\end{pmatrix}, \quad i = 1, 2, 3.
\]
And the set of lifted vector fields is given by $\check{\mathcal{X}}=\{\check{X_1}, \check{X_2}, \check{X_3}\}$.
\begin{lem}\label{Lie condition of projective chain}
	Let $\check{\mathcal{X}}$ be as above. Then $\text{Lie}_x(\check{\mathcal{X}})=T_x S\mathbb{T}^3$ for some $\check{x}\in S\mathbb{T}^3$.
\end{lem}
\begin{proof}
	By a direct calculation, the Lie brackets of $\check{X_i}$ and $\check{X_j}$ for $i\neq j$ are given by 
	\begin{align*}
		[\check{X_i},\check{X_j}](\check{x})&=D\check{X_j}(\check{x})\check{X_i}(\check{x})-D\check{X_i}(\check{x})\check{X_j}(\check{x})\\
		&=\begin{pmatrix}
			[X_i,X_j](x)\\
			D[X_i,X_j](x)u-\langle D[X_i,X_j](x)u, u \rangle u
		\end{pmatrix}\\
		&=\check{[X_i,X_j]}(\check{x}).
	\end{align*}
	By a direct computation, we find that:
	\[
	[X_1,X_2](x)=\begin{pmatrix}
		-f_1'(x_3)f_2'(x_1)\\
		(f_1(x_3)-f_1''(x_3))f_2'(x_1)\\
		f_1(x_3)f_2''(x_1)
	\end{pmatrix}, \quad 
	[X_1,X_3](x)=\begin{pmatrix}
		(f_3''(x_2)-f_3(x_2))f_1'(x_3)\\
		-f_1''(x_3)f_3(x_2)\\
		f_1'(x_3)f_3'(x_2)
	\end{pmatrix},
	\]
	\[
	D[X_1,X_2](x)u=\begin{pmatrix}
		-u_1f_1'(x_3)f_2''(x_1)-u_3f_1''(x_3)f_2'(x_1)\\
		u_1(f_1(x_3)-f_1''(x_3))f_2''(x_1)+u_3(f_1'(x_3)-f_1'''(x_3))f_2'(x_1)\\
		u_1f_1(x_3)f_2'''(x_1)+u_3f_1'(x_3)f_2''(x_1)
	\end{pmatrix},
	\]
	\[
	D[X_1,X_3](x)u=\begin{pmatrix}
		u_2(f_3'''(x_2)-f_3'(x_2))f_1'(x_3)+u_3f_1''(x_3)(f_3''(x_2)-f_3(x_2))\\
		-u_2f_3'(x_2)f_1''(x_3)-u_3f_1'''(x_3)f_3(x_2)\\
		u_2f_1'(x_3)f_3''(x_2)+u_3f_1''(x_3)f_3'(x_2)
	\end{pmatrix}.
	\]
	The lifted vector fields are given by:
	\[\check{X_1}(\check{x})= 
	\begin{pmatrix}
		f_1(x_3) \\
		f_1'(x_3) \\
		0 \\
		-u_1 (u_1 u_3 f_1'(x_3) + u_2 u_3 f_1''(x_3)) + u_3 f_1'(x_3) \\
		-u_2 (u_1 u_3 f_1'(x_3) + u_2 u_3 f_1''(x_3)) + u_3 f_1''(x_3) \\
		-u_3 (u_1 u_3 f_1'(x_3) + u_2 u_3 f_1''(x_3))
	\end{pmatrix},\]
	\[
	\check{X}_2(\check{x}) = 
	\begin{pmatrix}
		0 \\
		f_2(x_1) \\
		f_2'(x_1) \\
		-u_1 (u_1 u_2 f_2'(x_1) + u_1 u_3 f_2''(x_1)) \\
		u_1 f_2'(x_1) - u_2 (u_1 u_2 f_2'(x_1) + u_1 u_3 f_2''(x_1)) \\
		u_1 f_2''(x_1) - u_3 (u_1 u_2 f_2'(x_1) + u_1 u_3 f_2''(x_1))
	\end{pmatrix},
	\]
	\[
	\check{X}_3(\check{x}) = 
	\begin{pmatrix}
		f_3'(x_2) \\
		0 \\
		f_3(x_2) \\
		-u_1 (u_1 u_2 f_3''(x_2) + u_2 u_3 f_3'(x_2)) + u_2 f_3''(x_2) \\
		-u_2 (u_1 u_2 f_3''(x_2) + u_2 u_3 f_3'(x_2)) \\
		u_2 f_3'(x_2) - u_3 (u_1 u_2 f_3''(x_2) + u_2 u_3 f_3'(x_2))
	\end{pmatrix},
	\]
	where $\check{x}=(x,u)=(x_1,x_2,x_3,u_1,u_2,u_3)$. Finally, define the matrix:
	\[M(\check{x}):=(\check{X_1}(\check{x}),\check{X_2}(\check{x}),\check{X_3}(\check{x}), [\check{X_1},\check{X_2}](\check{x}),[\check{X_1},\check{X_3}](\check{x})).\]
	Since $f_1 \in C^\omega(\mathbb{S}^1, \mathbb{R})$ is not a constant function, let $y_3$ be a point where $f_1$ achieves its maximum value, and $z_3$ be a point where $f_1$ achieves its minimum value. By assumption (H1), $f_1''(y_3) < 0$ and $f_1''(z_3) > 0$. Moreover, either $f_1(y_3) > 0$ or $f_1(z_3) < 0$. Consequently, either $f_1(y_3) - f_1''(y_3) > 0$ or $f_1(z_3) - f_1''(z_3) < 0$. A similar conclusion holds for $f_2$ and $f_3$. Thus, we can choose $(x_1, x_2, x_3)$ such that $f_1'(x_3)=f_2'(x_1)=f_3'(x_2)=0$, $f_1(x_3)\neq f_1''(x_3)$, $f_2(x_1)\neq f_2''(x_1)$, $f_3(x_2)\neq f_3''(x_2)$. Let $u_2 = 0 $, with $u_1 \neq 0$, $u_3 \neq 0$ and $u_1^2+u_3^2=1$, where the specific values of $u_1$ and $u_3$ will be provided below. Then, the matrix $M(x_1,x_2,x_3,u_1,0,u_3)$ is given by
	\[
	\resizebox{\textwidth}{!}{$
		\begin{pmatrix}
			f_1(x_3) & 0 & 0 & 0 & 0 \\
			0 & f_2(x_1) & 0 & 0 & -f_1''(x_3)f_3(x_2) \\
			0 & 0 & f_3(x_2) & f_1(x_3)f_2''(x_1) & 0 \\
			0 & -u_1^2u_3f_2''(x_1) & 0 & -u_1^2u_3f_1(x_3)f_2'''(x_1) & u_3^3f_1''(x_3)(f_3''(x_2)-f_3(x_2)) \\
			u_3f_1''(x_3) & 0 & 0 & u_1(f_1(x_3)-f_1''(x_3))f_2''(x_1) & -u_3f_1'''(x_3)f_3(x_2) \\
			0 & u_1f_2''(x_1)-u_1u_3^2f_2''(x_1) & 0 & u_1^3f_1(x_3)f_2'''(x_1) & -u_1u_3^2f_1''(x_3)(f_3''(x_2)-f_3(x_2))
		\end{pmatrix}.
		$}
	\]
	Define the matrix $N(x_1,x_2,x_3,u_1,0,u_3)$ as 
	\[
	\resizebox{\textwidth}{!}{$
		\begin{pmatrix}
			f_1(x_3) & 0 & 0 & 0 & 0 \\
			0 & f_2(x_1) & 0 & 0 & -f_1''(x_3)f_3(x_2) \\
			0 & 0 & f_3(x_2) & f_1(x_3)f_2''(x_1) & 0 \\
			u_3f_1''(x_3) & 0 & 0 & u_1(f_1(x_3)-f_1''(x_3))f_2''(x_1) & -u_3f_1'''(x_3)f_3(x_2) \\
			0 & u_1f_2''(x_1)-u_1u_3^2f_2''(x_1) & 0 & u_1^3f_1(x_3)f_2'''(x_1) & -u_1u_3^2f_1''(x_3)(f_3''(x_2)-f_3(x_2))
		\end{pmatrix},
		$}
	\]
	Then, we have
	\[
	\det N(x_1,x_2,x_3,u_1,0,u_3)=ac+bu_1^3u_3\neq 0,
	\]
	where 
	\[
	a:=a(u_1,u_3)=(1-u_3^2)f_2''(x_1)f_3(x_2)-f_2(x_1)(f_3''(x_2)-f_3(x_2))u_3^2,
	\]
	\[
	b:=b(u_1,u_3)=f_1^2(x_3)f_2(x_1)f_3^2(x_2)f_1'''(x_3)f_2'''(x_1),
	\] 
	and 
	\[
	c:=c(u_1,u_3)=u_1^2f_1(x_3)f_3(x_2)f_1''(x_3)f_2''(x_1)(f_1(x_3)-f_1''(x_3)).
	\]
	Clearly, $ c \neq 0 $. Based on the choices of $ x_1 $ and $ x_2 $, we can choose $ u_3 \neq 0 $ such that $ a \neq 0 $. If $ b = 0 $, for any $ u_1 $ such that $ u_1^2 + u_3^2 = 1 $, we have $ \det N(x_1, x_2, x_3, u_1, 0, u_3) \neq 0 $. Let $ d := ac $. If $ db > 0 $, choose $ u_1 $ such that $ u_1 u_3 > 0 $ and $ u_1^2 + u_3^2 = 1 $, then $ \det N(x_1, x_2, x_3, u_1, 0, u_3) \neq 0 $. If $ db < 0 $, choose $ u_1 $ such that $ u_1 u_3 < 0 $ and $ u_1^2 + u_3^2 = 1 $, then $ \det N(x_1, x_2, x_3, u_1, 0, u_3) \neq 0 $. Thus, we can choose $(x_1,x_2,x_3,u_1,u_2,u_3)$ such that the matrix $N(x_1,x_2,x_3,u_1,u_2,u_3)$ has rank 5. Consequently, $\text{Lie}_x(\check{\mathcal{X}})=T_x S\mathbb{T}^3$ for some $\check{x}\in S\mathbb{T}^3$.
\end{proof}
\begin{thm}[\textbf{Uniform geometric ergodicity for the projective Markov chain}]\label{Uniformly geometrically ergodic for projective chain}
	Let $\{\check{\Phi}_{\underline{\tau}}^m\}$ be as above. Then the transition kernel $\check{P}$ of $\{\check{\Phi}_{\underline{\tau}}^m\}$ is uniformly geometrically ergodic.
\end{thm}
\begin{proof}
	Similar to the proof of Theorem \ref{Uniformly geometrically ergodic}, it is sufficient to apply Theorem \ref{conditions for uniformly geometrically ergodic}. The existence of an open small set is established by Lemma \ref{Lie condition of projective chain} and Lemma \ref{small set}. Topological irreducibility is guaranteed by Lemma \ref{irreducibility for projective chain}. Strong aperiodicity follows directly from the condition $\mathbb{P}(\tau < \varepsilon) > 0$ for any $\varepsilon > 0$ and the fact that $\check{\Phi}_0(x,u) = (x,u)$. Finally, since $S\mathbb{T}^3$ is compact, $\check{P}$ satisfies a Lyapunov-Foster drift condition with the drift function $V \equiv 1$ automatically.
\end{proof}
\subsection{Proof of Corollary \ref{Ideal dynamo}}\label{subsection 5.2}
\quad Consider the kinematic dynamo equations on $\mathbb{T}^3$ under the frozen-flux limit ($\kappa = 0$):
\begin{equation*}
	\begin{cases}
		\partial_t B + (u \cdot \nabla) B - (B \cdot \nabla) u = 0, & \text{in } (0, \infty) \times \mathbb{T}^3, \\
		\nabla \cdot B = 0, & \text{in } (0, \infty) \times \mathbb{T}^3, \\
		B(0, \cdot) = B_0, & \text{in } \mathbb{T}^3,
	\end{cases}
\end{equation*}
where $u$ is the incompressible velocity field. Recall that
\begin{align*}
	&u_{\underline{\tau}}(t,x)\\
	=&\sum_{n=0}^{\infty} \left( \tau_{3n+1} X_1(x) \chi_{(3n,3n+1]}(t) + \tau_{3n+2} X_2(x) \chi_{(3n+1,3n+2]}(t) + \tau_{3n+3} X_3(x) \chi_{(3n+2,3n+3]}(t) \right).
\end{align*}
Let $\psi_{\underline{\tau}}^t$ be the flow of $\dot{x}=u_{\underline{\tau}}(t,x)$, then the solution of the above equation with $u=u_{\underline{\tau}}$ can be expressed as 
\[
B_{\underline{\tau}}(t,x)=D_{(\psi_{\underline{\tau}}^t)^{-1}(x)} \psi_{\underline{\tau}}^t B_0((\psi_{\underline{\tau}}^t)^{-1}(x)) \quad \text{for } (t,x) \in [0, \infty) \times \mathbb{T}^3.
\]
Since $\psi_{\underline{\tau}}^t$ is a diffeomorphism here, the above expression is equivalent to 
\begin{equation}\label{solu to magnetic eq}
	B_{\underline{\tau}}(t, \psi_{\underline{\tau}}^t(x))=D_x \psi_{\underline{\tau}}^t B_0(x)\quad \text{for } (t,x) \in [0, \infty) \times \mathbb{T}^3.
\end{equation}
Note that for discrete time, when $t=3n$ with $n \in \mathbb{N}$, we have $\psi_{\underline{\tau}}^{3n}(x)=\Phi_{\underline{\tau}}^n(x)$.

By Theorem \ref{Uniformly geometrically ergodic for projective chain}, the stationary probability measure of $\check{P}$ of the projective Markov chain $\{\check{\Phi}_{\underline{\tau}}^m\}$ is unique. Combining this result with Theorem III.1.2 in \cite{Ki}, we obtain the following lemma:
\begin{lem}\label{a version of Kifer}
	Let $\{\check{\Phi}_{\underline{\tau}}^m\}$ be as above. Then, for $\mu$-a.e. $x\in\mathbb{T}^3$, any $v\in \mathbb{S}^2$,  
	\begin{equation*}
		\lim_{m \to \infty} \frac{1}{m} \log \left| D_x \Phi_{\underline{\tau}}^m(x) v \right| = \lambda_1, \quad \mathbb{P}\text{-a.s.},
	\end{equation*}
    where $\lambda_1$ is the top positive Lyapunov exponent from Theorem \ref{Positivity of the top Lyapunov exponent}.
\end{lem}
We now proceed to prove the main result of this section, demonstrating that $u_{\underline{\tau}}$ is an ideal dynamo for $\mathbb{P}$-a.s. $\underline{\tau}$.
\begin{proof}[\textbf{Proof of Corollary \ref{Ideal dynamo}}]
	First, we claim that for any $\varepsilon \in (0,\lambda_1)$,  for $\mu$-a.e. $x\in\mathbb{T}^3$, any $v \in \mathbb{S}^2$ and $\mathbb{P}$-a.s. $\underline{\tau}$, there exists a positive constant $G=G(\varepsilon,x,\underline{\tau},v)$ such that for any $n \in \mathbb{N}$, there holds
	\[
	|D_x \Phi_{\underline{\tau}}^n v| \geq G e^{(\lambda_1-\varepsilon)n}.
	\]
	In fact, by Lemma \ref{a version of Kifer}, there exists $N=N(\varepsilon,x,\underline{\tau},v)>0$ such that for $n \geq N$, 
	\[
	\left| \frac{1}{n} \log|D_x \Phi_{\underline{\tau}}^n v|-\lambda_1  \right| < \varepsilon,
	\]
	which implies $|D_x \Phi_{\underline{\tau}}^n v| > e^{(\lambda_1-\varepsilon)n}$. For $n < N$, by utilizing the relation $|u|=|A^{-1}Au|\leq \|A^{-1}\||Au|$, we obtain that
	\[
	|D_x \Phi_{\underline{\tau}}^n v| \geq \frac{1}{C_0^{n} \prod_{i=1}^{3n} (1+2\tau_i)},
	\]
	where $C_0>1$ is independent of $n$ and $\underline{\tau}$. Therefore, the claim follows by taking 
	\[
	G(\varepsilon,x,\underline{\tau},v)=C_0^{-N}e^{-(\lambda_1-\varepsilon)N}\prod_{i=1}^{3N} (1+2\tau_i)^{-1}.
	\] 
	
	Next, we seek to establish a lower bound for the $L^2$ norm of $B(t)$. Applying H\"older's inequality, it suffices to focus on the $L^1$ norm. By \eqref{solu to magnetic eq} and the fact that $\psi_{\underline{\tau}}^t$ is volume preserving diffeomorphism, we have
	\begin{equation*}
		\|B_{\underline{\tau}}(t)\|_{L^1}= \int_{\mathbb{T}^3} |B_{\underline{\tau}}(t, \psi_{\underline{\tau}}^t(x))| dx = \int_{\mathbb{T}^3} |D_x \psi_{\underline{\tau}}^t B_0(x)| dx. 
	\end{equation*}   
    Take $B_0(x)$ as a constant unit vector field on $\mathbb{T}^3$. Then, for any $\varepsilon\in (0,\lambda_1)$ and any $n \in \mathbb{N}$,
	\begin{align}\label{integer}
		\|B_{\underline{\tau}}(3n)\|_{L^1} &= \int_{\mathbb{T}^3} \left| D_x \psi_{\underline{\tau}}^{3n} B_0(x) \right| dx\nonumber \\
		& = \int_{\mathbb{T}^3} \left|D_x \Phi_{\underline{\tau}}^n B_0(x) \right| dx \nonumber\\
		& \geq e^{(\lambda_1-\varepsilon)n} \int_{\mathbb{T}^3} G(\varepsilon,x,\underline{\tau},B_0(x))dx =: \hat{G} e^{(\lambda_1-\varepsilon)n}.
	\end{align}
	where $\hat{G}=\hat{G}(\varepsilon,\underline{\tau},B_0(x))$ is almost surely positive. 
	
	For time $t \in (3n,3(n+1))$, we have
	\[
	D_x\psi_{\underline{\tau}}^t=D_{\psi_{\underline{\tau}}^{3n}(x)} \psi_{\theta_{3n}\underline{\tau}}^{t-3n} \cdot D_x \psi_{\underline{\tau}}^{3n}. 
	\]
	Therefore, 
	\[
	|D_x \psi_{\underline{\tau}}^t B_0(x)| \geq\frac{|D_x \psi_{\underline{\tau}}^{3n} B_0(x)|}{\|(D_{\psi_{\underline{\tau}}^{3n}(x)} \psi_{\theta_{3n}\underline{\tau}}^{t-3n})^{-1}\|}\geq \frac{|D_x \psi_{\underline{\tau}}^{3n} B_0(x)|}{C_0\prod_{i=1}^{3} (1+2\tau_{3n+i})}.
	\]
	Then, by \eqref{integer} we obtain
	\begin{align*}
		\|B_{\underline{\tau}}(t)\|_{L^1} &= \int_{\mathbb{T}^3} |D_x \psi_{\underline{\tau}}^t B_0(x)| dx \\
		&\geq \frac{\hat{G} e^{(\lambda_1-\varepsilon)n}}{C_0\prod_{i=1}^{3} (1+2\tau_{3n+i})} =: \tilde{G} e^{\frac{1}{3}(\lambda_1-\varepsilon)t},
	\end{align*}
	where $\tilde{G}=\tilde{G}(\varepsilon,\underline{\tau},B_0(x)) =  C_0^{-1}e^{-(\lambda_1-\varepsilon)}\hat{G}\prod_{i=1}^{3} (1+2\tau_{3n+i})^{-1}$ is almost surely positive. Thus, the conclusion follows from the above estimate and the definition of $\gamma(0)$. 
\end{proof}

\section{Geometric ergodicity for the two-point Markov chain}\label{section 6}
    \quad To prove the almost-sure quenched correlation decay for the one-point chain, we rely on the $V$-uniformly geometrically ergodicity of the two-point chain $\{\tilde{\Phi}_{\underline{\tau}}^m\}$ and the Borel-Cantelli lemma. This involves applying Theorem \ref{conditions for uniformly geometrically ergodic} to the two-point chain. Since $\mathcal{I}$ is invariant under $\tilde{\Phi}_{\underline{\tau}}$, leading to at least one additional stationary measure for $P^{(2)}$, it is necessary to verify the conditions of Theorem \ref{conditions for uniformly geometrically ergodic} in the noncompact space $\mathcal{T}=\mathbb{T}^3\times \mathbb{T}^3\setminus \mathcal{I}$. The process of verifying the first three conditions is analogous to the one-point chain case as in the proof of Theorem \ref{Uniformly geometrically ergodic}, though with more careful treatment. To verify the Lyapunov-Foster drift condition, we introduce the concept of the twisted semigroup associated with the projective chain and proceed with spectral estimates. To obtain the desired properties of the eigenvalues and eigenfunctions, we utilize the uniformly geometrically ergodic property for the projective chain, as stated in \ref{Uniformly geometrically ergodic for projective chain}, along with the consequence of the positivity of the top Lyapunov exponent, as stated in Theorem \ref{Positivity of the top Lyapunov exponent}.
   \subsection{Topological irreducibility and small set property}
    \quad To establish V-uniform geometric ergodicity for the two-point Markov chain $\{\tilde{\Phi}_{\underline{\tau}}^m\}$, we need to verify the conditions of Theorem \ref{conditions for uniformly geometrically ergodic}. Denote
    \begin{align*}
    	\mathcal{I}_1 := \Bigl\{(x_1,x_2,x_3,y_1,y_2,y_3) &\in \mathcal{O} \times \mathcal{O}: \\
    	&f_1(y_3 + t) \equiv  f_1(x_3 + t),\\
    	&f_2(y_1 + t) \equiv  f_2(x_1 + t),\\
    	&f_3(y_2 + t) \equiv f_3(x_2 + t), \quad \text{for all } t \in \mathbb{R} \Bigr\},
    \end{align*}
    \begin{align*}
    	\mathcal{I}_2 := \Bigl\{(x_1,x_2,x_3,y_1,y_2,y_3) &\in \mathcal{O} \times \mathcal{O}: \\
    	&f_1(y_3 - t) \equiv f_1(x_3 + t),\\
    	&f_2(y_1 + t) \equiv -f_2(x_1 + t),\\
    	&f_3(y_2 - t) \equiv -f_3(x_2 + t), \quad \text{for all } t \in \mathbb{R} \Bigr\},
    \end{align*}
    \begin{align*}
    	\mathcal{I}_3 := \Bigl\{(x_1,x_2,x_3,y_1,y_2,y_3) &\in \mathcal{O} \times \mathcal{O}: \\
    	&f_1(y_3 + t) \equiv -f_1(x_3 + t),\\
    	&f_2(y_1 - t) \equiv -f_2(x_1 + t),\\
    	&f_3(y_2 - t) \equiv f_3(x_2 + t), \quad \text{for all } t \in \mathbb{R} \Bigr\},
    \end{align*}
    and 
    \begin{align*}
    	\mathcal{I}_4 := \Bigl\{(x_1,x_2,x_3,y_1,y_2,y_3) &\in \mathcal{O} \times \mathcal{O}: \\
    	&f_1(y_3 - t) \equiv -f_1(x_3 + t),\\
    	&f_2(y_1 - t) \equiv f_2(x_1 + t),\\
    	&f_3(y_2 + t) \equiv -f_3(x_2 + t), \quad \text{for all } t \in \mathbb{R} \Bigr\}.
    \end{align*}
    It follows from a direct computation that $\mathcal{I}=\cup_{k=1}^4 \mathcal{I}_k$ is the invariant set of $\tilde{\Phi}_{\underline{\tau}}$. We first prove the irreducibility of the two-point Markov chain on $\mathcal{T}=\mathbb{T}^3\times \mathbb{T}^3\setminus\mathcal{I}$.
    \begin{lem}\label{irreducible for two-point chain}
    	Assume (H1) holds. Given $(x',y')\in \mathcal{T}$ and $\varepsilon>0$, $B_\varepsilon(x',y')$ is reachable from $\mathcal{T}$. In particular, $\{\tilde{\Phi}_{\underline{\tau}}^m\}$ is topologically irreducible.
    \end{lem}
    \begin{proof}
    	For any given $(x,y)$, $(x',y')\in \mathcal{T}$, we represent their components as $x=(x_1,x_2,x_3)$, $y=(y_1,y_2,y_3)$, $x'=(x_1',x_2',x_3')$, $y'=(y_1',y_2',y_3')$. Define the sets as follows:
    	\[
    	D=\{(x,y)\in \mathcal{T}: f_1 (x_3) \text{ and } f_1(y_3) \text{ are rationally linearly independent}\},
    	\]
    	\[
    	E=\{(x,y)\in \mathcal{T}: x_1 \in B_{\varepsilon/12}(x_1'), y_1\in B_{\varepsilon/12}(y_1')\},
    	\]
    	\[
    	\resizebox{\textwidth}{!}{$
    		F=\{(x,y)\in \mathcal{T}: (x,y)\in B_{\varepsilon/12}(E), f_2(x_1), f_2(y_1), f_2'(x_1) \text{ and }f_2'(y_1) \text{ are rationally linearly independent}\}.
    		$}
    	\]
    	To establish that $B_\varepsilon(x',y')$ is reachable from $\mathcal{T}$, it is sufficient to demonstrate the following:
    	\begin{itemize}
    		\item[(1)] $D$ is reachable from $\mathcal{T}$.\\
    		Fix any $(x,y)\in \mathcal{T}$ and let $\delta>0$ small enough. Define the projection map $\pi: I\to\mathbb{T}^2$ by $\pi(x,y)=(x_3,y_3)$. Define $F_{(x,y)}:[0,\delta)^3\to \mathbb{T}^2$ by 
    		\begin{align*}
    			F_{(x,y)}(\tau_1,\tau_2,\tau_3)=\pi\circ\tilde{\Phi}_{\underline{\tau}}(x,y)
    			=\begin{pmatrix}
    				x_3+\tau_2 f_2'(x_1^{(1)})+\tau_3  f_3(x_2^{(2)}) \\
    				y_3+\tau_2 f_2'(y_1^{(1)})+\tau_3  f_3(y_2^{(2)}) 
    			\end{pmatrix}=:\begin{pmatrix}
    			x_3^{(3)} \\
    			y_3^{(3)}
    		\end{pmatrix},
    		\end{align*} 
    		where $x_1^{(1)}=x_1+\tau_1 f_1( x_3)$, $x_2^{(2)}=x_2+\tau_1 f_1'(x_3)+\tau_2 f_2(x_1^{(1)})$, $y_1^{(1)}=y_1+\tau_1 f_1(y_3)$, $y_2^{(2)}=y_2+\tau_1 f_1'(y_3)+\tau_2 f_2(y_1^{(1)})$. Then $D_{(\tau_1,\tau_2,\tau_3)}F_{(x,y)}$ is given by 
    		\[
    		\begin{pmatrix}
    			* & f_2'(x_1^{(1)})+\tau_3 f_3'(x_2^{(2)})f_2(x_1^{(1)})& f_3(x_2^{(2)}) \\
    			* & f_2'(y_1^{(1)})+\tau_3 f_3'(y_2^{(2)})f_2(y_1^{(1)})& f_3(y_2^{(2)})
    		\end{pmatrix},
    		\]
    		\begin{itemize}
    			\item[(I)] If $|f_1(x_3)| \neq |f_1(y_3)|$, we claim that  there exists $(\tau_1,\tau_2,\tau_3)\in [0,\delta)^3$ such that the determinant of the matrix
    			\begin{equation*}
    				A_{(x,y)}(\tau_1,\tau_2,\tau_3):=\begin{pmatrix}
    					f_2'(x_1^{(1)})+\tau_3 f_3'(x_2^{(2)})f_2(x_1^{(1)})& f_3(x_2^{(2)}) \\
    					f_2'(y_1^{(1)})+\tau_3 f_3'(y_2^{(2)})f_2(y_1^{(1)})& f_3(y_2^{(2)})
    				\end{pmatrix},
    			\end{equation*}
    			is nonzero. Without loss of generality, we may assume $f_1(x_3)\neq 0$. By contradiction, assume that for all $(\tau_1,\tau_2,\tau_3)\in [0,\delta)^3$, the determinant of $\det A_{(x,y)}(\tau_1,\tau_2,\tau_3)$ is identically zero. Then, we have 
    			\begin{equation*}
    				\left(f_2'(x_1^{(1)})+\tau_3 f_3'(x_2^{(2)})f_2(x_1^{(1)})\right)  f_3(y_2^{(2)}) \equiv \left( f_2'(y_1^{(1)})+\tau_3 f_3'(y_2^{(2)})f_2(y_1^{(1)})\right)  f_3(x_2^{(2)})
    			\end{equation*}
    			for all $(\tau_1,\tau_2,\tau_3)\in [0,\delta)^3$. And thus, for all $(\tau_1,\tau_2,\tau_3)\in [0,\delta)^3$, we have
    			\begin{equation}\label{eq51}
    				f_2'(x_1^{(1)}) f_3(y_2^{(2)}) \equiv f_2'(y_1^{(1)}) f_3(x_2^{(2)}),
    			\end{equation}
    			and 
    			\begin{equation}\label{eq52}	 
    				f_3'(x_2^{(2)})f_2(x_1^{(1)})f_3(y_2^{(2)}) \equiv f_3'(y_2^{(2)})f_2(y_1^{(1)})f_3(x_2^{(2)}).
    			\end{equation}
    			By a direct calculation, we obtain that there exist $\tau_1, \tau_2\in[0,\delta)$ such that $f_3( x_2^{(2)})\neq 0$ and $f_3( y_2^{(2)}) \neq 0$. By a direct calculation of \eqref{eq52}, we conclude that $f_3(y_2^{(2)}) \equiv c f_3(x_2^{(2)})$ where $c$ is a non-zero constant.
    			We divide the analysis into two cases:
    			\begin{itemize}
    				\item[(i)] $c\neq \pm 1$. Without loss of generality, we assume $|c|<1$. Fix $\tau_1 \in [0,\delta)$ such that $f_2(x_1+\tau_1 f_1(x_3)) \neq 0$. Then, for all $t\in \mathbb{R}$, we have
    				\[
    				f_3(t) \equiv c f_3\left( \frac{t-x_2-\tau_1 f_1'(x_3)}{f_2(x_1^{(1)})} \cdot f_2(y_1^{(1)}) +\tau_1 f_1'(y_3)+y_2 \right) =: c f_3(a+bt).
    				\]
    				By iterating the above equation, for any $k \geq 1$, 
    				\[
    				f_3(t) \equiv cf_3(a+bt) \equiv \cdots \equiv 
    				\begin{cases}
    					c^{k} f_3(\frac{1-b^k}{1-b}a+b^k t), & b\neq 1,\\
    					c^k f_3(ka+t), & b=1.
    				\end{cases}               
    				\] 
    				Therefore, we have $|f_3(t)| \leq |c|^{k} \sup_{x\in \mathbb{S}^1} |f_3(x)|$. Let $k \to \infty$, we obtain $f_3(t) \equiv 0$ for all $t \in \mathbb{R}$, which is a contradiction. 
    				\item[(ii)] $c=\pm 1$. Without loss of generality, we assume $c=1$. Then, $f_3(x_2^{(2)}) \equiv f_3(y_2^{(2)})$.	By a direct calculation, we obtain that there exist $\tau_1, \tau_2\in[0,\delta)$ such that $f_3( x_2^{(2)})=f_3( y_2^{(2)})\neq 0$. Thus, by \eqref{eq51} we have $f_2'(x_1+\tau_1 f_1(x_3)) \equiv f_2'(y_1+\tau_1 f_1(y_3))$ for all $\tau_1$ in an open interval within $(0,\delta)$. Since $|f_1(x_3)| \neq |f_1(y_3)|$, then for all $t\in \mathbb{R}$, we have
    					\[
    					f_2'(t) \equiv f_2'\left( \frac{t-x_1}{f_1(x_3)} \cdot f_1(y_3)+y_1 \right) =: f_2'(a+bt). 
    					\]
    					We may assume $|b|=|f_1(y_3)/f_1(x_3)|<1$. Then, by iterating the above equation, for any $k \geq 1$,
    					\[
    					f_2'(t) \equiv f_2'(a+bt) \equiv \cdots \equiv f_2'\left(\frac{1-b^k}{1-b}a+b^k t\right).
    					\]
    					Let $k \to \infty$, we obtain $f_2'(t) \equiv f_2'(a/(1-b))$ for all $t \in \mathbb{R}$, which is a contradiction.
    			\end{itemize}
    			\item[(II)] If $|f_1(x_3)| = |f_1(y_3)|\neq 0$, we claim that there exist $(\tau_1,\tau_2,\tau_3)\in [0,\delta)^3$ such that $|f_1(x_3^{(3)})|\neq |f_1(y_3^{(3)})|$. We can then repeat the proof of (1)(I). Now, we prove the claim. Without loss of generality, we assume $f_1(x_3)=f_1(y_3)\neq 0$.
    			\begin{itemize}
    				\item[(i)] If there exists $\tau_1\in [0,\delta)$ such that $|f_2'(x_1+\tau_1f_1(x_3))|\neq |f_2'(x_1+\tau_1f_1(y_3))|$, then we can choose $\tau_2$ such that $|f_1(x_3+\tau_2f_2'(x_1+\tau_1f_1(x_3)))|\neq |f_1(y_3+\tau_2f_2'(y_1+\tau_1f_1(y_3)))|$. Setting $\tau_3=0$ completes the proof.
    				\item[(ii)] If $|f_2'(x_1 + \tau_1 f_1(x_3))| \equiv |f_2'(x_1 + \tau_1 f_1(y_3))|$ for all $\tau_1 \in [0, \delta)$, and since $f_2 \in C^{\omega}(\mathbb{S}^1, \mathbb{R})$, it follows that $|f_2'(x_1 + t)| \equiv |f_2'(x_1 + t)|$ for all $t \in \mathbb{R}$. Without loss of generality, assume that $f_2'(x_1 + t) \equiv -f_2'(y_1 + t)$, which implies that $f_2(x_1 + t) \equiv -f_2(y_1 + t)+c$ for all $t \in \mathbb{R}$, where $c\in\mathbb{R}$ is a constant. Choose $\tau_1\in [0,\delta)$ such that $f_2'(x_1+\tau_1f_1(x_3)))\neq 0$, $f_2(x_1+\tau_1f_1(x_3)))\neq 0$, and 
    				\[
    				\frac{c}{f_2(x_1 + \tau_1 f_1(x_3))} - 1 \notin \{ 0, \pm 1 \} \quad \text{if} \quad c \neq 0.
    				\]
    				
    				If there exists $\tau_2\in [0,\delta)$ such that $|f_1(x_3+\tau_2f_2'(x_1+\tau_1f_1(x_3)))|\neq |f_1(y_3+\tau_2f_2'(y_1+\tau_1f_1(y_3)))|$. Setting $\tau_3=0$ completes the proof. 
    				
    				If $|f_1(x_3+\tau_2f_2'(x_1+\tau_1f_1(x_3)))|\equiv |f_1(y_3+\tau_2f_2'(y_1+\tau_1f_1(y_3)))|$ for all $\tau_2\in [0,\delta)$, and since $f_1(x_3)=f_1(y_3)\neq 0$, it follows that $f_1(x_3+t)\equiv f_1(y_3-t)$ for all $t\in\mathbb{R}$. If there exists $\tau_2\in [0,\delta)$ such that $|f_3(x_2^{(2)}))|\neq |f_3(y_2^{(2)})|$, then we can choose $\tau_3$ such that $|f_1(x_3^{(3)})|\neq |f_1(y_3^{(3)})|$, completing the proof. If $|f_3(x_2^{(2)})|\equiv |f_3(y_2^{(2)})|$ for all $\tau_2\in [0,\delta)$, we proceed by splitting the analysis into the following two cases:
    				\begin{enumerate}
    					\item[(a)] If $f_3(x_2^{(2)})\equiv f_3(y_2^{(2)})$ for all $\tau_2\in [0,\delta)$, and since $f_2(x_1 + t) \equiv -f_2(y_1 + t)+c$ for all $t \in \mathbb{R}$, along with $f_1(x_3)=f_1(y_3)\neq 0$, it follows that $f_3(x_2+\tau_2f_2(x_1+\tau_1f_1(x_3)))\equiv f_3(y_2-\tau_2f_2(x_1+\tau_1f_1(x_3))+c\tau_2)$. 
    					
    					If $c\neq 0$, for all $t\in \mathbb{R}$, we have
    					\begin{align*}
    						f_3(t) &\equiv f_3\left(x_2+y_2- \frac{c x_2}{f_2(x_1+\tau_1f_1(x_3))}+\left(\frac{c}{f_2(x_1+\tau_1f_1(x_3))}-1\right)t \right) \\
    						&=:f_3(a+bt).
    					\end{align*}
    				    Without loss of generality, assume $|b|<1$. By iterating this equation, we obtain for any $k \geq 1$,
    				    \[
    				    f_3(t) \equiv f_3(a+bt) \equiv \cdots \equiv f_3\left(\frac{1-b^k}{1-b}a+b^k t\right).
    				    \]
    				    Let $k \to \infty$, we obtain $f_3(t) \equiv f_3(a/(1-b))$ for all $t \in \mathbb{R}$, which is a contradiction. 
    				    
    				    If $c=0$, $f_3(x_2+\tau_2f_2(x_1+\tau_1f_1(x_3)))\equiv f_3(y_2-\tau_2f_2(x_1+\tau_1f_1(x_3)))$ for all $\tau_2\in [0,\delta)$. Since $f_1\in C^{\omega}(\mathbb{S}^1, \mathbb{R})$ is non-constant, we can choose $\tau_2\in [0, \delta)$ such that $f_1(p)\neq f_1(2(y_3-\tau_2f_2'(x_1+\tau_1f_1(x_3)))-p)$ for some $p\in\mathbb{S}^1$. If $|f_1(x_3+\tau_2f_2'(x_1+\tau_1f_1(x_3))+\tau_3f_3(x_2+\tau_2f_2(x_1+\tau_1f_1(x_3))))|\equiv |f_1(y_3-\tau_2f_2'(y_1+\tau_1f_1(y_3))+\tau_3f_3(y_2+\tau_2f_2(y_1+\tau_1f_1(y_3))))|$ for all $\tau_3\in [0,\delta)$, and since $f_1(x_3)=f_1(y_3)\neq 0$, $f_2(x_1+t)\equiv- f_2(y_1+t)$, $f_1(x_3+t)\equiv f_1(y_3-t)$ for all $t\in\mathbb{R}$, it follows that $f_1(t_3)\equiv f_1(2(y_3-\tau_2f_2'(x_1+\tau_1f_1(x_3)))-t_3)$ for all $t_3\in\mathbb{R}$, which is a contradiction.
    					\item[(b)] If $f_3(x_2^{(2)})\equiv -f_3(y_2^{(2)})$ for all $\tau_2\in [0,\delta)$, and since $f_2(x_1 + t) \equiv -f_2(y_1 + t)+c$ for all $t \in \mathbb{R}$ and $f_1(x_3)=f_1(y_3)\neq 0$, it follows that $f_3(x_2+\tau_2f_2(x_1+\tau_1f_1(x_3)))\equiv -f_3(y_2-\tau_2f_2(x_1+\tau_1f_1(x_3))+c\tau_2)$. If $c \neq 0$, the situation is analogous to case (1)(II)(ii)(a). Now, assuming $c = 0$, we obtain
    					\[
    					f_3(x_2 + t) \equiv -f_3(y_2 - t) \quad \text{for all} \quad t \in \mathbb{R}.
    					\]
    					This implies that $(x,y)\in \mathcal{I}_2$, which leads to a contradiction.
    				\end{enumerate}
    			\end{itemize}
    			\item[(III)] If $f_1( x_3)=f_1(y_3)=0$, we can select $\tau_1\in[0,\delta)$ such that $f_3(x_2+\tau_1 f_1'(x_3))\neq 0$, set $\tau_2=0$, and choose $\tau_3\in[0,\delta)$ such that $f_1(x_3+\tau_3 f_3(x_2+\tau_1 f_1'(x_3))\neq 0$. This can be reduced to the previous argument.
    		\end{itemize}
    		Thus, by the constant rank theorem, $D$ is reachable from $\mathcal{T}$. 
    		\item[(2)] $E$ is reachable from $D$.\\
    		Fix any $(x,y)\in D$, the flow map $(x_1,y_1) \mapsto (x_1+t f_1(x_3), y_1+tf_1(y_3))$, $t\in [0,\infty)$ is dense on $\mathbb{T}^2$, since $f_1(x_3)$ and $f_1(y_3)$ are rationally linearly independent. Then, we can choose $\tau_1\geq 0$ and $\tau_2=\tau_3=0$ such that $x_1+\tau_1 f_1(x_3) \in B_{\varepsilon/12}(x_1')$ and $y_1+\tau_1 f_1(y_3)\in B_{\varepsilon/12}(y_1')$. Thus, $E$ is reachable from $D$.
    		 \item[(3)] $F$ is reachable from $E$.\\
    		 This proof is similar to case (1), and  by choosing 
    		 \[\delta< \frac{\varepsilon}{24M},\]
    		 where $M:=\sup_{x \in \mathbb{S}^1}\max\{|f_2(x)|, |f_3'(x)|\}$, we deduce that $F$ is reachable from $E$. It is important to note that, in this case, we require not only that $f_2(x_1)$ and $f_2(y_1)$ are rationally linearly independent for some $x_1$ and $y_1$, but also that $f_2(x_1)$, $f_2'(x_1)$, $f_2(y_1)$, $f_2'(y_1)$ are rationally linearly independent. Nevertheless, this condition can still be satisfied. In fact, for any $(x,y)\in E$, define $G_{(x,y)}: [0,\delta)^3\to \mathbb{T}^2$ by
    		 \begin{align*}
    		 	&G_{(x,y)}(\tau_1,\tau_2,\tau_3)\\
    		 	=&\begin{pmatrix}
    		 		x_1+\tau_1f_1(x_3)+\tau_3f_3'(x_2+\tau_1f_1'(x_3)+\tau_2f_2(x_1+\tau_1f_1(x_3)))\\
    		 		y_1+\tau_1f_1(y_3)+\tau_3f_3'(y_2+\tau_1f_1'(y_3)+\tau_2f_2(y_1+\tau_1f_1(y_3)))
    		 	\end{pmatrix}\\
    		 	=:&(x_1^{(3)}, y_1^{(3)}).
    		 \end{align*}  
    		 Following the method used in case (1), we can demonstrate that there exists $\underline{\tau} \in [0, \delta)^3$ where $G_{(x,y)}$ is a submersion at $\underline{\tau}$. By the constant rank theorem, the image of $G_{(x,y)}$ contains a connected open neighborhood $U$ of $(x_1^{(3)}, y_1^{(3)})$. For any $q=(q_1, q_2, q_3, q_4)\in \mathbb{Q}^4\setminus\{\textbf{0}\}$, define $H_q: U\to\mathbb{R}$ by
    		 \[
    		 H_q(z_0,z_1)=q_1f_1(z_0)+q_2f_1'(z_0)+q_3f_1(z_1)+q_4f_1'(z_1).
    		 \]
    		 Since $f_1\in C^{\omega}(\mathbb{S}^1,\mathbb{R})$ is non-constant, $H_q$ is not identically zero. Using a result established in \cite{M}, it follows that for every $q \in \mathbb{Q}^4 \setminus \{\mathbf{0}\}$,
    		 \[
    		 H_q^{-1}(0):=\{(z_0,z_1):H_q(z_0,z_1)=0\}
    		 \]
    		 has zero Lebesgue measure. Denote
    		 \[
    		 \mathcal{N}:=\bigcup_{q\in \mathbb{Q}^4\setminus\{\textbf{0}\}}H_q^{-1}(0),
    		 \]
    		 and note that $\mathcal{N}$ is also a set of Lebesgue measure zero. Therefore, for any $(z_0, z_1)\in U\setminus \mathcal{N}$, $f_2(z_0)$, $f_2'(z_0)$, $f_2(z_1)$, and $f_2'(z_1)$ are rationally linearly independent.
    		 \item[(4)] $B_\varepsilon(x', y')$ is reachable from $F$.\\            
    		 By applying a similar proof strategy as in case (2), we demonstrate that $B_\varepsilon(x',y')$ is reachable from $F$. In this case, the flow map is $(x_2,x_3,y_2,y_3)\mapsto(x_2+t f_2(x_1)), x_3+t f_2'(x_1), y_2+t f_2(y_1), y_3+t f_2'(y_1))$. And we choose $\tau_2\geq 0$ and $\tau_1=\tau_3=0$ such that $x_2+\tau_2 f_2(x_1)\in B_{\varepsilon/3}(x_2')$, $x_3+\tau_2 f_2'(x_1)\in B_{\varepsilon/3}(x_3')$, $y_2+\tau_2 f_2(y_1)\in B_{\varepsilon/3}(y_2')$ and $y_3+\tau_2 f_2'(y_1)\in B_{\varepsilon/3}(y_2')$. 
    	\end{itemize}
    \end{proof}
    We next establish the Lie bracket condition for the two-point Markov chain, which ensures the existence of an open small set. 
    \begin{lem}\label{Lie bracket of two-point}
    	Let $\tilde{\mathcal{X}}$ be as above. Then, there exists some ${\tilde{x}}\in \mathcal{T}$ such that $\text{Lie}_{\tilde{x}}(\tilde{\mathcal{X}})=T_{\tilde{x}}\mathcal{T}$. 
    \end{lem}
    \begin{proof}
    	For any $(x,y)=(x_1, x_2, x_3, y_1, y_2, y_3)\in\mathcal{T}$,
    \begin{equation*}
    	\tilde{X}_1(x, y)=\begin{pmatrix}
    		f_1(x_3)\\
    		f_1'(x_3)\\
    		0\\
    		f_1(y_3)\\
    		f_1'(y_3)\\
    		0
    	\end{pmatrix},\quad
    	\tilde{X}_2(x, y)=\begin{pmatrix}
    		0\\
    		f_2(x_1)\\
    		f_2'(x_1)\\
    		0\\
    		f_2(y_1)\\
    		f_2'(y_1)
    	\end{pmatrix},\quad
    	\tilde{X}_3(x, y)=\begin{pmatrix}
    		f_3'(x_2)\\
    		0\\
    		f_3(x_2)\\
    		f_3'(y_2)\\
    		0\\
    		f_3(y_2)
    	\end{pmatrix}.
    \end{equation*}
    And their Lie bracket are
    \begin{equation*}
    	[\tilde{X}_1,\tilde{X}_2](x,y)=\begin{pmatrix}
    		-f_1'(x_3)f_2'(x_1)\\
    		(f_1(x_3)-f_1''(x_3))f_2'(x_1)\\
    		f_1(x_3)f_2''(x_1)\\
    		-f_1'(y_3)f_2'(y_1)\\
    		(f_1(y_3)-f_1''(y_3))f_2'(y_1)\\
    		f_1(y_3)f_2''(y_1)
    	\end{pmatrix},\quad
    	[\tilde{X}_1,\tilde{X}_3](x,y)=\begin{pmatrix}
    		(f_3''(x_2)-f_3(x_2))f_1'(x_3)\\
    		-f_1''(x_3)f_3(x_2)\\
    		f_1'(x_3)f_3'(x_2)\\
    		(f_3''(y_2)-f_3(y_2))f_1'(y_3)\\
    		-f_1''(y_3)f_3(y_2)\\
    		f_1'(y_3)f_3'(y_2)
    	\end{pmatrix},
    \end{equation*}
    and
    \begin{equation*}
    	[\tilde{X}_2,\tilde{X}_3](x,y)=\begin{pmatrix}
    		f_2(x_1)f_3''(x_2)\\
    		-f_2'(x_1)f_3'(x_2)\\
    		(f_2(x_1)-f_2''(x_1))f_3'(x_2)\\
    		f_2(y_1)f_3''(y_2)\\
    		-f_2'(y_1)f_3'(y_2)\\
    		(f_2(y_1)-f_2''(y_1))f_3'(y_2)
    	\end{pmatrix}.
    \end{equation*}
    Define matrix 
    \begin{equation*}
    	M(x,y):=(\tilde{X}_1(x,y),\tilde{X}_2(x,y), \tilde{X}_3(x,y), [\tilde{X}_1,\tilde{X}_2](x,y), [\tilde{X}_1,\tilde{X}_3](x,y), [\tilde{X}_2,\tilde{X}_3](x,y)).
    \end{equation*}
    We claim that there exist $(x,y)\in \mathcal{T}$ such that $\det M(x,y) \neq 0$. Let $x_{i-1}, y_{i-1}\in C_{f_i'}$, to be determined later. Then, for $(x,y)=(x_1,x_2,x_3,y_1,y_2,y_3)$, 
    \begin{equation*}
    	M(x,y)=\begin{pmatrix}
    		f_1(x_3) & 0 & 0&0&0 &f_2(x_1)f_3''(x_2) \\
    		0& f_2(x_1) &0&0&-f_1''(x_3)f_3(x_2) & 0 \\
    		0&0&f_3(x_2)&f_1(x_3)f_2''(x_1)&0&0\\
    		f_1(y_3)&0&0&0&0& f_2(y_1)f_3''(y_2)\\
    		0& f_2(y_1) &0&0&-f_1''(y_3)f_3(y_2) & 0 \\
    		0&0&f_3(y_2)&f_1(y_3)f_2''(y_1)&0&0\\
    	\end{pmatrix}.
    \end{equation*}
    By a direct calculation, 
    \begin{equation*}
    	\det M(x,y)= \Pi_{i=1}^3 f_i(y_{i-1})f_i''(y_{i-1}) \left( f_i(x_{i-1})- f_i(y_{i-1}) \frac{f_{i+1}(x_i)f_{i+2}''(x_{i+1})}{f_{i+1}(y_i)f_{i+2}''(y_{i+1})} \right).
    \end{equation*}
    Let $y_j=x_j\in C_{f_{j+1}'}$ for $j=1, 2$. Let $x_3,y_3$ be the minimum and maximum points of $f_1$ respectively. By a directi calculation, we have $(x,y) \in \mathcal{T}$ and $\det M(x,y)\neq 0$. Thus, the claim follows.
    \end{proof}
   The next two subsections focus on constructing a function that satisfies the Lyapunov-Foster drift condition for $\tilde{\Phi}_{\underline{\tau}}$.
   \subsection{Projective and twisted semigroups}
   \quad Recall the projective chain $\check{\Phi}_{\underline{\tau}}^m$ on the sphere bundle $S\mathbb{T}^3$ with the associated Markov semigroup $\check{P}$. To construct a function $V$ that satisfies the Lyapunov-Foster drift condition for $\{\tilde{\Phi}_{\underline{\tau}}\}$, we introduce the concept of the twisted Markov semigroup, regarded as a perturbation of the Markov semigroup $\check{P}$, and analyze its associated spectral properties.
   
    Consider the twisted projective semigroup $\{\check{P}_q\}_{q\in\mathbb{R}}$ defined for bounded and measurable $\psi$ on $S\mathbb{T}^3$ by
    \[
    \check{P}_q \psi(x,u)=\mathbb{E} |D_x \Phi_{\underline{\tau}} u|^{-q} \psi \left(\Phi_{\underline{\tau}}(x), {D_x \Phi_{\underline{\tau}} u}/{|D_x \Phi_{\underline{\tau}} u|} \right).
    \]
    Notice that $\check{P}_0=\check{P}$. In order to construct the ideal Lyapunov function $V$ with drift condition, we need to explore the spectral theory of $\check{P},\check{P}_q$. To facilitate the presentation, we first outline some background from spectral theory. For further details, we recommend references such as \cite{Ka, DS}.
    
    Let $X$ be a complex Banach space, and $T$ be a bounded linear operator in $X$. Denote $I$ as the identity operator. The \textbf{spectrum} $\sigma(T)$ of $T$ is the set of all $z\in\mathbb{C}$  for which the operator  $zI-T$ does not have an inverse that is a bounded linear operator. A point $z_0\in\sigma(T)$ is said to be an \textbf{isolated} point of $\sigma(T)$, if there is a neighborhood $U$ of $z_0$ such that $\sigma(T)\cap U=\{z_0\}$. A point $z_0\in\sigma(T)$ is called \textbf{dominant} if it has the largest modulus among all spectrum of $T$. A point $z_0\in\sigma(T)$ is said to be an \textbf{eigenvalue} of $T$, if there exists a nonzero $x\in X$, called an eigenvector, such that $Tx=z_0x$. An eigenvalue $z_0$ of $T$ is said to be \textbf{simple} if the generalized eigenspace 
    \[
    \mathcal{N}_{z_0}=\bigcup_{k\geq 1}\mathcal{N}((z_0I - T)^k)
    \]
    has dimension $1$, where $\mathcal{N}((z_0I - T)^k) = \{x \in X: (z_0I - T)^k x = 0\}$. 
    
    Let $z\to T(z)$ be an operator valued function of a complex parameter $z$ which is holomorphic in the uniform operator topology. We summarize some standard results on spectral perturbation from \cite{DS}; see Lemma 3 (p. 585), Lemma 7 (p. 587), and Theorem 9 (p. 587).
    \begin{lem}\cite{DS}\label{spectral perturbation}
    	Let $T$ be a bounded linear operator on a complex Banach space $X$.
    	\begin{itemize}
    		\item[(i)] For any $\varepsilon>0$, there is a $\delta>0$ such that if $T_1$ is a bounded linear operator on $X$ and $\|T_1-T\|<\delta$, then $\sigma(T_1)\subset B_\varepsilon(\sigma(T))$.
    		\item[(ii)] Let $\Pi_1, \Pi_2$ be two projections in $X$ such that $\|\Pi_1-\Pi_2\|<\min\{\|\Pi_1\|^{-1}, \|\Pi_2\|^{-1}\}$. If one of the projections has a finite dimensional range, then so does the other and $\dim \Pi_1 X=\dim \Pi_2 X$.
    		\item[(iii)] Let $\gamma>0$ and let $T(z)$ be a bounded linear operator valued function defined and holomorphic for $|z|<\gamma$. Let $r_0$ be an isolated point of $\sigma(T(0))$, and suppose the range of the spectral projection $\Pi_{r_0}$ has finite dimension $m$. Let $U$ be an open set with $\overline{U}\cap \sigma(T(0))=\{r_0\}$. Then there is a positive $\delta<\gamma$, an integer $k\leq m$ such that for $|z|<\delta$, $U\cap \sigma(T(z))$ is a finite set $\{r_1(z), r_2(z), \cdots, r_k(z)\}$. Each functions $r_i(z)$ depends holomorphically on the principal value of the fractional power $z^{1/n}$ of $z$, and satisfies $r_i(0)=r_0$.
    	\end{itemize}
    \end{lem}
    \begin{rem}\label{holomorphic of r_i(z)}
    	By the proof of Theorem 9 (p. 587) in \cite{DS}, if $m=1$, then $k=1$ and $n=1$, i.e., $r_1(z)$ is holomorphic in $z$.
    \end{rem}
    We also require the following results from Proposition 49.1 (p. 206) in \cite{H1982}.
    \begin{lem}\cite{H1982}\label{Range of Riesz Projector}
    	Let $T$ be a bounded linear operator on a complex Banach space $X$. If $z_0$ is an isolated point of the spectrum of $T$, then the range of the Riesz projector $\Pi_{z_0}$ satisfies
    	\[
    	\text{Ran} (\Pi_{z_0})=\left\{x:\lim_{n \to \infty}\|(z_0 I-T)^nx\|^{1/n}=0\right\}.
    	\]
    \end{lem}
    To apply Lemma \ref{spectral perturbation} and \ref{Range of Riesz Projector}, we consider the analytic continuation of $\check{P}_q \psi(x,u)$, $\psi \in \mathcal{B}_b(S \mathbb{T}^3)$, from $q \in \mathbb{R}$ to the complex domain $\mathbb{C}$. First, we present the following lemma:
    \begin{lem}\label{boundedness of the series}
    	For any $z\in\mathbb{C}$, we have
    	\begin{equation*}
    		\sum_{k= 0}^{\infty} \frac{|z|^k}{k!} \mathbb{E}\left| \ln |D_{x} \Phi_{\underline{\tau}}u|\right|^k < \infty.
    	\end{equation*}
    \end{lem}
    \begin{proof}
    	By utilizing the relation $|u|=|A^{-1}Au|\leq \|A^{-1}\||Au|$, it can be directly computed that for any $(x,u) \in S\mathbb{T}^3$, there exists a constant $C_0>1$ satisfying
    	\begin{equation}\label{polynomial}
    		\frac{1}{C_0 (1 + 2\tau_1)(1 + 2\tau_2)(1 + 2\tau_3)}\leq|D_x \Phi_{\underline{\tau}}u| \leq C_0 (1 + 2\tau_1)(1 + 2\tau_2)(1 + 2\tau_3). 
    	\end{equation}
        Then, for any $(x,u) \in S\mathbb{T}^3$, we have 
    	\begin{align}\label{ln function}
    		\mathbb{E}\left| \ln |D_{x} \Phi_{\underline{\tau}}u|\right|^k &\leq \int_{\mathbb{R}_{+}^3} \left( \sum_{i=1}^3 \ln [C_0^{1/3} (1+2t_i)]\right)^k \frac{1}{h^3} e^{-(t_1+t_2+t_3)/h} dt_1dt_2dt_3 \nonumber\\
    		&\leq \frac{3^{k-1}}{h^3} \int_{\mathbb{R}_{+}^3} \left( \sum_{i=1}^3 \ln^k [C_0^{1/3} (1+2t_i)]\right)  e^{-(t_1+t_2+t_3)/h} dt_1dt_2dt_3\nonumber \\
    		&= 3^k \int_0^\infty \ln^k [C_0^{1/3} (1+2ht)]  e^{-t} dt,
    	\end{align}
    	where the second inequality holds due to H\"older's inequality. Then, we have
    	\begin{align*}
    		\sum_{k= 0}^{\infty} \frac{|z|^k}{k!} \mathbb{E}\left| \ln |D_{x} \Phi_{\underline{\tau}}u|\right|^k&\leq \int_{0}^{\infty}\sum_{k= 0}^{\infty} \frac{3^k|z|^k}{k!} \ln^k [C_0^{1/3} (1+2ht)]  e^{-t} dt\\
    		&=\int_{0}^{\infty} e^{3|z|\ln  [C_0^{1/3} (1+2ht)]}e^{-t} dt\\
    		&=\int_{0}^{\infty}[C_0^{1/3} (1+2ht)]^{3|z|}e^{-t}dt<\infty.
    	\end{align*}
    \end{proof}
    Next, applying Lemma \ref{boundedness of the series} and Fubini's theorem, for any $\psi \in \mathcal{B}_b(S \mathbb{T}^3)$, we have
    \begin{align*}
    	\check{P}_q \psi(x,u)&=\mathbb{E} |D_x \Phi_{\underline{\tau}} u|^{-q} \psi \left(\Phi_{\underline{\tau}}(x), {D_x \Phi_{\underline{\tau}} u}/{|D_x \Phi_{\underline{\tau}} u|} \right)\\
    	&=\mathbb{E} e^{-q\ln |D_x \Phi_{\underline{\tau}} u|}\psi \left(\Phi_{\underline{\tau}}(x), {D_x \Phi_{\underline{\tau}} u}/{|D_x \Phi_{\underline{\tau}} u|} \right)\\
    	&=\mathbb{E} \sum_{k=0}^{\infty}\frac{q^k}{k!}(-\ln |D_x \Phi_{\underline{\tau}} u|)^k\psi \left(\Phi_{\underline{\tau}}(x), {D_x \Phi_{\underline{\tau}} u}/{|D_x \Phi_{\underline{\tau}} u|} \right)\\
    	&=\sum_{k=0}^{\infty}\frac{q^k}{k!}\mathbb{E}(-\ln |D_x \Phi_{\underline{\tau}} u|)^k\psi \left(\Phi_{\underline{\tau}}(x), {D_x \Phi_{\underline{\tau}} u}/{|D_x \Phi_{\underline{\tau}} u|} \right).
    \end{align*}
    Naturally, we consider the analytic continuation of $\check{P}_q \psi(x,u)$. Specifically, for any $z\in\mathbb{C}$, we define
    \begin{equation}\label{analytic continuation}
    	\check{P}_z \psi(x,u)=\sum_{k=0}^{\infty}\frac{z^k}{k!}\mathbb{E}(-\ln |D_x \Phi_{\underline{\tau}} u|)^k\psi \left(\Phi_{\underline{\tau}}(x), {D_x \Phi_{\underline{\tau}} u}/{|D_x \Phi_{\underline{\tau}} u|} \right).
    \end{equation}
    It follows from Lemma \ref{boundedness of the series} that the right-hand side of (\ref{analytic continuation}) converges and is holomorphic for all $z \in \mathbb{C}$. Therefore, $\check{P}_z \psi(x, u)$ is well-defined for every $z \in \mathbb{C}$. Now, we define the linear operator $\check{P}_z$ on $L^{\infty}(S\mathbb{T}^3;\mathbb{C})$ as follows:
    \[
    \check{P}_z g(x, u) := \check{P}_z \text{Re}(g(x, u)) + i \check{P}_z \text{Im}(g(x, u)),
    \]
    where $\text{Re}(g(x, u))$ and $\text{Im}(g(x, u))$ are the real and imaginary parts of $g(x, u)$, respectively. It is obvious that $\check{P}_z \psi(x, v)=\check{P}_q \psi(x, v)$ when $z=q\in\mathbb{R}$ and $\psi(x, v)\in L^{\infty}(S\mathbb{T}^3)$.
    To apply Lemma \ref{spectral perturbation}, it suffices to show that $\check{P}_z$, $z\in\mathbb{C}$ is a bounded perturbation of $\check{P}_0$.
    \begin{lem}\label{bound for matrix norm of P_q}
    	Let $\check{P}_z$, $z\in\mathbb{C}$ be as above. We have the following properties:
    	\begin{itemize}
    		\item[(i)] $\check{P}_q$ is a bounded linear operator on $L^{\infty}(S\mathbb{T}^3;\mathbb{C})$ for all $z\in\mathbb{C}$.
    		\item[(ii)] The family of operators $\check{P}_z$ converges to $\check{P}_0$ in norm on $L^\infty(S\mathbb{T}^3;\mathbb{C})$ as $z \to 0$.
    	\end{itemize}
    \end{lem}
    \begin{proof}
    	We first consider the proof of (i). For $g\in L^{\infty}(S\mathbb{T}^3;\mathbb{C})$, by Lemma \ref{boundedness of the series}, we have  
    	\[
    	\|\check{P}_z g\|_{L^{\infty}(S\mathbb{T}^3;\mathbb{C})}\leq  2C(z,h)\|g\|_{L^{\infty}(S\mathbb{T}^3;\mathbb{C})},
    	\]
    	where $C(z,h) > 0$ is a constant depending on $z$ and $h$.
    	
    	Next, consider (ii). For $g(x,v)\in L^{\infty}(S\mathbb{T}^3;\mathbb{C})$, by (\ref{ln function}), we have  
    	\begin{align*}
    		\|\check{P}_z g-\check{P}_0 g\|_{L^{\infty}(S\mathbb{T}^3;\mathbb{C})}
    		\leq & 2\|g\|_{L^{\infty}(S\mathbb{T}^3;\mathbb{C})}\sum_{k=1}^{\infty}\frac{|z|^k}{k!}\mathbb{E}|\ln |D_x \Phi_{\underline{\tau}} u||^k\\
        \leq&2\|g\|_{L^{\infty}(S\mathbb{T}^3;\mathbb{C})}\int_{0}^{\infty}\sum_{k=1}^{\infty} \frac{3^k|z|^k}{k!} \ln^k [C_0^{1/3} (1+2ht)]  e^{-t} dt\\
        =&2\|g\|_{L^{\infty}(S\mathbb{T}^3;\mathbb{C})}\int_{0}^{\infty} \left([C_0^{1/3} (1+2ht)]^{3|z|}-1\right)e^{-t} dt,
        \end{align*} 
        which tends to zero as $z \to 0$, by Lebesgue's dominated convergence theorem.
    \end{proof}
     Let $T$ be a bounded linear operator on a Banach space $X$ with a simple, dominant, isolated eigenvalue $r$. The operator $T$ is said to have a \textbf{spectral gap} if there exists $\varepsilon > 0$ such that
     \[
     \sigma(T) \setminus \{r\} \subset B_{|r|-\varepsilon}(0),
     \]
     where $\sigma(T)$ denotes the spectrum of $T$ and $B_{|r|-\varepsilon}(0)$ is the open ball of radius $|r| - \varepsilon$ centered at the origin.
    \begin{lem}\label{spectral gap}
    	If $\check{P}$ is uniformly geometrically ergodic with respect to a unique stationary measure $\check{\pi}$ on $S\mathbb{T}^3$, then 
    	there exists $q_0 > 0$ such that for all $q \in (-q_0,q_0)$, $\check{P}_q$ has a simple, dominant, isolated eigenvalue $r(q)\in\mathbb{R}$ and a spectral gap. 
    \end{lem}
    \begin{proof}
    	First, we prove that $\check{P}$ has a simple, dominant, isolated eigenvalue $1$ and a spectral gap. Since $\check{P}$ is uniformly geometrically ergodic with respect to $\check{\pi}$, there exists $C>0,\gamma \in (0,1)$ such that for all $n \geq 1$,
    	\begin{equation}\label{uniform geometric ergodic in SM}
    		\|\check{P}^n - 1\otimes\check{\pi}\|_{L^{\infty}(S\mathbb{T}^3)} \leq C \gamma^n,			
    	\end{equation}
    	where for a function $r$ and measure $\mu$ we define $(r\otimes \mu)(x,dy):=r(x)\mu(dy)$. Similar to $\check{P}_{z}$, we define the linear operator $1\otimes\check{\pi}$ on $L^{\infty}(S\mathbb{T}^3;\mathbb{C})$ as follows:
    	\[
    	(1\otimes\check{\pi})g := (1\otimes\check{\pi}) \text{Re}(g) + i (1\otimes\check{\pi}) \text{Im}(g),
    	\]
    	where $\text{Re}(g)$ and $\text{Im}(g)$ are the real and imaginary parts of $g$, respectively. 
    	
    	By \eqref{uniform geometric ergodic in SM}, it follows that for any complex number $\lambda$ with $|\lambda| > \gamma$,
    	\[
    	\sum_{n=1}^{\infty }\|\lambda^{-n}(\check{P}_0-1 \otimes \check{\pi})^n\|_{L^{\infty}(S\mathbb{T}^3;\mathbb{C})} =\sum_{n=1}^{\infty }\|\lambda^{-n}(\check{P}_0^n-1 \otimes \check{\pi})\|_{L^{\infty}(S\mathbb{T}^3;\mathbb{C})} \leq 2C \sum_{n=1}^{\infty }\left(\frac{\gamma}{|\lambda|}\right)^n < \infty.
    	\]
    	Hence, if $|\lambda| > \gamma$, the Neumann series
    	\[
    	\sum_{n = 0}^{\infty}\lambda^{-n}(\check{P}-1 \otimes \check{\pi})^n
    	\]
    	converges in operator norm. As a consequence, if $|\lambda| > \gamma$, we obtain
    	\[
    	Z_\lambda:=(\lambda I - \check{P}_0 + 1 \otimes \check{\pi})^{-1} = \frac{1}{\lambda}\sum_{n = 0}^{\infty}\lambda^{-n}(\check{P}_0-1 \otimes \check{\pi})^n,
    	\]
    	and this operator is bounded and linear on $L^{\infty}(S\mathbb{T}^3;\mathbb{C})$. Furthermore, for $|\lambda| > \gamma$ and $\lambda \neq 1$, we deduce from
    	\[
    	\lambda I-\check{P}_0=(\lambda I-\check{P}_0+1 \otimes \check{\pi})-1 \otimes \check{\pi},
    	\]
    	and 
    	\[
    	Z_\lambda\left( I+(\lambda-1)^{-1} 1\otimes\check{\pi} \right)=Z_\lambda \left( I-\lambda^{-1}1\otimes\check{\pi} \right)^{-1}=Z_\lambda \left( I-Z_\lambda1\otimes\check{\pi} \right)^{-1},
    	\]
    	that 
    	\begin{equation}\label{the inverse of operator}
    		(\lambda I-\check{P}_0)Z_\lambda \left( I+(\lambda-1)^{-1} 1\otimes\check{\pi} \right)=I.
    	\end{equation}
    	Thus, $(\lambda I-\check{P}_0)^{-1}$ also exists as a bounded linear operator on $L^{\infty}(S\mathbb{T}^3;\mathbb{C})$ whenever $|\lambda| > \gamma$ and $\lambda \neq 1$. Since $\gamma < 1$ and $\check{P}$ is uniformly geometrically ergodic, we can imply that $1$ is the dominant, isolated eigenvalue of $\check{P}_0$ and the eigenspace corresponding to $1$, namely $\ker(I - \check{P}_0)$, has dimension 1. Let $\Pi_1$ be the Riesz projection of $\check{P}_0$ at $1$. It follows from Lemma \ref{Range of Riesz Projector} that 
    	\[
    	\text{Ran} (\Pi_{1})=\left\{f\in L^{\infty}(S\mathbb{T}^3; \mathbb{C}):\lim_{n \to \infty}\|(I-\check{P}_0)^n f\|^{1/n}=0\right\}.
    	\]
    	It is clear that generalized eigenspace $\mathcal{N}_{1}=\bigcup_{k\geq 1}\mathcal{N}((I - T)^k)$ is a subspace of $\text{Ran} (\Pi_{1})$. Next, we prove that $\text{Ran} (\Pi_{1})\subset \ker(I - \check{P}_0)$. By (\ref{the inverse of operator}), there exists $\delta>0$ such that for any $\lambda\in\mathbb{C}$ with $0<|\lambda-1|\leq \delta$, we have
    	\begin{align*}
    		\left\|(\lambda-1)(\lambda I - \check{P}_0)^{-1}\right\|&=\left\|(\lambda-1)Z_\lambda \left( I+(\lambda-1)^{-1} 1\otimes\check{\pi}\right)\right\|\\
    		&=\left\|((\lambda-1)I+1\otimes\check{\pi})Z_\lambda\right\|\\
    		&\leq C,
    	\end{align*}
        where $C>0$ is a constant. Therefore, we conclude that
    	\begin{align*}
    		(I-\check{P}_0)\Pi_1&=\frac{1}{2\pi i} \int _ {\Gamma_1} (I-\check{P}_0)(\lambda I - \check{P}_0)^{-1} d\lambda\\
    		&= \frac{1}{2\pi i} \int _ {\Gamma_1} (1-\lambda)(\lambda I - \check{P}_0)^{-1} d\lambda\\
    		&=0,
    	\end{align*}
        where $\Gamma_1$ is a circ around $1$ which isolates $1$ from $\sigma(\check{P}_0)\setminus\{1\}$. Thus, $\dim(\mathcal{N}_{1})=\dim (\text{Ran} (\Pi_{1}))=\dim(\ker(I - \check{P}_0))=1$. As a consequence, $\check{P}_0$ has a spectral gap. Let $\alpha\in (0,1)$ such that $\sigma(\check{P}_0)\setminus\{1\}\subset B_{\alpha}(0)$. Set $\varepsilon\ll (1-\gamma)/3$ where $\gamma$ as in \eqref{uniform geometric ergodic in SM}. By Lemma \ref{spectral perturbation}(i), there exists $\delta_1>0$ such that $\sigma(\check{P}_z) \subset B_{\alpha+\varepsilon}(0)\cup B_\varepsilon(1)$ for $|z|< \delta_1$.   
    	Let $U$ be an open ball in $\mathbb{C}$ centered at $1$ with radius smaller than $(1-\gamma)/3$, then $\overline{U} \cap \sigma(\check{P}_0)=\{1\}$. By the fact that $1$ is a simple eigenvalue of $\check{P}_0$ and utilizing Lemma \ref{spectral perturbation}(iii), we find that there exists a positive $\delta_2 < \delta_1$ such that for $|z|<\delta_2$, $\overline{U} \cap \sigma(\check{P}_z)=\{r(z)\}$. Hence, $r(z)$ is an isolated, dominant point of $\sigma(\check{P}_z)$. 
    	
    	Let $\Pi_{r(z)}$ be the spectral projection of $\check{P}_z$ at $r(z)$ and let $\Gamma$ be a positively oriented small circle centered at $1$ with radius $(1-\gamma)/3$, we have
    	\begin{align*}
    		&\|\Pi_{r(z)}-\Pi_0\|_{L^{\infty}(S\mathbb{T}^3;\mathbb{C})}\\
    		 =&\frac{1}{2\pi} \left\| \int_\Gamma \left[ (\lambda I - \check{P}_z)^{-1} - (\lambda I - \check{P}_0)^{-1} \right] d\lambda \right\|_{L^{\infty}(S\mathbb{T}^3;\mathbb{C})} \\
    		=&\frac{1}{2\pi} \left\| \int_\Gamma  (\lambda I - \check{P}_z)^{-1} (\check{P}_z-\check{P}_0) (\lambda I - \check{P}_0)^{-1}  d\lambda \right\|_{L^{\infty}(S\mathbb{T}^3;\mathbb{C})} \\
    		\leq &\frac{1}{2\pi} \| \check{P}_z-\check{P}_0\|_{L^{\infty}(S\mathbb{T}^3;\mathbb{C})} \left\| \int_\Gamma  (\lambda I - \check{P}_z)^{-1}(\lambda I - \check{P}_0)^{-1}  d\lambda \right\|_{L^{\infty}(S\mathbb{T}^3;\mathbb{C})} \\
    		\leq &\frac{|\Gamma|}{2\pi} \| \check{P}_z-\check{P}_0\|_{L^{\infty}(S\mathbb{T}^3;\mathbb{C})}  \sup_{\lambda \in \Gamma} \|(\lambda I - \check{P}_z)^{-1}\| _{L^{\infty}(S\mathbb{T}^3;\mathbb{C})}
    		\sup_{\lambda \in \Gamma} \|(\lambda I - \check{P}_0)^{-1}\|_{L^{\infty}(S\mathbb{T}^3;\mathbb{C})}, 
    	\end{align*}
    	which converges to $0$ as $z \to 0$, by Lemma \ref{bound for matrix norm of P_q}(ii). Then, take $\varepsilon>0$ sufficiently small, there exists $\delta_3>0$ such that when $|z|<\delta_3$
    	\begin{equation}\label{projection estimate}
    		\|\Pi_{r(z)}-\Pi_0\|_{L^{\infty}(S\mathbb{T}^3;\mathbb{C})}<\varepsilon.  
    	\end{equation}
    	Therefore, by Lemma \ref{spectral perturbation}(ii), the range of $\Pi_{r(z)}$ has dimension $1$ for $|z|<\delta_3$.
    	And since $r(z)$ is an isolated point of $\sigma(\check{P}_z)$, it follows that $r(z)$ is an eigenvalue of $\check{P}_z$. Next, we show that this eigenvalue is real when $z\in\mathbb{R}$. Notice that when $z\in \mathbb{R}$, $ \check{P}_z$ maps the real and imaginary parts of $g$ to the real and imaginary parts, respectively. If $\text{Im}(r(z)) \neq 0$, then the complex conjugate $\overline{r(z)}$ would also be an eigenvalue of $\check{P}_z$ which is a contradiction to $\overline{U} \cap \sigma(\check{P}_z)=\{r(z)\}$. 
    	
    	Finally, set $q_0=\delta_3$, then $\check{P}_q$ has a simple, dominant, isolated eigenvalue $r(q)\in (1-\varepsilon, 1+\varepsilon)$ and a spectral gap for $q \in (-q_0,q_0)$. We now turn our attention to the eigenfunction corresponding to $r(q)$.
    \end{proof}    
In Lemma \ref{bound for matrix norm of P_q} and \ref{spectral gap}, we regard $\check{P}_z, z\in\mathbb{C}$ as bounded linear operators on $L^{\infty}(S\mathbb{T}^3;\mathbb{C})$. In the subsequent, we consider the family $\{\check{P}_q\}_{q \in \mathbb{R}}$ defined for $\psi \in L^\infty(S \mathbb{T}^3)$ by
\[
\check{P}_q \psi(x,u)=\mathbb{E} |D_x \Phi_{\underline{\tau}} u|^{-q} \psi \left(\Phi_{\underline{\tau}}(x), {D_x \Phi_{\underline{\tau}} u}/{|D_x \Phi_{\underline{\tau}} u|} \right),
\]
as bounded linear operators on $L^\infty(S \mathbb{T}^3)$.
\begin{lem}\label{eigenfunction as a limit}
Assume $\check{P}$ is uniformly geometrically ergodic. Let $\{\check{P}_q\}_{q \in \mathbb{R}}$ be bounded linear operators on $L^\infty(S \mathbb{T}^3)$ as above. Let $q_0$ and $r(q)$ for $q\in[0, q_0]$ be as given in Lemma \ref{spectral gap}. Then for all $q\in [0, q_0]$, there is an eigenfunction $\psi_q$ corresponding to the dominant eigenvalue $r(q)$ of the operator $\check{P}_q$ satisfies
\begin{equation}\label{def of eigenfunction}
    \lim_{n \to \infty} \left\|\psi_q-r(q)^{-n} \check{P}_q^n \mathbf{1}\right\|=0.
\end{equation}
\end{lem}
\begin{proof}
Let $\Pi_{r(q)}$ be the Riesz projection of $\check{P}_q$ at $r(q)$. We decompose $\check{P}_q$ as follows:
\begin{equation*}
\check{P}_q= \check{P}_q \Pi_{r(q)} + \check{P}_q (I-\Pi_{r(q)}) = r(q) \Pi_{r(q)}+\check{P}_q (I-\Pi_{r(q)}) ,
\end{equation*}
Hence, for any $n \in \mathbb{N}$, we have 
\begin{equation*}
\check{P}_q^n = r(q)^n \Pi_{r(q)} + \check{P}_q^n (I - \Pi_{r(q)}).
\end{equation*}
Notice that $\check{P}_q(1-\Pi_{r(q)})$ is also a linear bounded operator on $L^\infty(S \mathbb{T}^3)$. Denote $s(q)$ the spectral radius of $\check{P}_q(1-\Pi_{r(q)})$. The Gelfand's formula shows 
\[
s(q)=\lim_{n \to \infty} \| \check{P}_q^n (I-\Pi_{r(q)}) \|^{{1}/{n}}.
\]
Since $\check{P}_q$ has a spectral gap with a simple, dominant, isolated eigenvalue $r(q)$, it follows that $s(q) < r(q)$. Consequently, we have 
\[
\lim_{n \to \infty}r(q)^{-n} \|\check{P}_q^n (I-\Pi_{r(q)}) \| = 0.
\] 
By \eqref{projection estimate} and the fact that $\Pi_{0} \mathbf{1}= \mathbf{1}$, we conclude that $q\in [0,q_0]$, we have $\Pi_{r(q)} \mathbf{1} \neq \mathbf{0}$. Define $\psi_q=\Pi_{r(q)} \mathbf{1}$. It follows that
\[
\lim_{n \to \infty} \left\|\psi_q-r(q)^{-n} \check{P}_q^n \mathbf{1}\right\|=0,
\]
and that
\[
\check{P}_q\psi_q=\check{P}_q\Pi_{r(q)} \mathbf{1}=r(q)\Pi_{r(q)} \mathbf{1}=r(q)\psi_q.
\]
\end{proof}
We now outline some additional properties of $\psi_q$ and $r(q)$, which will be utilized to establish the drift condition for the two-point chain.
\begin{lem}\label{property of psi}
	Assume $\check{P}$ is uniformly geometrically ergodic. Then, we have the following:
	\begin{itemize}
		\item[(i)] For any $\varepsilon>0$, there exists a constant $C_{\varepsilon}>0$ such that $\psi_q$ can be decomposed as $\psi_q=\psi_{q,0}+\psi_{q,1}$, where $\| \psi_{q,0} \| \leq \varepsilon$ and $\psi_{q,1}$ is continuously differentiable, satisfies $\| \psi_{q,1}\|_{C^1} \leq C_{\varepsilon}$. 
		\item[(ii)] Moreover, if $\check{P}$ is topologically irreducible, then $\psi_q$ is strictly positive on $S\mathbb{T}^3$.
	\end{itemize}
\end{lem}
\begin{proof}
Consider part (i) first. By \eqref{def of eigenfunction}, for any $\varepsilon>0$, there exists $m>0$ such that 
\[
\sup_{(x,u) \in S\mathbb{T}^3}|\psi_q(x,u)-r(q)^{-m} \check{P}_q^m \mathbf{1}(x,u)|<\varepsilon.
\]
Define $\psi_{q,1}:= r(q)^{-m} \check{P}_q^m \mathbf{1},\, \psi_{q,0}:= \psi_q -\psi_{q,1}$. Clearly, $\| \psi_{q,0} \| \leq \varepsilon$. Furthermore, the bound $\| \psi_{q,1}\|_{C^1} \leq C_{\varepsilon}$ follows from $\check{P}_q^n \mathbf{1}(x,u)=\mathbb{E} |D_x \Phi_{\underline{\tau}}^n u|^{-q}$ is continuously differentiable on $S\mathbb{T}^3$. 

Turning to part (ii). By the fact that $\{\check{P}_q^n \mathbf{1}\}_n$ is a sequence of continuous functions and by \eqref{def of eigenfunction}, $\psi_q\geq 0$ is a continuous function on $S\mathbb{T}^3$ and is not identically zero. Consequently, there exists a constant $C>0$ and an open set $U \subset S\mathbb{T}^3$ such that $\psi_q |_{U} \geq C$. Hence, for any $(x,u) \in S\mathbb{T}^3$, $n \geq 1$, 
\begin{align*}
\psi_q (x,u) 
&= r(q)^{-n} \mathbb{E} |D_x \check{\Phi}_{\underline{\tau}}^n u|^{-q} \psi_q(\Phi_{\underline{\tau}}^n(x), {D_x \Phi_{\underline{\tau}}^n u}/{|D_x \Phi_{\underline{\tau}}^n u|}) \\
&\geq C r(q)^{-n} \check{P}^n\left( (x, u) \in U \right)  \cdot \mathbb{E} |D_x \check{\Phi}_{\underline{\tau}}^n u|^{-q}.
\end{align*}
Using \eqref{polynomial}, we observe that $\mathbb{E} |D_x \check{\Phi}_{\underline{\tau}}^n u|^{-q} >0$. By Lemma \ref{irreducibility for projective chain}, $\{\check{\Phi}_{\underline{\tau}}^m\}$ is topologically irreducible. Then, there exists $n\geq 1$ such that $\check{P}^n\left( (x, u) \in U \right) >0$. Therefore, $\psi_q$ is strictly positive on $S\mathbb{T}^3$. 
\end{proof}

\begin{lem}\label{derivative of r(q)}
	Assume $\check{P}$ is uniformly geometrically ergodic.  Then, $r'(0)= -\lambda_1$. In particular, if $\lambda_1 >0$, then $r(q) <1$ for sufficiently small $q>0$.
\end{lem}
\begin{proof}
	Let $q_0$ be as in Lemma \ref{spectral gap}, for any $q \in (-q_0, q_0)$, by Remark \ref{holomorphic of r_i(z)} and the prove of Lemma \ref{spectral gap}, $r(q)$ is analytic and $r(q)>0$. Therefore, we define $\Lambda(q)=\log r(q)$. Since $r(0)=1$, we only need to prove $\Lambda'(0)=-\lambda_1$. For any $(x,u) \in S\mathbb{T}^3$, we have $\check{P}_q^n \mathbf{1}(x,u)=\mathbb{E} |D_x \Phi_{\underline{\tau}}^n u|^{-q}.$ By \eqref{def of eigenfunction}, we have 
	\begin{equation*}
		\psi_q(x,u)=\lim_{n \to \infty}r(q)^{-n}\check{P}_q^n\mathbf{1}(x,u),
	\end{equation*}
    where the limit holds uniformly in $(x,u)\in S\mathbb{T}^3$.
    Then, by Lemma \ref{property of psi}(ii), we obtain
    \begin{align}\label{uniform limit}
    	\lim_{n \to \infty} \frac{1}{n} \log \mathbb{E} |D_x \Phi_{\underline{\tau}}^n u|^{-q}=&\lim_{n \to \infty} \frac{1}{n} \log \check{P}_q^n \mathbf{1}(x,u)\nonumber\\
    	=&\lim_{n \to \infty}\frac{1}{n}\log (r(q)^{-n}\check{P}_q^n\mathbf{1}(x,u))+\log r(q)\nonumber\\
    	=&\lim_{n\to\infty}\frac{1}{n}\log \psi_q(x,u)+\Lambda(q)=\Lambda(q),
    \end{align}
     where the limit holds uniformly in $(x,u)\in S\mathbb{T}^3$. For any $q \in (-q_0, q_0)$ and $n\in\mathbb{N}$, take $u_n$ such that $\|D_x \Phi_{\underline{\tau}}^n\|=|D_x \Phi_{\underline{\tau}}^n u_n|$. Then, by Jensen's inequality, we have
	\[
	\mathbb{E} |D_x \Phi_{\underline{\tau}}^n u_n|^{-q} = \mathbb{E} e^{-q\log |D_x \Phi_{\underline{\tau}}^n u_n|} \geq e^{-q \mathbb{E} \log \|D_x \Phi_{\underline{\tau}}^n\|}.
	\]
	Taking the logarithm of both sides, dividing by $n$, and letting $n \to \infty$. Since the convergence in \eqref{uniform limit} is uniform, we obtain
	\[
	\Lambda(q) \geq -q \lambda_1.
	\]
	Hence, for $q \in (-q_0,0)$, we have $\Lambda(q)/q \leq -\lambda_1$ and for $q \in (0,q_0)$, we have $\Lambda(q)/q \geq -\lambda_1$. Moreover, we conclude that $\Lambda'(0-) \leq -\lambda_1$ and $\Lambda'(0+) \geq -\lambda_1$. Since $\Lambda(q)$ is analytic for $q \in (-q_0,q_0)$, $\Lambda'(0)=-\lambda_1$, and the result follows.
\end{proof}

\subsection{Lyapunov-Foster drift condition}
\quad Building on the preliminary estimates established earlier, this subsection aims to provide an explicit construction of the function $V$ for the two-point chain $\{\tilde{\Phi}_{\underline{\tau}}^m\}$. We begin by decomposing the invariant sets $\mathcal{I}_{i}$, for $i=1,2,3,4$ in order to establish the translation relation between these sets to be used in constructing the function $V$. 

We start with $\mathcal{I}_1$. For any $(x,y) \in \mathcal{I}_1$, $y_1-x_1$ is a period of $f_2$, $y_2-x_2$ is a period of $f_3$, $y_3-x_3$ is a period of $f_1$. Since $f_i\in C^{\omega}(\mathbb{S}^1,\mathbb{R})$, $i=1, 2, 3$ are non-constant functions, there exist finite constants $a_i, b_j, c_k$ such that for any $(x,y) \in \mathcal{I}_1$, there exist $i$, $j$ and $k$ satisfying $y_1-x_1=a_i$, $y_2-x_2=b_j$ and $y_3-x_3=c_k$. Consequently, $\mathcal{I}_1$ can be characterized as the union of the following finitely many invariant sub-manifolds:
\[
\mathcal{I}_{a_i, b_j, c_k}:=\left\{ (x,y) \in \mathcal{O} \times \mathcal{O}: y_1-x_1=a_i,\, y_2-x_2=b_j,\,y_3-x_3=c_k \right\}.
\]
We proceed with $\mathcal{I}_i,i=2,3,4$ which share the similar structures. For any $(x,y) \in \mathcal{I}_2$, $(x_3+y_3)/2$ is an axis of even symmetry for $f_1$, 
 $2(y_1-x_1)$ is a period of $f_2$ and $((x_2+y_2)/2, 0)$ is the center of symmetry of the graph of $f_3$. Since $f_i \in C^w(\mathbb{S}^1;\mathbb{R})$ are non-constant, there exist finite constants $a_i', b_j', c_k'$ such that for any $(x,y) \in \mathcal{I}_2$, there exist there exist $i$, $j$ and $k$ satisfying $2(y_1-x_1)=a_i'$, $(x_2+y_2)/2=b_j'$ and $(x_3+y_3)/2=c_k'$. Consequently, $\mathcal{I}_2$ can be characterized as the union of the following finitely many invariant sub-manifolds:
\begin{align*}
	\mathcal{I}_{a_i', b_j', c_k'}:=\big\{ (x,y) \in \mathcal{O} \times \mathcal{O}: & 2(y_1-x_1)=a_i',\, (x_2+y_2)/2=b_j',\,(x_3+y_3)/2=c_k',\, \\
	& f_2(y_1+t) \equiv -f_2(x_1+t) \text{ for all } t \in \mathbb{R} \big\}.
\end{align*}
Similarly, there exist finite constants $a_i'', b_j'', c_k''$ such that $\mathcal{I}_3$ can be characterized as the union of the following finitely many invariant sub-manifolds:
\begin{align*}
	\mathcal{I}_{a_i'', b_j'', c_k''}:=\big\{ (x,y) \in \mathcal{O} \times \mathcal{O}: & 2(y_3-x_3)=c_k'',\, (x_1+y_1)/2=a_i'',\,(x_2+y_2)/2=b_j'' ,\, \\
	& f_1(y_3+t) \equiv -f_1(x_3+t) \text{ for all } t \in \mathbb{R}\big\}.
\end{align*} 
And there exist finite constants $a_i''', b_j''', c_k'''$ such that $\mathcal{I}_4$ can be characterized as the union of the following finitely many invariant sub-manifolds:
\begin{align*}
	\mathcal{I}_{a_i''', b_j''', c_k'''}:=\big\{ (x,y) \in \mathcal{O} \times \mathcal{O}: 
	& 2(y_2-x_2)=b_j''',\, (x_3+y_3)/2=c_k''',\,(x_1+y_1)/2=a_i''',\, \\
	& f_3(y_2+t) \equiv -f_3(x_2+t) \text{ for all } t \in \mathbb{R}\big\}.
\end{align*}

Next, we prove the Lyapunov-Foster drift condition for the two-point chain $\{\tilde{\Phi}_{\underline{\tau}}^m\}$. Without
loss of generality, we only need to consider $I:= \cup_{i=0}^4 I_i$ where $I_0=\left\{ (x,y) \in \mathcal{O} \times \mathcal{O}: x=y \right\}$, $I_1=\mathcal{I}_{a_1,b_1,c_1}$, $I_2=\mathcal{I}_{a_1',b_1',c_1'}$, $I_3=\mathcal{I}_{a_1'',b_1'',c_1''}$ and $I_4=\mathcal{I}_{a_1''',b_1''',c_1'''}$.

For $s>0$, we define 
\[
\Delta_i(s):=\left\{(x,y)\in\mathcal{T}: d(\tilde{x}(i),y)<s\right\},\ \ i=0, 1, 2, 3,4
\]
where $x=(x_1,x_2,x_3)$, $\tilde{x}(0)=x$, $\tilde{x}(1)=(a_1+x_1, b_1+x_2, c_1+x_3) \mod 2\pi\mathbb{Z}^3$, $\tilde{x}(2)=(a_1'/2+x_1, 2b_1'-x_2, 2c_1'-x_3) \mod 2\pi\mathbb{Z}^3$,  $\tilde{x}(3)=(2a_1''-x_1, 2b_1''-x_2, c_1''/2+x_3) \mod 2\pi\mathbb{Z}^3$ and $\tilde{x}(4)=(2a_1'''-x_1, b_1'''/2+x_2, 2c_1'''-x_3) \mod 2\pi\mathbb{Z}^3$. Let $s_0>0$ be the minimal injectivity radius of the exponential map $\exp_x$ for $x \in \mathbb{T}^3$. Hence, if $d(x,y) < s_0$, $w(x,y):=\exp_{x}^{-1}(y)$ is well-defined. Let $\hat{w}(x,y)$ be the normalization of $w(x,y)$.

Due to Lemma \ref{derivative of r(q)} and Lemma \ref{Lie bracket of lifted chain}, fix $q \in (0,q_0)$ such that $r(q)<1$. For $(x,y) \in \Delta_0(s)$, with $s>0$ to be determined, define
\begin{equation*}
	V_0(x,y)=d(x,y)^{-q} \psi_q(x, \hat{w}(x,y)).
\end{equation*}
We introduce function $\hat{V}: T\mathbb{T}^3\to\mathbb{R}$ defined by
\begin{equation*}
	\hat{V}(x,v):=|v|^{-q} \psi_q (x, {v}/{|v|}).
\end{equation*}
Recall that $\hat{P}$ is the transition kernel of Markov chain $\{\hat{\Phi}_{\underline{\tau}}^m\}$. We can verify that $\hat{V}(x,v)$ is the eigenfunction of operator $\hat{P}$ corresponding to the eigenvalue $r(q)$ as follows: 
\begin{align}\label{hat V as eigenfunction}
	\hat{P} \hat{V}(x,v) &= \mathbb{E} \hat{V}\left(\Phi_{\underline{\tau}}(x), D_x\Phi_{\underline{\tau}} v\right) = \mathbb{E} |D_x\Phi_{\underline{\tau}} v|^{-q} \psi_q \left( \Phi_{\underline{\tau}}(x), {D_x\Phi_{\underline{\tau}} v}/{|D_x\Phi_{\underline{\tau}} v|}\right)\nonumber \\
	&= |v|^{-q} \check{P}_q \psi_q (x,v/|v|) =r(q) |v|^{-q} \psi_q (x,v/|v|) =r(q) \hat{V}(x,v).
\end{align}

Recall that $\{\tau_i\}_{i=1}^\infty$ is a sequence of independent exponential random variables with mean $h>0$. In the following lemma, we approximate $P^{(2)}V_0(x,y)$ by $\hat{P} \hat{V}(x,w(x,y))$ for all $(x,y)\in\Delta_0(s)$.
\begin{lem}\label{approximate for V}
Let $\hat{P}$, $P^{(2)}$, $q$ and $h$ be as above. Let $T>1$ be sufficiently large. Then, for any $\varepsilon>0$, there exists a constant $C_{\varepsilon}>0$ such that 
\begin{align*}
&\left| P^{(2)}V_0(x,y)- \hat{P} \hat{V}(x,w(x,y))\right| \\
\lesssim &\left( \varepsilon d(x,y)^{-q} + C_{\varepsilon}(T+h+1)e^{-\frac{T}{h}} d(x,y)^{-q} +C_\varepsilon T^{3q+8}d(x,y)^{1-q}\right),
\end{align*}
for all $(x,y) \in \Delta_0(s_0/(C_0 T^5))$, uniformly in $\varepsilon$ and $T$. Here, $s_0$ is the minimal injectivity radius of the exponential map $\exp_x$ for $x \in \mathbb{T}^3$, $C_0>1$ is a constant.
\end{lem}
\begin{proof}
	By Lemma \ref{property of psi}(i), for any $\varepsilon>0$ and $(x,y) \in \Delta_0(s_0/(C_0T^5))$, we have 
	\begin{align*}
		P^{(2)}V_0(x,y) &= \mathbb{E} d\left( \Phi_{\underline{\tau}}(x), \Phi_{\underline{\tau}}(y) \right)^{-q} \psi_q \left( \Phi_{\underline{\tau}}(x),\hat{w}(\Phi_{\underline{\tau}}(x),\Phi_{\underline{\tau}}(y))\right)  \nonumber \\
		&= \mathbb{E}  d\left( \Phi_{\underline{\tau}}(x), \Phi_{\underline{\tau}}(y) \right)^{-q} \psi_{q,1} \left( \Phi_{\underline{\tau}}(x),\hat{w}(\Phi_{\underline{\tau}}(x),\Phi_{\underline{\tau}}(y))\right)   + O \left( \varepsilon d(x,y)^{-q} \right),
	\end{align*}   
	and 
	\begin{align*}
		&~~~\hat{P} \hat{V}(x,(w(x,y))) \nonumber\\
		&= \mathbb{E}  |D_x\Phi_{\underline{\tau}} (w(x,y))|^{-q} \psi_q \left( \Phi_{\underline{\tau}}(x), D_x\Phi_{\underline{\tau}} (w(x,y))/|D_x\Phi_{\underline{\tau}} (w(x,y))|\right)  \nonumber \\
		&= \mathbb{E}  |D_x\Phi_{\underline{\tau}} (w(x,y))|^{-q} \psi_{q,1} \left( \Phi_{\underline{\tau}}(x), D_x\Phi_{\underline{\tau}} (w(x,y))/|D_x\Phi_{\underline{\tau}} (w(x,y))|\right) + O \left( \varepsilon d(x,y)^{-q} \right).
	\end{align*}
	Let $T>1$ be sufficiently large. For any $\underline{\tau} \in [0,T]^3$, there exists a constant $C_0>1$, independent of $T$, such that $\|\Phi_{\underline{\tau}}\|_{C^2} \leq C_0 T^5$. Then, by the first order Taylor expansion with Lagrange remainder term on $\mathbb{T}^3$, we obtain
	\begin{equation*}
		d \left( \Phi_{\underline{\tau}}(y), \exp_{\Phi_{\underline{\tau}}(x)}  D_{x} \Phi_{\underline{\tau}} (w(x,y)) \right) \leq C_0 T^5 d(x,y)^2.
	\end{equation*}
	for all $(x,y) \in \Delta_0(s_0/(C_0 T^5))$, $s_0$ is the minimal injectivity radius of the exponential map $\exp_x$ for $x \in \mathbb{T}^3$. It follows that
	\begin{equation*} 
		\left| w(\Phi_{\underline{\tau}}(x),\Phi_{\underline{\tau}}(y))- D_x\Phi_{\underline{\tau}} (w(x,y)) \right| \leq C_0 T^5 d(x,y)^2,
	\end{equation*}
	Using the mean value theorem, we deduce
	\begin{equation*}
		\left| |w(\Phi_{\underline{\tau}}(x),\Phi_{\underline{\tau}}(y))|^{-q}- |D_x\Phi_{\underline{\tau}} (w(x,y))|^{-q} \right| \leq K T^{3q+8} d(x,y)^{1-q},
	\end{equation*}
	where $K>1$ may differ from line to line. However, it is a constant that depends only on $h$ and $q$, and is independent of $T$ and $\varepsilon$. For any vectors $u_1$ and $u_2$, we have
	\begin{equation*}
		\left| \frac{u_1}{|u_1|}-\frac{u_2}{|u_2|} \right| \leq 2 \frac{|u_1-u_2|}{\min\{|u_1|,|u_2|\}}.
	\end{equation*}
	Applying this to $u_1=w(\Phi_{\underline{\tau}}(x),\Phi_{\underline{\tau}}(y))$ and $u_2=D_x\Phi_{\underline{\tau}} (w(x,y))$, we obtain
	\begin{equation*}
		\left| \hat{w}(\Phi_{\underline{\tau}}(x),\Phi_{\underline{\tau}}(y))- \frac{D_x\Phi_{\underline{\tau}} (w(x,y))}{\left|D_x\Phi_{\underline{\tau}} (w(x,y))\right|} \right| \leq K T^{8} d(x,y).
	\end{equation*}
	For any $\underline{\tau} \in [0,T]^3$ and $(x,y) \in \Delta_0(s_0/(C_0 T^5))$, combining these estimates with Lemma \ref{property of psi}(i), there exists $C_{\varepsilon}>1$ such that
	\begin{align*}
		&\left|d\left( \Phi_{\underline{\tau}}(x), \Phi_{\underline{\tau}}(y) \right)^{-q} \psi_{q,1} \left( \Phi_{\underline{\tau}}(x),\hat{w}(\Phi_{\underline{\tau}}(x),\Phi_{\underline{\tau}}(y))\right)- |D_x\Phi_{\underline{\tau}} w|^{-q} \psi_{q,1} \left( \Phi_{\underline{\tau}}(x), D_x\Phi_{\underline{\tau}} w/|D_x\Phi_{\underline{\tau}} w|\right)  \right| \nonumber \\
		\leq&\left|( d\left( \Phi_{\underline{\tau}}(x), \Phi_{\underline{\tau}}(y) \right)^{-q} -|D_x\Phi_{\underline{\tau}} w|^{-q})\psi_{q,1} \left( \Phi_{\underline{\tau}}(x),\hat{w}(\Phi_{\underline{\tau}}(x),\Phi_{\underline{\tau}}(y))\right)\right|\nonumber\\
		&+\left| |D_x\Phi_{\underline{\tau}} w|^{-q}(\psi_{q,1} \left( \Phi_{\underline{\tau}}(x),\hat{w}(\Phi_{\underline{\tau}}(x),\Phi_{\underline{\tau}}(y))\right)-\psi_{q,1} \left( \Phi_{\underline{\tau}}(x), D_x\Phi_{\underline{\tau}} w/|D_x\Phi_{\underline{\tau}} w|\right))\right|\nonumber\\
		\leq &K T^{3q+8}d(x,y)^{1-q}(\|\psi_{q,1}\|_{C^0}+\|D\psi_{q,1}\|_{C^0})\nonumber\\
		\leq&KC_\varepsilon T^{3q+8}d(x,y)^{1-q}.
	\end{align*}
    Denote $D_{T}:=(T,\infty) \times [0,\infty)^2\cup [0,\infty) \times (T,\infty) \times [0,\infty)\cup [0,\infty)^2\times (T,\infty)$. For any $\underline{\tau} \in D_T$, we first estimate the contribution of the term involving 
    $d\left(\Phi_{t}(x), \Phi_{t}(y)\right)^{-q} \psi_{q,1}$. By a direct calculation, we obtain:
	\begin{align*}
		&~~~ \frac{1}{h^3} \int_{D_T} d\left( \Phi_{t}(x), \Phi_{t}(y) \right)^{-q} \left| \psi_{q,1} \left( \Phi_{t}(x),\hat{w}(\Phi_{t}(x),\Phi_{t}(y))\right) \right|  e^{-\frac{t_1+t_2+t_3}{h}} dt \nonumber \\
		&\leq K C_{\varepsilon}(T+h+1)e^{-\frac{T}{h}} d(x,y)^{-q},
	\end{align*}
	where $t=(t_1,t_2,t_3)$, $dt=dt_1dt_2dt_3$. 
	
	Similarly, for the corresponding term involving 
	$\left|D_x\Phi_{\underline{\tau}}(w(x, y))\right|^{-q} \psi_{q,1}$, we have:
	\begin{align*}
		&~~~ \frac{1}{h^3} \int_{D_T} |D_x\Phi_{\underline{\tau}} (w(x,y))|^{-q} \left|\psi_{q,1} \left( \Phi_{\underline{\tau}}(x), D_x\Phi_{\underline{\tau}} (w(x,y))/|D_x\Phi_{\underline{\tau}} (w(x,y))|\right)\right| e^{-\frac{t_1+t_2+t_3}{h}} dt \nonumber\\
		&\leq K C_{\varepsilon}(T+h+1)e^{-\frac{T}{h}} d(x,y)^{-q}.
	\end{align*}
	Combining the above estimates, we deduce
	\begin{align*}
		&\left| P^{(2)}V_0(x,y)-  \hat{P} \hat{V}(x,(w(x,y))) \right| \\
		\leq & K\left( \varepsilon d(x,y)^{-q} + C_{\varepsilon}(T+h+1)e^{-\frac{T}{h}} d(x,y)^{-q} +C_\varepsilon T^{3q+8}d(x,y)^{1-q}\right).
	\end{align*}
    This completes the proof.
\end{proof}
Next, we show that $V_0$ satisfies the drift condition for $P^{(2)}$ on $\Delta_0(s)$ with $s > 0$ chosen sufficiently small.
\begin{lem}\label{diagonal drift condition}
    Let $P^{(2)}$, $q$ be as above. Let $q>0$ be sufficiently small such that $r(q)<1$. Then, there exist $\gamma\in(0,1)$ and $s>0$ such that 
\begin{equation*}
    P^{(2)}V_0(x,y) < \gamma V_0(x,y),
\end{equation*}
for all $(x,y)\in \Delta_0(s)$. 
\end{lem}
\begin{proof}
	Let $\varepsilon>0$ be sufficiently small, with constraints specified later. For any $(x,y)\in \Delta_0(s)$, where $s<s_0/(C_0T^5)$ (to be determined later), it follows from Lemma \ref{approximate for V} and \eqref{hat V as eigenfunction} that
\begin{align*}
P^{(2)}V_0(x,y) &\leq \hat{P}\hat{V}(x,w)+ K\left( \varepsilon d(x,y)^{-q} + C_{\varepsilon}(T+h+1)e^{-\frac{T}{h}} d(x,y)^{-q} +C_\varepsilon T^{3q+8}d(x,y)^{1-q}\right) \\
&\leq d(x,y)^{-q} \left( r(q)\psi_{q}(x,\hat{w}(x,y))+ K\varepsilon +KC_{\varepsilon}(T+h+1)e^{-\frac{T}{h}}+KC_\varepsilon T^{3q+8}s\right)
\end{align*}
where $K>0$ is a constant depending only on $h$ and $q$, but independent of $T$ and $\varepsilon$. We now specify the constants $\varepsilon$, $T$ and $s$. By Lemma \ref{property of psi}(ii) and the compactness of $S\mathbb{T}^3$, we know that $\inf_{(x,v) \in S\mathbb{T}^3} \psi_q(x, v) > 0$. Choose $\varepsilon>0$ to be sufficiently small, satisfying
\begin{equation*}
\varepsilon<\frac{1-r(q)}{300K}\inf_{(x,v)\in S\mathbb{T}^3}\psi_q(x,v).
\end{equation*}
Since $\lim_{t\to\infty}(t+h+1)e^{-\frac{t}{h}}=0$, we select $T>1$ sufficiently large such that 
\[
(T+h+1)e^{-\frac{T}{h}}<\frac{\varepsilon}{C_{\varepsilon}}.
\] 
Next, set $s>0$ small enough to satisfy
\[
s<\frac{\varepsilon}{C_{\varepsilon}T^{3q+8}}. 
\]
Substituting these choices of $\varepsilon$, $T$ and $s$ into the bound, we deduce
\begin{equation*}
P^{(2)}V_0(x,y)< d(x,y)^{-q}\left(r(q)+\frac{1-r(q)}{100}\right)\psi_q(x,\hat{w}(x,y)),
\end{equation*}
for $(x,y)\in \Delta_0(s)$. Setting $\gamma=r(q)+\frac{1-r(q)}{100}$, we note that $\gamma<1$, which completes the proof.
\end{proof}
Now, we are ready to prove the Lyapunov-Foster drift condition for the two-point Markov process $\{\tilde{\Phi}_{\underline{\tau}}^m\}$. 
\begin{pro}\label{drift condition for two-point process}
	Let $\{\tilde{\Phi}_{\underline{\tau}}^m\}$ be as above. Then, there exists an integrable function $V: \mathcal{T}\to [1,\infty)$ which satisfies the Lyapunov-Foster drift condition for $\{\tilde{\Phi}_{\underline{\tau}}^m\}$.
\end{pro}
\begin{proof}
	Let $s$ be the value specified in Lemma \ref{diagonal drift condition}. For any $(x,y) \in \Delta_i(s)$, $i=1, 2, 3, 4$, define
	\begin{equation*}
		V_i(x,y)= V_0(\tilde{x}(i),y). 
	\end{equation*}
    From Lemma \ref{diagonal drift condition}, it follows that
	\begin{equation*}
		P^{(2)}V_i(x,y)=P^{(2)}V_0(\tilde{x}(i),y)< \gamma V_0(\tilde{x}(i),y) =\gamma V_i(x,y),
	\end{equation*}
    for all $(x,y) \in \Delta_i(s)$. Set $\Delta(s)=\cup_{i=0}^4\Delta_i(s)$. Let $s_0\in (0,s)$ such that $\max_{i}V_i(x,y)\geq 1$ for all $(x,y)\in \Delta(s_0)$. This is possible since $\inf_{(x,v) \in S\mathbb{T}^3} \psi_q(x, v) > 0$. Define the drift function $V: \mathcal{T}\to [1,\infty)$ by
	\begin{equation*}
		V(x,y)= \chi_{\Delta(s)}(x,y) \max_{i}V_i(x,y) + b\chi_{\mathcal{T}\setminus \Delta(s)}(x,y),
	\end{equation*}
	where $\chi_D$ is a characteristic function on $D$ and $b\geq 1$ is a constant. Directly calculating, we obtain
	\begin{equation*}
		P^{(2)}V(x,y)\leq \gamma V(x,y)+ b\chi_{\mathcal{T}\setminus \Delta(s)}(x,y).
	\end{equation*}
\end{proof}
We are now prepared to prove the main theorem of this section, which establishes the $V$-uniform geometric ergodicity for the two-point chain.
 \begin{thm}[\textbf{$V$-uniform geometric ergodicity for the two-point chain}]\label{Uniformly geometrically ergodic for two-point chain}
	Let $\{\tilde{\Phi}_{\underline{\tau}}^m\}$ be as above. Then there exists an integrable function $V:\mathcal{T}\to [1,\infty)$ such that the transition kernel $P^{(2)}$ of $\{\tilde{\Phi}_{\underline{\tau}}^m\}$ is $V$-uniformly geometrically ergodic.
\end{thm}
\begin{proof}
    It is sufficient to apply Theorem \ref{conditions for uniformly geometrically ergodic} to $P^{(2)}$. The existence of an open small set can be deduced from Lemma \ref{Lie bracket of two-point} and Lemma \ref{small set}. The topological irreducibility property is established via Lemma \ref{irreducible for two-point chain}. Strong aperiodicity follows immediately from $\mathbb{P}(\tau<\varepsilon)>0$ for any $\varepsilon>0$ and $\tilde{\Phi}_0(x,y)=(x,y)$.  Finally, the Lyapunov-Foster drift condition is established with an integrable function $V$ in Proposition \ref{drift condition for two-point process}. Consequently, the transition kernel $\check{P}$ of $\{\check{\Phi}_{\underline{\tau}}^m\}$ is $V$-uniformly geometrically ergodic.
\end{proof}

\section{Quenched correlation decay}\label{section 7}
\subsection{Proof of Theorem \ref{Exponential mixing}}
\quad Having established the geometric ergodicity of the two-point Markov chain, we are now ready to prove Theorem \ref{Exponential mixing}. While the proof is essentially covered in \cite{BBP}, we include it here for the sake of completeness. Assume (H1) holds, by Theorem \ref{Uniformly geometrically ergodic}, the Lebesgue probability measure on $\mathbb{T}^3$ is the unique stationary measure for $\{\Phi_{\underline{\tau}}^m\}$. For sufficiently regular observables $f, g$ with mean zero, denote the following correlation term:
\[
\text{Cor}_m(f,g):=\left| \int_{\mathbb{T}^3} f(x) g \left(\Phi_{\underline{\tau}}^m(x) \right) dx\right|.
\]
Let $\{e_k\}_{k \in \mathbb{Z}^3}$ be the Fourier basis of $L^2(\mathbb{T}^3)$, that is, $e_k(x)=e^{ik\cdot x}$. Denote $\mathbb{Z}_0^3=\mathbb{Z}^3\setminus\{0\}$.

 \begin{lem}\label{Exponential mixing with 5/2}
	Let $\{\Phi_{\underline{\tau}}^m\}$ be as above. Then, for any $q>0, s'>5/2$, there exists a random variable $\xi=\xi_{q, s'}\geq 1$ and a constant $c=c_{q, s'}>0$ such that for all mean-zero function $f, g\in H^{s'}(\mathbb{T}^3)$, we have the almost sure estimate 
	\begin{equation*}
		\left|\int_{\mathbb{T}^3} f(x)g\left(\Phi_{\underline{\tau}}^m(x)\right)dx\right|\leq \xi(\underline{\tau})e^{-c m}\|f\|_{H^{s'}}\|g\|_{H^{s'}},
	\end{equation*}
	while the random variable $\xi$ satisfies the moment bound $\mathbb{E}[\xi^q]<\infty$.
\end{lem}
\begin{proof}
	For $k_1, k_2\in \mathbb{Z}_0^3$ and $\varepsilon>0$ by Chebyshev inequality, we derive
	\begin{align*}
		\mathbb{P}\left( \text{Cor}_m(e_{k_1},e_{k_2}) > \varepsilon \right) &\leq \varepsilon^{-2} \mathbb{E} \left| \int_{\mathbb{T}^3} e_{k_1}(x) e_{k_2} \left(\Phi_{\underline{\tau}}^m(x) \right) dx\right|^2 \\
		&=\varepsilon^{-2} \mathbb{E} \left| \int_{\mathbb{T}^3 \times \mathbb{T}^3} \tilde{e}_{k_1}(x,y) \tilde{e}_{k_2} \left(\Phi_{\underline{\tau}}^m(x),\Phi_{\underline{\tau}}^n(y) \right) dx dy \right| \\
		&=\varepsilon^{-2} \left| \int_{\mathbb{T}^3 \times \mathbb{T}^3} \tilde{e}_{k_1}(x,y) \left( P^{(2)} \right)^m \left( \tilde{e}_{k_2} (x,y) \right) dx dy \right|,
	\end{align*}
	where $\tilde{e}_{k_i}(x,y)=e_{k_i}(x)e_{k_i}(y), i=1,2$. By Theorem \ref{Uniformly geometrically ergodic for two-point chain}, $P^{(2)}$ is $V$-uniformly geometrically ergodic, there exists constants $C>0$ and $\gamma \in (0,1)$ such that for any $f \in \mathcal{M}_V(\mathbb{T}^3 \times \mathbb{T}^3)$, $m \geq 1$, 
	\begin{equation*}
		\left| \left( P^{(2)} \right)^m f(x,y) - \int_{\mathbb{T}^3 \times \mathbb{T}^3} f(x,y)dxdy \right| \leq C \gamma^m V(x,y) \|f\|_V .
	\end{equation*}
	Combining with $ \tilde{e}_{k_2}$ is mean-zero, we have
	\begin{equation}\label{Cor estimate}
		\mathbb{P}\left( \text{Cor}_m(e_{k_1},e_{k_2}) > \varepsilon \right) \leq \varepsilon^{-2} \cdot C \gamma^m  \|V\|_{L^1(\mathbb{T}^3\times \mathbb{T}^3)} \lesssim \varepsilon^{-2} \gamma^m.
	\end{equation}
	Let $\varepsilon=\varepsilon_m=\gamma^{\alpha m}$ with $\alpha\in (0,1/2)$ to be determined. By Borel-Cantelli lemma we have 
	\[
	\mathbb{P}\left( \limsup_{m \to \infty}\left\{ \text{Cor}_m(e_{k_1},e_{k_2}) > \varepsilon_m \right\} \right)=0.
	\] 
	Hence, for $k_1, k_2\in \mathbb{Z}_0^3$, the random variable
	\[
	N_{k_1,k_2}:=\max\left\{m\geq 0: \text{Cor}_m(e_{k_1},e_{k_2})>\gamma^{\alpha m} \right\}
	\]
	is finite with probability $1$. By \eqref{Cor estimate}, the tail estimate
	\begin{equation}\label{tail estimate}
		\mathbb{P} \left(  N_{k_1,k_2} > n \right) \leq \sum_{m \geq n +1} \mathbb{P}\left( \text{Cor}_m(e_{k_1},e_{k_2}) > \varepsilon_m \right) \lesssim \gamma^{(1-2\alpha)n}
	\end{equation}
    holds uniformly in $k_1, k_2$. Since $\text{Cor}_m(e_{k_1},e_{k_2}) \leq 1$, let $D\geq \alpha \ln (\gamma^{-1})$ to be determined. Then, for any $m \in \mathbb{N}$, we have the almost sure estimate
	\begin{equation*}
		\left| \int_{\mathbb{T}^3} e_{k_1}(x) e_{k_2} \left(\Phi_{\underline{\tau}}^m(x) \right) dx\right| \leq e^{D N_{k_1,k_2}(\underline{\tau})} \gamma^{\alpha m}. 
	\end{equation*}
	Let $f, g\in L^2(\mathbb{T}^3)$ be regular enough to be determined with zero mean value. We have Fourier expansions $f=\sum_{k \in \mathbb{Z}_0^3}f_k e_k$, $g=\sum_{k \in \mathbb{Z}_0^3}g_k e_k$. Then, we have the almost sure estimate
	\begin{equation*}
		\left| \int_{\mathbb{T}^3} f(x) g \left(\Phi_{\underline{\tau}}^m(x) \right) dx\right| \leq \gamma^{\alpha m} \sum_{k_1,k_2 \in \mathbb{Z}_0^3} |f_{k_1}| |g_{k_2}| e^{D N_{k_1,k_2}(\underline{\tau})} .
	\end{equation*}
	Define the random variable 
	\begin{equation*}
		K(\underline{\tau})=\max \left\{ |k_1|\vee|k_2|: e^{D N_{k_1,k_2}(\underline{\tau})} > |k_1||k_2| \right\},
	\end{equation*}
    where $a\vee b=\max\{a, b\}$. By \eqref{Cor estimate}, we have
	\begin{align*}
		\mathbb{P} \left( e^{D N_{k_1,k_2}} > |k_1||k_2| \right) &= \mathbb{P} \left(  N_{k_1,k_2} > \frac{1}{D} \ln ( |k_1||k_2|) \right) \\
		&\leq \sum_{m \geq \lfloor \frac{1}{D} \ln ( |k_1||k_2|)\rfloor +1} \mathbb{P}\left( \text{Cor}_m(e_{k_1},e_{k_2}) > \varepsilon_m \right) \\
		&\lesssim \sum_{m \geq \lfloor \frac{1}{D} \ln ( |k_1||k_2|)\rfloor+1} \gamma^{(1-2\alpha)m} \lesssim \gamma^{(1-2\alpha) \cdot \frac{1}{D} \ln ( |k_1||k_2|) }.
	\end{align*}
	where $\lfloor a \rfloor$ denotes the largest integer smaller than $a$. Let $\alpha<1/5$, $D<(1-2\alpha)\ln (\gamma^{-1})/3$. Define $\beta=(2\alpha-1)\ln (\gamma^{-1})/D$, so that $\beta<-3$.
	Then, for any positive integer $l$, 
	\begin{align}\label{K}
		\mathbb{P}(K >l ) &\leq 2 \sum_{k_1,k_2 \in \mathbb{Z}_0^3,|k_1|>l}  \mathbb{P} \left( e^{D N_{k_1,k_2}} >  |k_1||k_2| \right) \nonumber\\
		& \lesssim \sum_{k_1,k_2 \in \mathbb{Z}_0^3,|k_1|>l} \gamma^{(1-2\alpha) \cdot \frac{1}{D} \ln ( |k_1||k_2|) } \nonumber\\
		&\lesssim \sum_{k_1 \in \mathbb{Z}_0^3,|k_1|>l} |k_1|^{\beta}  \sum_{k_2 \in \mathbb{Z}_0^3} |k_2|^{\beta}\nonumber \\
		&\lesssim \sum_{k_1 \in \mathbb{Z}_0^3,|k_1|>l} |k_1|^{\beta} \lesssim l^{3+\beta}.
	\end{align}
	Therefore, $K(\underline{\tau})$ is almost surely finite, and so is
	\begin{equation*}
		\xi(\underline{\tau})=\max_{|k_1| \vee |k_2| \leq K} e^{D N_{k_1,k_2}(\underline{\tau})}.
	\end{equation*}
    Let $s'>5/2$. For any $f, g\in H^{s'}$, we obtain
	\begin{align*}
		\left| \int_{\mathbb{T}^3} f(x) g \left(\Phi_{\underline{\tau}}^m(x) \right) dx\right| &\leq \xi(\underline{\tau}) \gamma^{\alpha m} \left( \sum_{k_1 \in \mathbb{Z}_0^3} |k_1| |f_{k_1}| \right) \left( \sum_{k_2 \in \mathbb{Z}_0^3} |k_2| |g_{k_2}| \right) \\
		&\leq \xi(\underline{\tau}) \gamma^{\alpha m} \left(\sum_{k\in \mathbb{Z}_0^3} |k|^{2-2s'}\right) \left(\sum_{k\in \mathbb{Z}_0^3} |k|^{2s'}|f_k|^2 \right)^{\frac{1}{2}} \left(\sum_{k\in \mathbb{Z}_0^3} |k|^{2s'}|g_k|^2 \right)^{\frac{1}{2}} \\
		& \lesssim \xi(\underline{\tau}) \gamma^{\alpha m} \|f\|_{H^{s'}} \|g\|_{H^{s'}}.
	\end{align*}
    Now we prove for any $q>0$, $\mathbb{E}\xi^q < \infty$. Let $\alpha<1/5$ sufficiently small and $D=\alpha \ln (\gamma^{-1})$. It follows that $\gamma^{1-2\alpha}e^{2qD} <1$. By \eqref{tail estimate}, the estimate
    \begin{align*}
    	\mathbb{E}e^{2qD N_{k_1,k_2}}=\sum_{n \geq 1} \mathbb{P}(N_{k_1,k_2}=n) e^{2qDn} \lesssim \sum_{n \geq 1} (\gamma^{1-2\alpha} e^{2qD})^n,
    \end{align*}
    holds uniformly in $k_1, k_2$. Therefore, by \eqref{K} and Cauchy-Schwarz inequality, we have 
    \begin{align*}
    	\mathbb{E}[\xi^q] &= \sum_{l \geq 1} \mathbb{E} \left(1_{\{K=l\}} \max_{|k_1|,|k_2|\leq l} e^{qD N_{k_1,k_2}} \right) \\
    	&\leq \sum_{l \geq 1} [\mathbb{P}(K=l)]^{\frac{1}{2}} \sum_{|k_1|,|k_2|\leq l} [\mathbb{E}e^{2qD N_{k_1,k_2}}]^{\frac{1}{2}}  \\
    	&\lesssim \sum_{l \geq 1} l^{\frac{3+\beta}{2}} \cdot l^6=\sum_{l \geq 1} l^{\frac{17\alpha-1}{2\alpha}}.
    \end{align*}
    This shows that the moment $\mathbb{E}[\xi^q]$ is finite when $\alpha<1/19$ is sufficiently small.
\end{proof}
To transition from $H^{s'}$ with $s' > \frac{5}{2}$ to $H^s$ for arbitrary $s > 0$, we adjust the exponential decay rate to account for the reduced regularity of the functions $f$ and $g$. This adjustment is facilitated by the standard approximation lemma from Lemma 7.1 in \cite{BBP} and Lemma 4.2 in \cite{CDE}.
\begin{lem}\label{general Hs approximation}
	Let $0<s<s'$ and let $f \in H^s$ be a mean-zero function. Then, for any $\varepsilon>0$ there exists a mean-zero function $f_{\varepsilon} \in H^{s'}$ such that  $\|f_\varepsilon\|_{L^2} \lesssim \|f\|_{L^2}$, $\|f_\varepsilon-f\|_{L^2} \lesssim \varepsilon \|f\|_{H^s}$ and $\|f_\varepsilon\|_{H^{s'}} \lesssim \varepsilon^{-\frac{s'-s}{s}} \|f\|_{H^s}$.
\end{lem}
\begin{proof}
	Consider $f \in H^s(\mathbb{T}^3)$, a mean-zero function, and define $h : [0, \infty) \times \mathbb{T}^3 \to \mathbb{R}$ as the solution to the fractional heat equation with initial condition $f$. The equation is given by:
	\[
	\begin{cases}
		\partial_t h(t,x) =- (-\Delta)^s h(t,x), &\text{on } [0, \infty) \times \mathbb{T}^3, \\
		h(0,x) = f(x), & \text{on } \mathbb{T}^3.
	\end{cases}
	\]
	Let $\lambda_k=|k|^2$ denote the eigenvalue of $-\Delta$ associated with the eigenfunction $e_{k}$, where $k \in \mathbb{Z}^3$. Since $f\in H^s(\mathbb{T}^3)$ is a mean-zero function, we can express it in terms of its Fourier series: $f(x)=\sum_{k \in \mathbb{Z}_0^3} f_k e_k(x)$. The solution $h(t,x)$ to the fractional heat equation can thus be written in terms of its Fourier series as
	\[
	h(t,x)=\sum_{k \in \mathbb{Z}_0^3} e^{-t\lambda_k^s} f_k e_k(x).
	\]
	For any $\varepsilon>0$, define
	\[
	f_{\varepsilon}(x):=h(\varepsilon^2,x) = \sum_{k \in \mathbb{Z}_0^3} e^{-\varepsilon^2 \lambda_k^s} f_ke_k(x).
	\]
	It is straightforward to confirm that $f_\varepsilon$ is mean-zero. Accordingly, we deduce the following:
	\[
	\|f_\varepsilon\|_{L^2}^2 = \sum_{k \in \mathbb{Z}_0^3} e^{-2 \varepsilon^2 \lambda_k^s} |f_k|^2 \lesssim \sum_{k \in \mathbb{Z}_0^3} |f_k|^2 = \|f\|_{L^2}^2.
	\]
	On the other hand, for any $s'>s$, 
	\[
	\|f_\varepsilon\|_{H^{s'}}^2 = \sum_{k \in \mathbb{Z}_0^3} \lambda_k^{s'} e^{-2 \varepsilon^2 \lambda_k^s} |f_k|^2 \leq \sup_{x>0} \left( x^{s'-s} e^{-2 \varepsilon^2 x^s} \right) \cdot \sum_{k \in \mathbb{Z}_0^3} \lambda_k^{s}  |f_k|^2 \lesssim_{s,s'} \varepsilon^{-\frac{2(s'-s)}{s}} \|f\|_{H^{s}}^2.
	\]
	Finally, we compute how well $f_{\varepsilon}$ approximates $f$:
	\[
	\|f_\varepsilon-f \|_{L^2}^2 = \sum_{k \in \mathbb{Z}_0^3} \left(1-e^{- \varepsilon^2 \lambda_k^s} \right)^2 |f_k|^2 \leq \sup_{x>0} \frac{(1-e^{-x})^2}{x} \cdot \varepsilon^2 \sum_{k \in \mathbb{Z}_0^3} \lambda_k^{s}  |f_k|^2 \lesssim \varepsilon^2 \|f\|_{H^{s}}^2.
	\] 
	Hence, the conclusion follows. 
\end{proof}
With the preceding steps in place, we can now prove the main theorem on almost-sure quenched correlation decay.
\begin{proof}[\textbf{Proof of Theorem \ref{Exponential mixing}}]
	By Lemma \ref{Exponential mixing with 5/2}, we only need to consider the case where $s\in (0,3)$. For any mean-zero function $f\in H^s$ and any $\varepsilon>0$, by Lemma \ref{general Hs approximation}, there exists $f_\varepsilon\in H^{3}$ such that 
	\[
	\|f_\varepsilon\|_{L^2} \lesssim \|f\|_{L^2}, \quad  \|f_\varepsilon-f\|_{L^2} \lesssim \varepsilon \|f\|_{H^s}, \quad  \|f_\varepsilon\|_{H^{3}} \lesssim \varepsilon^{-\frac{3-s}{s}} \|f\|_{H^s} .
	\]
	A similar statement holds for any mean-zero function $g\in H^s$. Then, it follows from Lemma \ref{Exponential mixing with 5/2} that
	\begin{align*}
		\left| \int_{\mathbb{T}^3} f(x) g \left(\Phi_{\underline{\tau}}^m(x) \right) dx\right| &\leq \left| \int_{\mathbb{T}^3} f_{\varepsilon}(x) g_{\varepsilon} \left(\Phi_{\underline{\tau}}^m(x) \right) dx\right| +\|g\|_{L^2}\|f_\varepsilon-f\|_{L^2} + \|f_{\varepsilon}\|_{L^2}\|g_\varepsilon-g\|_{L^2} \\
		&\lesssim \xi(\underline{\tau}) \gamma^{\alpha m} \|f_{\varepsilon}\|_{H^{3}} \|g_{\varepsilon}\|_{H^{3}} + \varepsilon \|f\|_{H^s} \|g\|_{H^s} \\
		&\lesssim \xi(\underline{\tau}) \left( \gamma^{\alpha m} \varepsilon^{-\frac{2(3-s)}{s}} + \varepsilon \right) \|f\|_{H^s} \|g\|_{H^s}.
	\end{align*}
	Take $\varepsilon=\left( \frac{6-2s}{s} \gamma^{\alpha m} \right)^{\frac{s}{6-s}}$, we have
	\begin{equation*}
		\left| \int_{\mathbb{T}^3} f(x) g \left(\Phi_{\underline{\tau}}^m(x) \right) dx\right| \leq \xi(\underline{\tau})  \gamma^{\alpha' m} \|f\|_{H^s} \|g\|_{H^s},
	\end{equation*}
	where $\alpha'=\frac{s}{6-s}\alpha$ and $\xi(\underline{\tau})$ absorbs an $s$-dependent constant. This completes the proof of Theorem \ref{Exponential mixing}.
\end{proof}

\subsection{Exponential mixing: Proof of Corollary \ref{exponential mixing}}
\quad In this subsection, we demonstrate that the velocity field $u_{\underline{\tau}}$ is almost-sure exponential mixing for passive scalars.
    \begin{proof}[\textbf{Proof of Corollary \ref{exponential mixing}}]
    	By incompressibility and $f(t)=f_0 \circ (\psi_{\underline{\tau}}^t)^{-1}$, it suffices to extend Theorem \ref{Exponential mixing} to the continuous-time case. Recall that the flow $\psi_{\underline{\tau}}^t$, generated by $\dot{x} = u_{\underline{\tau}}(t,x)$, satisfies $\psi_{\underline{\tau}}^{3n}(x) = \Phi_{\underline{\tau}}^n(x)$ at discrete times $t = 3n$ where $n \in \mathbb{N}$. For time $t \in (3n,3(n+1))$, let $t=3n+\tilde{t}$ with $\tilde{t} \in (0,3)$. Applying Theorem \ref{Exponential mixing}, we have 
    	\begin{equation}\label{continuous estimate}
    		\left| \int_{\mathbb{T}^3} f(x) g ( \psi_{\underline{\tau}}^t (x)) \right| \leq \xi(\underline{\tau})e^{-\alpha n}\|f\|_{H^s} \|g \circ \psi_{\theta_{3n} \underline{\tau}}^{\tilde{t}} \|_{H^s}.
    	\end{equation}
    	Combining Lemma 3.1.1 from \cite{Vasilev2000} with Proposition 1.4 in \cite{BO}, we establish that for any $s>0$ and $f \in H^s(\mathbb{T}^d)$, the norm $\|f\|_{H^s}^2$ is equivalent to the norm 
    	\begin{equation*}
    		\sum_{|k|=\lfloor s \rfloor} \int_{\mathbb{T}^d} \int_{\mathbb{T}^d} \frac{|D^k f(x)-D^k f(y)|^2}{d(x,y)^{2\{s\}+d}} dxdy + \|f\|_{L^2}^2,
    	\end{equation*}
    	where $s=\lfloor s \rfloor+ \{s\}$ with $\{s\} \in (0,1)$. Then, we have
    	\begin{align*}
    		\|g &\circ \psi_{\theta_{3n} \underline{\tau}}^{\tilde{t}} \|_{H^s}^2 \lesssim \|g\circ \psi_{\theta_{3n} \underline{\tau}}^{\tilde{t}}\|_{L^2}^2 + \sum_{|k|=\lfloor s \rfloor} \int_{\mathbb{T}^3} \int_{\mathbb{T}^3} \frac{|D^k g(\psi_{\theta_{3n} \underline{\tau}}^{\tilde{t}}(x))- D^k g( \psi_{\theta_{3n} \underline{\tau}}^{\tilde{t}}(y))|^2}{d(x,y)^{2\{s\}+3}} dxdy \\
    		&= \|g\|_{L^2}^2+ \sum_{|k|=\lfloor s \rfloor} \int \int \frac{|D^k g(\psi_{\theta_{3n} \underline{\tau}}^{\tilde{t}}(x))-D^k g( \psi_{\theta_{3n} \underline{\tau}}^{\tilde{t}}(y))|^2}{d(\psi_{\theta_{3n} \underline{\tau}}^{\tilde{t}}(x),\psi_{\theta_{3n} \underline{\tau}}^{\tilde{t}}(y))^{2\{s\}+3}} \cdot \frac{|\psi_{\theta_{3n} \underline{\tau}}^{\tilde{t}}(x)-\psi_{\theta_{3n} \underline{\tau}}^{\tilde{t}}(y)|^{2\{s\}+3}}{d(x,y)^{2\{s\}+3}}   dxdy \\
    		&\lesssim \|g\|_{L^2}^2+ \|D\psi_{\theta_{3n} \underline{\tau}}^{\tilde{t}}\|_{L^{\infty}}^{2\{s\}+3} \|g\|_{H^s}^2,
    	\end{align*}
    	where we use the fact that $\psi_{\underline{\tau}}^{t}$ is a volume-preserving diffeomorphism on $\mathbb{T}^3$ for every $\underline{\tau}$. We consequently obtain
    	\begin{equation*}
    		\sup_{\tilde{t} \in (0,3)} \|g \circ \psi_{\theta_{3n} \underline{\tau}}^{\tilde{t}} \|_{H^s} \lesssim \|g\|_{H^s} \cdot \sup_{\tilde{t} \in (0,3)} \|D\psi_{\theta_{3n} \underline{\tau}}^{\tilde{t}}\|_{L^{\infty}}^{\{s\}+\frac{3}{2}}.
    	\end{equation*}
    	Since $\psi_{\underline{\tau}}^t$ is generated by $\dot{x}=u_{\underline{\tau}}(t,x)$, we can deduce that 
    	\begin{equation*}
    		\frac{d D_x\psi_{\underline{\tau}}^{t}(x)}{dt}=\nabla u_{\underline{\tau}}(t,\psi_{\underline{\tau}}^{t}(x)) \cdot D_x\psi_{\underline{\tau}}^{t}(x).
    	\end{equation*}
    	By applying Gr\"{o}nwall's inequality, we have
    	\begin{equation*}
    		\|D \psi_{\underline{\tau}}^t \|_{L^{\infty}} \lesssim \sup_{x \in \mathbb{T}^3} \exp{\int_0^t \|\nabla u_{\underline{\tau}}(\gamma,\psi_{\underline{\tau}}^\gamma(x))\|d\gamma} \leq \exp{\int_0^t \|\nabla u_{\underline{\tau}}(\gamma,\psi_{\underline{\tau}}^\gamma)\|_{L^{\infty}}d\gamma}.
    	\end{equation*}
    	Consequently, we establish
    	\begin{equation}\label{75}
    		\sup_{\tilde{t} \in (0,3)} \|g \circ \psi_{\theta_{3n} \underline{\tau}}^{\tilde{t}} \|_{H^s} \lesssim \|g\|_{H^s} \cdot  \exp{ \left( c\int_0^3 \|\nabla u_{\underline{\tau}}(\gamma,\psi_{\theta_{3n} \underline{\tau}}^\gamma)\|_{L^{\infty}}d\gamma \right)} =: \|g\|_{H^s} \cdot \Gamma(\theta_{3n} \underline{\tau}),
    	\end{equation}
    	where $c=\{s\}+\frac{3}{2}$.
    	
    	Denote $M = \sup_{x \in \mathbb{T}^3} \{ \|D X_i(x)\| : i = 1, 2, 3 \} < \infty$. By the definition of $u_{\underline{\tau}}$, for any $\underline{\tau} = (\tau_1, \tau_2, \dots)$, we have
    	\[
    	\mathbb{E} \Gamma^q (\underline{\tau}) \leq \mathbb{E} e^{q c M (\tau_1 + \tau_2 + \tau_3)}.
    	\]
    	Since the $\{\tau_i\}_{i=1}^\infty$ are independent and exponentially distributed with mean $h>0$, the right-hand side is finite if and only if $q < (cMh)^{-1}$. For any $\varepsilon>0$, $0<q <(cMh)^{-1}$ and $n \in \mathbb{N}$, Chebyshev’s inequality yields
    	\begin{equation*}
    		\mathbb{P}(\Gamma(\theta_{3n} \underline{\tau}) >e^{\varepsilon n}) \leq e^{-q \varepsilon n} \mathbb{E}\Gamma^q.
    	\end{equation*}
    	By the Borel–Cantelli lemma, there exists an almost surely finite integer-valued random variable $N=N(\varepsilon,\underline{\tau})$ such that for all $n \in \mathbb{N}$,
    	\begin{equation}\label{76}
    		\Gamma(\theta_{3n} \underline{\tau}) \leq e^{\varepsilon (n+N)}.
    	\end{equation}
    	By combining estimates \eqref{75}, \eqref{76}, and \eqref{continuous estimate}, we establish that for any $\varepsilon>0$ 
    	\begin{equation*}
    		\left| \int_{\mathbb{T}^3} f(x) g ( \psi_{\underline{\tau}}^t (x)) \right| \leq \tilde{\xi}(\underline{\tau})e^{\frac{-\alpha+\varepsilon}{3} t}\|f\|_{H^s} \|g \|_{H^s},
    	\end{equation*}
    	where $\tilde{\xi}(\underline{\tau})=c_0\xi(\underline{\tau})e^{\varepsilon (N-1)+\alpha}$ is almost surely finite, with $c_0>0$ being a constant. Choosing $\varepsilon<\alpha$, this establishes the almost sure exponential mixing of the velocity field $u_{\underline{\tau}}$.
    \end{proof}

\end{document}